\documentclass[11pt]{article} % use larger type; default would be 10pt

\usepackage[utf8]{inputenc} % set input encoding (not needed with XeLaTeX)

\usepackage{geometry} % to change the page dimensions
\usepackage{graphicx} % support the \includegraphics command and options

\usepackage{booktabs} % for much better looking tables
\usepackage{array} % for better arrays (eg matrices) in maths
\usepackage{paralist} % very flexible & customisable lists (eg. enumerate/itemize, etc.)
\usepackage{verbatim} % adds environment for commenting out blocks of text & for better verbatim
\usepackage{subfig} % make it possible to include more than one captioned figure/table in a single float
\usepackage{amsmath}
\usepackage{amssymb}
\usepackage{amsthm}
\usepackage{listings}
\usepackage{color}
\usepackage{mathtools}
\usepackage{amsfonts}
\usepackage{caption}
\usepackage{amscd}
\usepackage{url}
\usepackage{ascmac}
\usepackage{enumerate}
\usepackage{fancybox}
\usepackage{amsfonts}
\usepackage{multicol}
\usepackage[normalem]{ulem}
\usepackage[all]{xy}
\usepackage{listings}
\usepackage{mathtools}
\usepackage{caption}
\usepackage{array}
\usepackage{multirow}
\usepackage{lscape}

\PassOptionsToPackage{linktopage}{hyperref}

\mathtoolsset{showonlyrefs}
\theoremstyle{theorem}
\newtheorem{theorem}{Theorem}[section]
\newtheorem{proposition}{Proposition}[section]
\newtheorem{lemma}{Lemma}[section]
\newtheorem{corollary}{Corollary}[section]

\theoremstyle{definition}
\newtheorem{definition}{Definition}[section]
\newtheorem{example}{Example}[section]
\newtheorem{remark}{Remark}[section]

\usepackage{fancyhdr} % This should be set AFTER setting up the page geometry
\usepackage{sectsty}
\allsectionsfont{\sffamily\mdseries\upshape} % (See the fntguide.pdf for font help)
\usepackage[nottoc,notlof,notlot]{tocbibind} % Put the bibliography in the ToC
\usepackage[titles,subfigure]{tocloft} % Alter the style of the Table of Contents

\newcommand{\codim}{\operatorname{codim}}

\newcommand{\V}[1]{\vspace{#1}}

\newcommand{\id}{\operatorname{id}}

\newcommand{\R}{\mathbb{R}}
\newcommand{\Z}{\mathbb{Z}}

\newcommand{\Pa}{\partial}

\newcommand{\Ker}{\operatorname{Ker}}

\newcommand{\Coker}{\operatorname{Coker}}
\renewcommand{\Im}{\operatorname{Im}}

\newcommand{\rank}{\operatorname{rank}}
\newcommand{\corank}{\operatorname{corank}}
\newcommand{\Hom}{\operatorname{Hom}}

\newcommand{\GL}{\operatorname{GL}}

\title{
Characterization of generic parameter families of constraint mappings in optimization}
\author{Naoki Hamada, Kenta Hayano and Hiroshi Teramoto}
\date{} % Activate to display a given date or no date (if empty),
\makeatletter
\@addtoreset{equation}{section}
\makeatother

\begin{document}

\maketitle
\begin{abstract}

The purpose of this paper is to understand generic behavior of constraint functions in optimization problems relying on singularity theory of smooth mappings. 
To this end, we will focus on the subgroup $\mathcal{K}[G]$ of the Mather's group $\mathcal{K}$, whose action to constraint map-germs preserves the corresponding feasible set-germs (i.e.~the set consisting of points satisfying the constraints). 
We will classify map-germs with small stratum $\mathcal{K}[G]_e$-codimensions, and calculate the codimensions of the $\mathcal{K}[G]$-orbits of jets represented by germs in the classification lists and those of the complements of these orbits.
Applying these results and a variant of the transversality theorem, we will show that families of constraint mappings whose germ at any point in the corresponding feasible set is $\mathcal{K}[G]$-equivalent to one of the normal forms in the classification list compose a residual set in the entire space of constraint mappings with at most $4$-parameters.
These results enable us to quantify genericity of given constraint mappings, and thus evaluate to what extent known test suites are generic. 

\end{abstract}

\tableofcontents

\section{Introduction}

Constrained optimization is a problem of minimizing \textit{objective function(s)} within the \textit{feasible set} that is described by the system of equalities and inequalities of \textit{constraint functions}.
This problem appears in a wide range of academic and industrial tasks, including planning, scheduling, design, development, and operation \cite{Arora_Book}.
Although there is an elegant and powerful theory for restricted cases (e.g.~for linear objective/constraint functions \cite{dantzig1997linear,dantzig1997linear2} and convex ones \cite{Boyd_Book,Nesterov_Book}), it is in general difficult to establish a theory for solving such problems.
Solvers for general optimization problems, such as Bayesian optimization \cite{7352306,10.1145/3582078} and evolutionary computation \cite{10.5555/2810085,10.1007/s10462-021-10042-y}, are developed primarily through empirical performance evaluation using a set of artificially designed optimization problems, which is called a \textit{test suite}.
A good test suite should represent typical classes of real-world problems and will facilitate the development of good solvers.

Since errors are inherent in observations and modeling processes of real-world problems, we would like to focus on properties that any problem possesses after a small perturbation and that are preserved by perturbations, that is, generic properties.
Therefore, in order to develop a well-designed test suite, it is necessary to quantify and estimate genericity of a given constrained optimization problem.
A well-known constrained test suite is the C-DTLZ \cite{Jain2014AnEM}, which adds artificially designed constraint functions to the DTLZ \cite{Zitzler2000ComparisonOM}, the de facto standard in unconstrained optimization.
Contrary to their fame, DTLZ and C-DTLZ have been criticized for dealing with exceptional functions that rarely occur in practice \cite{8790342}.
It would be nice to examine whether or not the functions given in existing test suites are generic.
(We will indeed show in Example~\ref{exam:C1-DTLZ1} that the constraint map-germ at a solution of C1-DTLZ1 is far from generic, that is, one cannot expect that it appears in real-world problems.)

The main purpose of this paper is to explore generic properties of smooth inequality/equality constraints on manifolds.
Our main result, Theorem~\ref{thm:main theorem detail}, establishes that the set of $b\leq 4$-parameter families satisfying specific conditions is residual in the space of smooth parameter families of constraint mappings. 
This theorem implies not only that a generic $b\leq 4$-parameter family of constraint mappings satisfy the conditions in the theorem, but also that one can make any (not necessarily generic) parameter family of constraint mappings satisfy the conditions by a small perturbation.
In particular, the main theorem provides a comprehensive understanding of generic behavior of constraint mappings, highlighting typical properties that can be expected in a generic (and thus real-world) setting. 
Note that the main theorem is described in terms of the key concept, \textit{reduction} of constraint map-germs, that is, the operation eliminating unnecessary inequality constraints and restricting to a submanifold satisfying equality constraints. (See Section~\ref{sec:reduction} for detail.)

In order to obtain the main theorem, we will focus on the subgroup $\mathcal{K} [G] \subset
\mathcal{K}$ defined in Section~\ref{sec:preliminaries}. 
This subgroup was originally introduced by Tougeron \cite{Tougeron1972} for a linear Lie group $G$. 
Its basic properties were investigated by Gervais \cite{Gervais1977,Gervais1982,Gervais1984} and further studied by
Izumiya et al. \cite{Izumiya_unpublished}, who provided many interesting examples.
For a suitable Lie group $G$ (given in Section~\ref{sec:preliminaries}), the group $\mathcal{K}[G]$ acts on constraint map-germs in a sensible way; the action of $\mathcal{K}[G]$ indeed preserves the corresponding feasible set-germs, and thus it is suitable for our purpose (i.e.~examining behavior of generic constraint mappings). 
In Theorem~\ref{thm:classification jets}, we will classify map-germs with small $\mathcal{K}[G]_e$-codimensions, and calculate the codimensions of the $\mathcal{K}[G]$-orbits of jets represented by germs in the classification lists and those of the complements of these orbits.
The main result then follows from this theorem together with a variant of the transversality theorem. 
Note that part of the classification in Theorem~\ref{thm:classification jets} has already been given in \cite{Siersma1981,Dimca1984}. 
See Remark~\ref{rem:previous work classification} for detail.

The classification and generic properties obtained in the manuscript is the first step toward creating good test suites with various desired properties and assessing known test suites properly.
By perturbing constraints in our classification lists (Tables~\ref{table:generic constraint q=0}, \ref{table:generic constraint r=0} and \ref{table:generic constraint q>0 r=1}), we can create various constraint mappings, which can be expected to appear in a generic setting.
Since $\mathcal{K}[G]$ is geometric in the sense of Damon~\cite{Damon1984Book}, it is enough to consider a versal unfolding of constraints as a perturbation. 
For understanding which types of constraints appear in a versal unfolding of each constraint (i.e.~obtaining a bifurcation diagram of a versal unfolding), we have to deal with the \textit{recognition problem} for each map-germ in the classification lists (with respect to $\mathcal{K}[G]$-equivalence, cf.~\cite{Gaffney1983}). 
Note that the solutions of the recognition problems are also useful to assess existing test suites. 
In a forthcoming paper, we will solve the recognition problems and give the bifurcation diagrams of versal unfoldings for map-germs in the lists with (stratum) $\mathcal{K}[G]_e$-codimension at most $3$. 

Throughout the manuscript, we will examine only constraint mappings, and not deal with objective functions. 
On the one hand, it is reported \cite{10.1145/3377929.3389974} that real-world problems often have a larger number of constraint functions than objective functions, and that many constraint functions will be active at the same time.
Thus, constraint functions themselves are important objects and have been studied from various perspectives in relevant references \cite{7969433}.
On the other hand, in order to determine generic behavior of objective functions and constraints following the same scheme, we will have to focus on another subgroup of $\mathcal{K}$ instead of $\mathcal{K}[G]$, preserving not only feasible set-germs but also natural ordering for objective functions, called the \textit{Pareto ordering} (cf.~\cite{Miettinen1999}). 
Since such a subgroup is not necessarily geometric in the sense of Damon~\cite{Damon1984Book}, it might be much more difficult to understand the action of this group to map-germs than that of $\mathcal{K}[G]$. 
We will study objective functions (possibly with constraints) in a future project.

This paper is organized as follows: 
after reviewing basic notions (e.g.~$\mathcal{K}[G]$-equivalence and (extended) intrinsic derivatives) in Section~\ref{sec:preliminaries}, we will define a reduction of a constraint map-germ in Section~\ref{sec:reduction}, which can be obtained from the original map-germ by removing inactive inequality constraints and composing an embedding to a submanifold-germ determined from equality constraints. 
In Section~\ref{sec:transversality parameter family}, we will then discuss transversality of (parameter families of) constraint map-germs and their reductions. 
Section~\ref{sec:classification_constraints} is devoted to the classification of jets appearing as (full) reductions of generic $b$($\leq 4$)-parameter families of constraint mappings (Theorem~\ref{thm:classification jets}). 
We will give the main theorem in full detail (Theorem~\ref{thm:main theorem detail}) after the proof of Theorem~\ref{thm:classification jets}.

\section{Preliminaries}\label{sec:preliminaries}

Let $N$ be a manifold,
\begin{math}
	q, r \in \mathbb{N} \cup \left\{ 0 \right\}
\end{math}
and
\begin{math}
	g \colon N \rightarrow \mathbb{R}^q
\end{math}, and 
\begin{math}
	h \colon N \rightarrow \mathbb{R}^r
\end{math}
be
\begin{math}
	C^\infty
\end{math}-mappings. 
The set
\begin{equation}
	M \left( g, h \right) = \left\{ x \in N \middle| \begin{matrix} g_i \left( x \right) \le 0 & \left( i \in \left\{ 1, \ldots, q \right\} \right) \\ h_j \left( x \right) = 0 & \left( j \in \left\{ 1, \ldots, r \right\} \right) \end{matrix} \right\}
\end{equation}
is called the \textit{feasible set} determined from the inequality and equality constraint mappings $g$ and $h$, respectively, where $g(x)=(g_1(x),\ldots, g_q(x))$ and $h(x)=(h_1(x),\ldots, h_r(x))$. 
In this paper, we write 
\begin{math}
	g \left( x \right) \le 0 \Leftrightarrow \forall i \in \left\{ 1, \ldots, q \right\}, g_i \left( x \right) \le 0
\end{math}.

Let $\mathcal{E}_n$ be the set of function-germs on $(\R^n,0)$, whose element is denoted by $f:(\R^n,0)\to \R$ or $f:(\R^n,0)\to (\R,f(0))$. 
The set $\mathcal{E}_n$ is a local ring with addition and multiplication induced from those on $\R$, and the maximal ideal $\mathcal{M}_n=\{f\in \mathcal{E}_n~|~f(0)=0\}$. 
One can regard the product $\mathcal{E}_n^p$($=(\mathcal{E}_n)^p$) as the $\mathcal{E}_n$-module of map-germs from $(\R^n,0)\to \R^p$ in the obvious way. 
We denote by $e_1,\ldots, e_p$ the standard generators of $\R^p$, and these are regarded as constant map-germs, in particular elements of $\mathcal{E}_n^p$.
For map-germs $g\in \mathcal{E}_n^q$ and $h\in \mathcal{M}_n\mathcal{E}_n^r$, we define a subset-germ $M(g,h)$ of $\R^n$ at $0$ in the same way as above. 
Note that $M(g,h)=\emptyset$ if $g_i(0)>0$ for some $i\in \{1,\ldots, q\}$, and $M(g,h)=M(\hat{g},h)$, where $\hat{g}\in \mathcal{E}_n^{q'}$ is obtained from $g$ by removing the components with negative values at $0$.

\subsection{$\mathcal{K}[G]$-equivalence}\label{sec:K[G] equivalence}

Let $G_d\subset \GL(q,\R)$ be the group of diagonal matrices with positive diagonal entries, \begin{math}
G_{gp}
\end{math} be the semidirect product of 
\begin{math}
G_d
\end{math}
and the group of 
\begin{math}
q \times q
\end{math}
permutation matrices 
\begin{math}
P_q
\end{math}, and
\begin{equation}
G = \left\{ \left(
\begin{array}{c|c}
C & B \\
\midrule
O_{r,q} & A 
\end{array}
\right) \middle| C \in G_{gp}, B \in M_{q,r} \left( \mathbb{R} \right), A \in \GL \left( r, \mathbb{R} \right) \right\},
\end{equation}
where 
\begin{math}
M_{q,r} \left( \mathbb{R} \right)
\end{math}
is the set of 
\begin{math}
q \times r
\end{math}
real matrices, and
\begin{math}
O_{r,q}
\end{math}
is the 
\begin{math}
r \times q
\end{math}
zero matrix. 
We define the group $\mathcal{K}[G]$ as follows: 
\[
\mathcal{K}[G] = \left\{(\Phi, \Psi) ~\left|~\begin{minipage}[c]{74mm}
$\Phi:(\R^n,0)\to (\R^n,0)$~:~diffeomorphism-germ 

$\Psi:(\R^n,0)\to G$~:~smooth map-germ
\end{minipage}\right.\right\}.
\]
Note that $\mathcal{K}[G]$ is a subgroup of $\mathcal{K}$, and contains the group $\mathcal{R}$, where $\mathcal{R}$ and $\mathcal{K}$ are groups introduced in \cite{Mather1968}, in particular $\mathcal{K}[G]$ acts on the set $\mathcal{M}_n(\mathcal{E}_n^{q}\times \mathcal{E}_n^r)\cong \mathcal{M}_n\mathcal{E}_n^{q+r}$ as follows:
\begin{align*}
&\left(\Phi(x),\begin{pmatrix}
C(x) & B(x) \\
O & A(x) 
\end{pmatrix}
\right)\cdot (g(x),h(x)) \\
=& \left(C(x)g(\Phi^{-1}(x))+B(x)h(\Phi^{-1}(x)),A(x)h(\Phi^{-1}(x))\right).
\end{align*}
Two map-germs $(g,h),(g',h')\in \mathcal{M}_n\mathcal{E}_n^{q+r}$ are said to be \textit{$\mathcal{K}[G]$-equivalent} if $(g,h)$ is contained in the $\mathcal{K}[G]$-orbit of $(g',h')$ (cf.~\cite{Izumiya_unpublished}). 
It is easy to see that if $(g,h)$ is equal to $(\Phi,\Psi)\cdot (g',h')$ for $(\Phi,\Psi)\in \mathcal{K}[G]$, then $M(g,h)$ is equal to $\Phi(M(g',h'))$. 
For a map-germ 
\begin{math}
\left( g, h \right) \colon \left( \mathbb{R}^n, 0 \right) \rightarrow \left( \mathbb{R}^{q+r}, 0 \right)
\end{math}, the formal tangent space and the extended tangent space of the 
\begin{math}
\mathcal{K} \left[ G \right]
\end{math}-equivalence class are
\begin{align}
T \mathcal{K} \left[ G \right] \left( g, h \right) &= \mathcal{M}_n \left< \cfrac{\partial \left( g, h \right)}{\partial x_1}, \ldots, \cfrac{\partial \left( g, h \right)}{\partial x_n} \right>_{\mathcal{E}_n} + T \mathcal{C} \left[ G \right] \left( g, h \right), \label{eq:ft} \\
T \mathcal{K} \left[ G \right]_e \left( g, h \right) &= \left< \cfrac{\partial \left( g, h \right)}{\partial x_1}, \ldots, \cfrac{\partial \left( g, h \right)}{\partial x_n} \right>_{\mathcal{E}_n} + T \mathcal{C} \left[ G \right] \left( g, h \right), \label{eq:fet} 
\end{align}
where 
\begin{math}
T \mathcal{C} \left[ G \right] \left( g, h \right) = \langle \mathfrak{g} \left( g, h \right) \rangle_{\mathcal{E}_n}
\end{math}
is the tangent space generated by vectors 
\begin{math}
\left( g, h \right)
\end{math}
multiplied by the Lie algebra
\begin{math}
\mathfrak{g}
\end{math}
of 
\begin{math}
G
\end{math}. Specifically, it is 
\begin{equation}
T \mathcal{C} \left[ G \right] \left( g, h \right) = \langle \mathfrak{g} \left( g, h \right) \rangle_{\mathcal{E}_n} = \left(\langle g_1 e_1, \ldots, g_q e_q \rangle_{\mathcal{E}_n}+\langle h_1,\ldots, h_r\rangle_{\mathcal{E}_n}{\mathcal{E}_n^q}\right) \oplus \langle h_1, \ldots, h_r \rangle_{\mathcal{E}_n} \mathcal{E}_n^r.
\end{equation}

\noindent
The \textit{$\mathcal{K}[G]$-codimension} and \textit{$\mathcal{K}[G]_e$-codimension} of $(g,h)\in \mathcal{M}_n\mathcal{E}_n^{q+r}$ are defined as the dimensions (as real vector spaces) of $\mathcal{M}_n\mathcal{E}_n^{q+r}/T\mathcal{K}[G](g,h)$ and $\mathcal{E}_n^{q+r}/T\mathcal{K}[G]_e(g,h)$, respectively. 
In this manuscript, we will also deal with map-germs $(g,h):(N,x)\to (\R^{q+r},(y,0))$ for a manifold $N$, $x\in N$, and $y\in \R^q$ with $y_j\leq 0$. 
Its $\mathcal{K}[G]$- and $\mathcal{K}[G]_e$-codimensions are defined to be those of $(\hat{g},h)\circ \varphi^{-1}$, where $\varphi:U\to \R^n$ is a chart around $x$ and $\hat{g}=(g_{k_1},\ldots, g_{k_s})$ for $\{k_1,\ldots, k_s\}=\{k\in \{1,\ldots,q\}~|~y_k=0\}$.

\begin{example}[C1-DTLZ1 \cite{6595567}]\label{exam:C1-DTLZ1}
C1-DTLZ1 is a benchmark problem for evolutionary many-objective optimization algorithms proposed by H.~Jain and K.~Deb \cite{6595567}. Let 
\begin{math}
k
\end{math}
be a positive integer and 
\begin{math}
M
\end{math}
be an integer greater than $1$. The problem has the following objective function 
\begin{math}
f \colon \mathbb{R}^{M-1+k} \rightarrow \mathbb{R}^{M-1}
\end{math}
along with a function 
\begin{math}
g \colon \mathbb{R}^{M-1+k} \rightarrow \mathbb{R}
\end{math}
involving an inequality constraint. 
For $(y,z)\in \R^{M-1}\times \R^k$, let
\begin{align}
f_1(y,z) &= 0.5 \left( 1+\tilde{f}(z) \right) \prod_{i=1}^{M-1} y_i \\
f_m(y,z) &= 0.5 \left( 1+\tilde{f}(z) \right) \prod_{i=1}^{M-m} y_i \left( 1-y_{M-m+1} \right) \quad \left( 2 \le m \le M \right) \\
g(y,z) &= 1 - \frac{f_M(y,z)}{0.6} - \sum_{i=1}^{M-1} \frac{f_i(y,z)}{0.5}
\end{align}
where
\begin{math}
\tilde{f}(z) = 100 \left\{ k + \sum_{i=1}^k \left( \left( z_i - 0.5 \right)^2 - \cos \left( 20 \pi \left( z_i - 0.5 \right) \right) \right) \right\}
\end{math}. 
By using the functions, C1-DTLZ1 is formulated as the minimization problem of 
\begin{math}
f = \left( f_1, \ldots, f_M \right)
\end{math}
subject to 
\begin{math}
g \ge 0
\end{math}
and 
\begin{math}
y_i, z_j \in \left[ 0, 1 \right]
\end{math}
for 
\begin{math}
i \in \left\{ 1, \ldots, M-1 \right\}
\end{math}
and 
\begin{math}
j \in \left\{ 1, \ldots, k \right\}
\end{math}. 

The function 
\begin{math}
g
\end{math}
can be rewritten as 
\begin{math}
g = 1 - \frac{5}{6}\left( 1 + \tilde{f} \right) \left( 1+\frac{y_1}{5} \right)
\end{math}. The inequality constraint 
\begin{math}
g \ge 0
\end{math}
is active if 
\begin{math}
y_1 = 1
\end{math}
and 
\begin{math}
z_i = 1/2
\end{math}
for all 
\begin{math}
i = \left\{ 1, \ldots, k \right\}
\end{math}. 
For $y'=(y_2,\ldots, y_{M-1})\in (0,1)^{M-2}$, we put $x(y') = (1,y',1/2,\ldots, 1/2)\in \R^{M-1+k}$. 
There are two active inequality constraints $g$ and $y_1$ at $x(y')$. 
One can easily check
\begin{math}
d \left( y_1 \right)_{x(y')} = \left( 1, 0, \ldots, 0 \right)
\end{math}
and 
\begin{math}
d \left( g \right)_{x(y')} = \left(-\frac{1}{6}, 0, \ldots, 0 \right)
\end{math}.
Thus the map-germ 
\begin{math}
\left( g, y_1 \right)
\end{math}
at $x(y')$ has non-isolated singularity, in particular its 
\begin{math}
\mathcal{K}
\end{math}-codimension is infinity. Therefore, its 
\begin{math}
\mathcal{K} \left[ G \right]
\end{math}-codimension is infinity as well. 
\end{example}

\begin{example}[C2-DTLZ2 \cite{6595567}]\label{exam:C2-DTLZ2}
C2-DTLZ2 is also a benchmark problem given in \cite{6595567}.
We take $k$ and $M$ as in Example~\ref{exam:C1-DTLZ1}.
Let 
\begin{math}
r
\end{math}
be a real number. The problem has the following objective function 
\begin{math}
f \colon \mathbb{R}^{M-1+k} \rightarrow \mathbb{R}^{M-1}
\end{math}
along with a function 
\begin{math}
g \colon \mathbb{R}^{M-1+k} \rightarrow \mathbb{R}
\end{math}
involving an inequality constraint. 
For $(y,z)\in \R^{M-1+k}$, let
\begin{align}
f_1(y,z) &= \left( 1+\tilde{f}(z) \right) \prod_{i=1}^{M-1} \cos \left( \frac{y_i \pi}{2} \right) \\
f_m(y,z) &= \left( 1+\tilde{f}(z) \right) \left( \prod_{i=1}^{M-m} \cos \left( \frac{y_i \pi}{2} \right) \right) \sin \left( \frac{y_{M-m+1} \pi}{2} \right) \quad \left( 2 \le m \le M \right) \\
g(y,z) &= \sum_{i=1}^M \left( f_i(y,z) - \lambda(y,z) \right)^2 - r^2
\end{align}
where
\begin{math}
\tilde{f}(z) = \sum_{i=1}^k \left( z_i - 0.5 \right)^2
\end{math}
and 
\begin{math}
\lambda(y,z) = \frac{1}{M} \sum_{i=1}^M f_i(y,z)
\end{math}. 
By using the functions, C2-DTLZ2 is formulated as the minimization problem of 
\begin{math}
f = \left( f_1, \ldots, f_M \right)
\end{math}
subject to 
\begin{math}
g \ge 0
\end{math}
and 
\begin{math}
y_i, z_j \in \left[ 0, 1 \right]
\end{math}
for 
\begin{math}
i \in \left\{ 1, \ldots, m \right\}
\end{math}
and 
\begin{math}
j \in \left\{ 1, \ldots, k \right\}
\end{math}.

By putting 
\begin{math}
\zeta_i = f_i/({1+\tilde{f}})
\end{math}
for 
\begin{math}
i = 1, \ldots, M
\end{math}, 
\begin{math}
\sum_{i=1}^M \zeta_i^2 = 1
\end{math}
holds and thus the variable
\begin{math}
\zeta =(\zeta_1,\ldots, \zeta_M)\in \R^{M}
\end{math}
is constrained on the unit sphere. This makes it possible to reformulate the problem in terms of 
\begin{math}
\zeta
\end{math}
and 
\begin{math}
z
\end{math}
with the additional equality constraint 
\begin{math}
\sum_{i=1}^M \zeta_i^2 = 1
\end{math}
as follows: Minimize 
\begin{math}
f
\end{math}
with respect to 
\begin{math}
z
\end{math}
and 
\begin{math}
\zeta
\end{math}
subject to 
\begin{math}
g = \left( 1+\tilde{f} \right)^2 \sum_{i=1}^M \left( \zeta_i - \frac{1}{M} \sum_{i=1}^M \zeta_i \right)^2 - r^2 \ge 0
\end{math}, 
\begin{math}
\sum_{i=1}^M \zeta_i^2 = 1
\end{math}, and 
\begin{math}
z_i, \zeta_j \in \left[ 0, 1 \right]
\end{math}
for all 
\begin{math}
i \in \left\{ 1, \ldots, k \right\}
\end{math}
and 
\begin{math}
j \in \left\{ 1, \ldots, M \right\}
\end{math}.
Note that the original problem is a reduction of the reformulated problem in the sense introduced in Section~\ref{sec:reduction}. 

Let 
\begin{math}
z_i = 1/2
\end{math}
for all 
\begin{math}
i \in \left\{ 1, \ldots, k \right\}
\end{math}. Let 
\begin{math}
\ell \in \left\{ 1, \ldots, M \right\}
\end{math}
and 
\begin{math}
\zeta_i = \frac{1}{\sqrt{\ell}}
\end{math}
for 
\begin{math}
i \in \left\{ 1, \ldots, \ell \right\}
\end{math}
and 
\begin{math}
\zeta_i = 0
\end{math}
otherwise. Suppose 
\begin{math}
r = \sqrt{1-\ell/M}
\end{math}, then, the set of active inequality constraints are 
\begin{math}
\zeta_{\ell+1}, \ldots, \zeta_M
\end{math}
and 
\begin{math}
g
\end{math}. 
In this case, the map-germ of
\begin{math}
\left( \zeta_{\ell+1}, \ldots, \zeta_M, g,\sum_{i=1}^M\zeta_i^2-1 \right)
\end{math}
at 
\begin{equation}
\left( \zeta_1, \ldots, \zeta_\ell, \zeta_{\ell+1}, \ldots, \zeta_M, z_1, \ldots, z_k \right) = \left( \frac{1}{\sqrt{\ell}}, \ldots, \frac{1}{\sqrt{\ell}}, 0, \ldots, 0, 1/2, \ldots, 1/2 \right)
\end{equation}
has 
\begin{math}
\mathcal{K} \left[ G \right]_e
\end{math}-codimension 
\begin{math}
1
\end{math}. This can be shown as follows: Let 
\begin{math}
\left( g, h \right)
\end{math}
be the map-germ. Then, the standard basis of 
\begin{math}
T \mathcal{K} \left[ G \right]_e \left( g, h \right)
\end{math}
with respect to the monomial ordering in Appendix~\ref{sec:calculation codim determinacy type 1k} consists of 
\begin{multline}
e_1 - \frac{2a}{\ell} e_{M-\ell+2}, \ldots, e_{M-\ell} - \frac{2a}{\ell} e_{M-\ell+2}, e_{M-\ell+1} - \frac{M-\ell+2}{\ell} e_{M-\ell+2}, \\
\zeta_1 e_{M-\ell+2}, \ldots, \zeta_M e_{M-\ell+2}, z_1 e_{M-\ell+2}, \ldots, z_k e_{M-\ell+2}.
\end{multline}
By using Theorem~\ref{thm:standard_basis}, the quotient space 
\begin{math}
\mathcal{E}_{M+k}^{M-\ell+2} / T \mathcal{K} \left[ G \right]_e \left( g, h \right)
\end{math}
is isomorphic to 
\[
\left<e_{M-\ell+2}\right>_\R\subset \R[[\zeta_1,\ldots, \zeta_M,z_1,\ldots, z_k]]^{M-\ell +2}.
\]
Lemma~\ref{lem:K[G]-codim reduction} implies that the 
\begin{math}
\mathcal{K} \left[ G \right]_e
\end{math}-codimension of the original reduced problem is equal to or less than 
\begin{math}
1
\end{math}. Since the rank of the differential of the inequality constraint of the reduced problem is at most that of the problem before reduction, the 
\begin{math}
\mathcal{K} \left[ G \right]_e
\end{math}-codimension of the original reduced problem is equal to 
\begin{math}
1
\end{math}.
In this case, the set of constraints with a parameter 
\begin{math}
r
\end{math}
exhibits a generic behavior. 

\end{example}

For manifolds $N, P$, let $J^m(N,P)$ be the $m$-jet bundle from $N$ to $P$, and $\pi:J^m(N,P)\to N$ be the source mapping. 
We denote the $m$-jet represented by $f:(N,x)\to P$ by $j^mf(x)$. 
We put $J^m(n,p) = \{j^mf(0)\in J^m(\R^n,\R^p)~|~f\in \mathcal{E}_n^p\}$ and $J^m(n,p)_0 = \{j^mf(0)\in J^m(n,p)~|~f(0)=0\}$. 
It is easy to check that the projection from $\mathcal{E}_n^p$ to $J^m(n,p)$ induces the isomorphism between $J^m(n,p)$ (resp.~$J^m(n,p)_0$) and $\mathcal{E}_n^p/\mathcal{M}_n^{m+1}\mathcal{E}_n^p$ (resp.~$\mathcal{M}_n\mathcal{E}_n^p/\mathcal{M}_n^{m+1}\mathcal{E}_n^p$).
Furthermore, we can consider the group of $m$-jets of elements in $\mathcal{K}[G]$, denoted by $\mathcal{K}[G]^m$ and its action on $J^m(n,q+r)_0$. 
Two $m$-jets are said to be \textit{$\mathcal{K}[G]^m$-equivalent} if these are contained in the same $\mathcal{K}[G]^m$-orbit. 
Let $\pi_m:\mathcal{E}_n^{q+r}\to \mathcal{E}_n^{q+r}/\mathcal{M}_n^{m+1}\mathcal{E}_n^{q+r}\cong J^m(n,q+r)$ and \begin{math}
\pi^m_{m'} \colon J^m \left(n,q+r\right) \rightarrow J^{m'} \left(n,q+r\right)
\end{math}
be the natural projections.
The tangent space of the $\mathcal{K}[G]^m$-orbit of an $m$-jet $\sigma = j^m(g,h)(0)\in J^m(n,q+r)$ is equal to $\pi_m(T\mathcal{K}[G](g,h))\subset J^m(n,q+r)_0$, in particular the latter subspace does not depend on the choice of a representative $(g,h)$ of $\sigma$. 
We denote this subspace by $T\mathcal{K}[G]^m(\sigma)$. 
The \textit{$\mathcal{K}[G]^m$-codimension} of an $m$-jet $\sigma\in\mathcal{M}_n\mathcal{E}_n^{q+r}$ is defined as the dimension (as a real vector space) of a quotient space $J^m(n,q+r)_0/T\mathcal{K}[G]^m(\sigma)$. 
A map-germ $(g,h)\in \mathcal{M}_n\mathcal{E}_n^{q+r}$ or its $m$-jet $j^m(g,h)(0)$ is said to be \textit{$m$-determined} relative to $\mathcal{K}[G]$ if any germ $(g',h')\in (\pi_m)^{-1}(j^m(g,h)(0))$ is $\mathcal{K}[G]$-equivalent to $(g,h)$, and $(g,h)$ is \textit{finitely determined} relative to $\mathcal{K}[G]$ if it is $m$-determined for some $m$. 

\begin{proposition}\label{prop:basic properties K[G]-eq/codim}

Let $(g,h)\in \mathcal{M}_n\mathcal{E}_n^{q+r}$.

\begin{enumerate}[(1)]

\item 
(Corollary 4.5 in \cite{Izumiya_unpublished})
The map-germ $(g,h)$ (or its $m$-jet $j^m(g,h)(0)$) is $m$-determined if $\mathcal{M}_n^m\mathcal{E}_n^{q+r}$ is contained in $T\mathcal{K}[G](g,h)$. 
(Note that this condition is equivalent to the condition that $\mathcal{M}_n^m\mathcal{E}_n^{q+r}$ is contained in $T\mathcal{K}[G](g,h)+\mathcal{M}_n^{m+1}\mathcal{E}_n^{q+r}$, which depends only on the $m$-jet $j^m(g,h)(0)$.)

\item 
(\cite{Izumiya_unpublished})
The map-germ $(g,h)$ is finitely determined relative to $\mathcal{K}[G]$ if and only if it has finite $\mathcal{K}[G]$-codimension.

\end{enumerate}

\end{proposition}

The $\mathcal{K}[G]_e$-codimension and the $\mathcal{K}[G]$-codimension are related as follows. 

\begin{proposition}\label{prop:relation K[G]/K[G]_e-cod}

Suppose that $(g,h)\in \mathcal{M}_n\mathcal{E}_{n}^{q+r}$ is not a submersion.
The basis of the kernel of the natural projection
\[
\frac{ \mathcal{E}_n^{q+r}}{t(g,h)(\mathcal{M}_n\mathcal{E}_n^n)+T\mathcal{C}[G](g,h)}\rightarrow \frac{\mathcal{E}_n^{q+r}}{t(g,h)(\mathcal{E}_n^n)+T\mathcal{C}[G](g,h)} = \frac{\mathcal{E}_n^{q+r}}{T\mathcal{K}[G]_e(g,h)}
\]
is $[t(g,h)(e_1)],\ldots, [t(g,h)(e_n)]$. 
In particular, the $\mathcal{K}[G]_e$-codimension of $(g,h)$ is equal to the sum of the $\mathcal{K}[G]$-codimension of $(g,h)$ and $-n+(q+r)$. 

\end{proposition}

\noindent
The proof of this proposition is quite similar to that for \cite[Theorem 2.5]{Mather1969}. 
The detail is left to the reader. 
Using this proposition, one can also obtain a generating set of the quotient space ${\mathcal{E}_n^{q+r}}/{T\mathcal{K}[G]_e(g,h)}$ from that of ${\mathcal{M}_n\mathcal{E}_n^{q+r}}/{T\mathcal{K}[G](g,h)}$.

In order to obtain normal forms of constraint map-germs, we will use the ($\mathcal{K}[G]$ version of) complete transversal theorem \cite{Bruce1997} explained below. 
Let 
\begin{math}
\mathcal{K} \left[ G \right]_l
\end{math}
be the normal subgroup of 
\begin{math}
\mathcal{K} \left[ G \right]
\end{math}
consisting of those germs whose 
\begin{math}
l
\end{math}-jet is equal to that of the identity for 
\begin{math}
l \in \mathbb{N}
\end{math}.
We also define a subgroup $\mathcal{K}[G]^m_l\subset \mathcal{K}[G]^m$ in the same way. 

\begin{theorem}[\cite{Bruce1997}]\label{thm:complete transversal}

Let $m\geq 2$ and
\begin{math}
T
\end{math}
be an 
\begin{math}
\mathbb{R}
\end{math}-vector subspace of 
\begin{math}
\mathcal{M}_n^m \mathcal{E}_n^{q+r}
\end{math}
such that 
\begin{equation}
\mathcal{M}_n^m \mathcal{E}_n^{q+r} \subset T + T \mathcal{K} \left[ G \right]_1 \left( g,h \right) + \mathcal{M}_n^{m+1} \mathcal{E}_n^{q+r}
\end{equation}
holds. 
Then, for any map-germ 
\begin{math}
(g',h')
\end{math}
such that 
\begin{math}
j^{m-1} (g',h') = j^{m-1} (g,h)
\end{math}
holds, there exists 
\begin{math}
t \in T
\end{math}
such that 
\begin{math}
j^m (g',h')
\end{math}
is 
\begin{math}
\mathcal{K} \left[ G \right]_1^m
\end{math}-equivalent to 
\begin{math}
j^m \left( (g,h) + t \right)
\end{math}. 

\end{theorem}

\subsection{Extended intrinsic derivative} \label{sec:extended_intrinsic_derivative}

Let $f:(\R^n,0)\to (\R^q,0)$ be a map-germ with $\rank df_0=q-1$.
In this subsection we will extend the intrinsic derivative of $f$ to a larger subspace, including $\Ker df_0$, which is the source of the original intrinsic derivative. 

We take a vector $\mu_f \in (\Im df_0)^\perp\setminus \{0\}$ and define a subspace $W_f\subset T_0\R^n$ as follows: 
\[
W_f = \bigcap_{\begin{minipage}[c]{13mm}
\scriptsize
$1\leq j\leq q$
$(\mu_f)_j \neq 0$
\end{minipage}} \Ker (df_j)_0,
\]
where $(\mu_f)_j$ and $f_j$ are the $j$-th component of $\mu_f$ and $f$, respectively. 
This subspace is a generalization of that considered in \cite{Guddat_book_1990} ($\tilde{T}$ in the equation (2.5.11) in p.~43).
Since the dimension of $(\Im df_0)^\perp$ is equal to $1$, $W_f$ do not depend on the choice of $\mu_f$. 
We consider the germ $df:(\R^n,0)\to \Hom(T \R^n, f^\ast T\R^q)$ of a section of $\Hom(T \R^n, f^\ast T\R^q)$ and its differential
\[
d(df)_0:T_0 \R^n \to T_{df_0}\Hom (T \R^n, f^\ast T\R^q). 
\]
The tangent space $T_{df_0}\Hom (T \R^n, f^\ast T\R^q)$ can be identified with $\R^n \times \Hom (T_0 \R^n, T_0\R^q)$ by taking the canonical trivializations of the bundles $T \R^n$ and $f^\ast T\R^q$. 
Let $p_2:\R^n \times \Hom (T_0 \R^n, T_0\R^q)\to \Hom (T_0 \R^n, T_0\R^q)$ be the projection and define a linear mapping $\tilde{D}^2 f : W_f\otimes W_f \to \Coker (df_0)$ as follows: 
\[
\tilde{D}^2f(v_1\otimes v_2) = \left[p_2\left(d(df)_0(v_1)\right)(v_2)\right],
\]
where $[v]\in \Coker(df_0)$ for $v\in T_0\R^q$ is a vector represented by $v$.
We call $\tilde{D}^2f$ the \textit{extended intrinsic derivative} of $f$. 
Note that the restriction $\tilde{D}^2f|_{\Ker df_0\otimes \Ker df_0}$ is the usual intrinsic derivative, which we denote by $D^2f$. 

\begin{theorem}\label{T:invariance D^2 under K}

Let $\phi:(\R^n,0)\to (\R^n,0)$ be a diffeomorphism germ, $\psi:(\R^n,0)\to G_{gp}$ be a germ, and $g := \psi \cdot (f \circ \phi)$.
The following diagram commutes: 
\[
\begin{CD}
W_f\otimes W_f @> \tilde{D}^2 f>> \Coker (df_0) \\
@V (d\phi_0)^{-1}\otimes (d\phi_0)^{-1} VV @VV \psi(0)V \\
W_g\otimes W_g @> \tilde{D}^2 g>> \Coker (dg_0).
\end{CD}
\] 
In other words, the extended intrinsic derivative is $\mathcal{K}[G]$-invariant for corank-$1$ inequality constraint map-germs. 
\end{theorem}

\noindent
We need the following lemma to show Theorem~\ref{T:invariance D^2 under K}: 

\begin{lemma}\label{L:invariance intrinsic derivative}

Let $\phi:(\R^n,0)\to (\R^n,0)$ be a diffeomorphism germ, $\psi:(\R^n,0)\to \mathrm{GL}(q,\R)$ be a germ, and $g := \psi \cdot (f \circ \phi)$.
The following diagram commutes: 
\[
\begin{CD}
\Ker(df_0)\otimes \Ker(df_0) @> D^2 f>> \Coker (df_0) \\
@V (d\phi_0)^{-1}\otimes (d\phi_0)^{-1} VV @VV \psi(0)V \\
\Ker(dg_0)\otimes \Ker(dg_0) @> D^2 g>> \Coker (dg_0).
\end{CD}
\]%
In other words, the intrinsic derivative is $\mathcal{K}$-invariant. 

\end{lemma}

\begin{proof}
In what follows, we represent germs and their representatives by the same symbols. 
We identify $\Hom(T\R^n, f^\ast T \R^q)$ with $\R^n \times \Hom (\R^n,\R^q)$ in the obvious way, and the second component is further identified with the set of $q\times n$ matrices.
The second component of the differential $dg$ can be calculated as follows: 
{\allowdisplaybreaks
\begin{align*}
p_2(dg)& = p_2(d (\psi \cdot (f\circ \phi))) \\
& = \left(\frac{\Pa}{\Pa x_j}\sum_k \psi_{ik}\cdot (f_k\circ \phi)\right)_{i,j} \\
& = \left(\sum_k \frac{\Pa \psi_{ik}}{\Pa x_j}\cdot (f_k\circ \phi)\right)_{i,j} + \left(\sum_k \psi_{ik}\cdot \frac{\Pa (f_k\circ \phi)}{\Pa x_j}\right)_{i,j} \\
& = \left(\sum_k \frac{\Pa \psi_{ik}}{\Pa x_j}\cdot (f_k\circ \phi)\right)_{i,j} + \psi \cdot df \circ d\phi.
\end{align*}
}%
Here, $(x_{i,j})_{i,j}$ represent a matrix whose $(i,j)$-entry is $x_{i,j}$. 
For $v\in \Ker(dg_0)\subset T_0\R^n$, the value $p_2(d(dg)_0(v))$, which is identified with a $q\times n$ matrix, can be calculated as follows:
{\allowdisplaybreaks
\begin{align*}
&p_2(d(dg)_0(v)) \\
=&\left(\sum_k v\left(\frac{\Pa \psi_{ik}}{\Pa x_j}\right)\cdot (f_k\circ \phi(0)) +\frac{\Pa \psi_{ik}}{\Pa x_j}(0) \cdot v(f_k\circ \phi) \right)_{i,j} \\
&+ v(\psi)\cdot df_0\circ d\phi_0 + \psi(0)\cdot (d\phi_0(v))(df) \circ d\phi_0 + \psi(0)\cdot df_0\circ v(d\phi)\\
=&v(\psi)\cdot df_0\circ d\phi_0 + \psi(0)\cdot (d\phi_0(v))(df) \circ d\phi_0 + \psi(0)\cdot df_0\circ v(d\phi)\\
&(\because f_k\circ \phi(0)=0,\hspace{.3em}v(f_k\circ \phi)=0). 
\end{align*}
}%
Thus, for $v_1,v_2\in \Ker(dg_0)$, the value $D^2g (v_1\otimes v_2)$ can be calculated as follows: 
{\allowdisplaybreaks
\begin{align*}
& D^2g (v_1\otimes v_2)\\
=&[p_2(d(dg)_0(v_1))(v_2)] \\
=&[v_1(\psi)\cdot df_0\circ d\phi_0(v_2) + \psi(0)\cdot (d\phi_0(v_1))(df) \circ d\phi_0(v_2) + \psi(0)\cdot df_0\circ v_1(d\phi)(v_2)]\\
=&[\psi(0)\cdot (d\phi_0(v_1))(df) \circ d\phi_0(v_2)] = \psi(0)\cdot D^2f(d\phi_0(v_1)\otimes d\phi_0(v_2)).
\end{align*}
}%
This completes the proof of the lemma. 
\end{proof}

\begin{proof}[Proof of Theorem~\ref{T:invariance D^2 under K}]
As we calculated, the following equality holds: 
\[
p_2(dg) = \left(\sum_k \frac{\Pa \psi_{ik}}{\Pa x_j}\cdot (f_k\circ \phi)\right)_{i,j} + \psi \cdot df \circ d\phi = \left(\frac{\Pa \psi_{i,\sigma^{-1}(i)}}{\Pa x_j}\cdot (f_{\sigma^{-1}(i)}\circ \phi)\right)_{i,j} + \psi \cdot df \circ d\phi, 
\]
where $\rho:G_{gp}\to P_q$ is the projection and $\sigma = \rho(\psi(0))$.
Since $f\circ \phi (0) = 0$, the differential $dg_0$ is equal to $\psi(0)\cdot df_0\circ d\phi_0$.

Suppose first that the image of $\psi$ is contained in $G_d = \Ker (\rho)$. 
Since $\psi(0)e_j = k e_j$ for some $k>0$ and $(\mu_f)_j=0$ if and only if $e_j\in \Im df_0$, $(\mu_f)_j=0$ if and only if $(\mu_g)_j=0$. 
The following equality thus holds: 
\begin{align*}
W_g &= \bigcap_{\begin{minipage}[c]{13mm}
\scriptsize
$1\leq j\leq q$
$(\mu_g)_j \neq 0$
\end{minipage}} \Ker (dg_j)_0\\
&= \bigcap_{\begin{minipage}[c]{13mm}
\scriptsize
$1\leq j\leq q$
$(\mu_f)_j \neq 0$
\end{minipage}} \Ker \left(\psi_{jj}(0)\cdot (df_j)_0\circ d\phi_0\right)=d\phi_0^{-1}\left(\bigcap_{\begin{minipage}[c]{13mm}
\scriptsize
$1\leq j\leq q$
$(\mu_f)_j \neq 0$
\end{minipage}} \Ker (df_j)_0\right)=d\phi_0^{-1}(W_f). 
\end{align*}
One can assume $(\mu_f)_1=\cdots = (\mu_f)_r=0$ without loss of generality. 
Let $\pi:\R^q\to \R^{q-r}$ be the projection to the latter components. 
The subspace $W_f$ is equal to $\Ker (d(\pi\circ f)_0)$. 
Since $d\pi_0(\Im df_0)$ is equal to $\Im d(\pi\circ f)_0$, $d\pi_0$ induces a well-defined homomorphism $\overline{d\pi_0}:\Coker(df_0)\to \Coker (d(\pi\circ f)_0)$. 
For any $v\in T_0\R^q$ with $[v]\in \Ker \overline{d\pi_0}$, there exists $w\in T_0\R^q$ with $d\pi_0(v)=d(\pi\circ f)_0(w)$, in particular $v-df_0(w)$ is contained in $\Ker d\pi_0$, which is further contained in $\Im df_0$ by the assumption. 
Thus, $v$ is contained in $\Im df_0$ and $\overline{d\pi_0}$ is an isomorphism.
Under the identification by $\overline{d\pi_0}$, $\tilde{D}^2 f$ coincides with the (usual) intrinsic derivative of $\pi\circ f$. 
Since $\pi\circ f$ and $\pi \circ g$ are $\mathcal{K}$-equivalent and, by Lemma~\ref{L:invariance intrinsic derivative}, the intrinsic derivative is invariant under $\mathcal{K}$-equivalence, the diagram in Theorem~\ref{T:invariance D^2 under K} commutes provided that the image of $\psi$ is in $G_d$. 

For a general $\psi$, one can regard it as a composition of a permutation of entries of $\R^q$ and a mapping whose image is contained in $G_d$. 
Since the diagram in Theorem~\ref{T:invariance D^2 under K} commutes if $\phi = \id$ and $\psi$ is a permutation, the diagram for general $\phi$ and $\psi$ also commutes. 
\end{proof}

Note that one can deduce as a corollary of the proof that $\tilde{D}^2 f$ is symmetric.

\section{Reductions of constraint map-germs and their jets}\label{sec:reduction}

In this section, we will introduce a \textit{reduction} procedure, which changes a map-germ (or its jet) keeping the corresponding feasible set-germ fixed up to diffeomorphisms. 
After explaining its definition, we will discuss various properties of it, especially relating with codimensions. 

Let $N$ be an $n$-manifold without boundary, $g=(g_1,\ldots,g_q):(N,\overline{x}) \to (\R^q,\overline{y})$ and $h=(h_1,\ldots, h_r):(N,\overline{x})\to (\R^r,0)$ be map-germs such that $M(g,h)\neq \emptyset$ (i.e.~$\overline{y}_k\leq 0$ for any $k\in \{1,\ldots,q\}$), $(i):= (i_1,\ldots,i_{r-l})$ and $(k):=(k_1,\ldots, k_{q-s})$ be systems of indices such that $\rank d(h_{i_1},\ldots, h_{i_{r-l}})_{\overline{x}}=r-l$ and $g_{k_1}(\overline{x}),\ldots, g_{k_{q-s}}(\overline{x})\neq 0$. 
Take an immersion-germ $\iota_{(i)}:(\R^{n-r+l},0)\to (N,\overline{x})$ so that the set-germ $\iota_{(i)}(\R^{n-r+l})$ is equal $(h_{i_1},\ldots, h_{i_{r-l}})^{-1}(0)$. 
We define a map-germ $(g,h)_{\iota_{(i)},(k)}:(\R^{n-r+l},0)\to (\R^{s+l},0)$, called a \textit{reduction} of $(g,h)$ as follows: 
\[
(g,h)_{\iota_{(i)},(k)}:= (g_{\iota_{(i)},(k)},h_{\iota_{(i)}}) := \left(g_1\circ\iota_{(i)},\overset{\hat{k}}{\ldots},g_q\circ \iota_{(i)},h_1\circ\iota_{(i)},\overset{\hat{i}}{\ldots},h_r\circ\iota_{(i)}\right), 
\]
where $\hat{k}$ and $\hat{i}$ mean that the components with indices $k_1,\ldots, k_{q-s},i_1,\ldots, i_{r-l}$ removed. 
It is easy to see that the feasible set-germ defined by $(g,h)$ is diffeomorphic to that defined by its reduction. 

The rank of the differential $(dg_{\iota_{(i)},(k)})_0$ is at most $\rank d(g_1,\overset{\hat{k}}{\ldots},g_q)_{\overline{x}}$, and possibly less than $\rank d(g_1,\overset{\hat{k}}{\ldots},g_q)_{\overline{x}}$. 
As for $\rank (dh_{\iota_{(i)}})_0$, the following holds: 

\begin{lemma}\label{H:lem:diff of h_{(i)} is 0}

The rank of the differential $(dh_{\iota_{(i)}})_0$ is equal to $\rank dh_{\overline{x}}-r+l$. 

\end{lemma}

\begin{proof}
We first observe that $(dh_{\iota_{(i)}})_0$ is the composition of the injection $(d\iota_{(i)})_0$ and the restriction $dh_{\overline{x}}|_{\Ker d(h_{i_1},\ldots, h_{i_{r-l}})_{\overline{x}}}$. 
Since $\Ker dh_{\overline{x}}$ is contained in $\Ker d(h_{i_1},\ldots, h_{i_{r-l}})_{\overline{x}}$, we obtain 
\begin{align*}
\rank (dh_{\iota_{(i)}})_0 &= (n-r+l) - \dim \Ker (dh_{\overline{x}}|_{\Ker d(h_{i_1},\ldots, h_{i_{r-l}})_{\overline{x}}})\\
&=(n-r+l)-\dim \Ker dh_{\overline{x}}\\
&=(n-r+l)-(n-\rank dh_{\overline{x}})=\rank dh_{\overline{x}}-r+l. 
\end{align*}
This completes the proof of the lemma.
\end{proof}

\noindent
We call a reduction $(g,h)_{\iota_{(i)},(k)}$ with 
$g_1 \circ \iota_{\left( i \right)} \left( 0 \right) = 0, \overset{\hat{k}}{\ldots}, g_q \circ \iota_{\left( i \right)} \left( 0 \right) = 0$ and
$(dh_{\iota_{(i)}})_0=0$ a \textit{full reduction} of $(g,h)$. 
By Lemma~\ref{H:lem:diff of h_{(i)} is 0}, a full reduction of $(g,h)$ necessarily has $r-\rank(dh_{\overline{x}})$ equality constraints. 
Note that if $g_k({\overline{x}})< 0$ for any $k\in \{1,\ldots,q\}$ (that is, $(k)=(1,\ldots, q)$) and $dh_{\overline{x}}$ is surjective, a full reduction of $(g,h):(N,{\overline{x}})\to (\R^{q+r},(\overline{y},0))$ is the constant germ $c:(\R^{n-r},0)\to \{0\}=\R^0$.  

\begin{lemma}\label{H:lem:basic properties germ (g,h)_{(i)}}

The $\mathcal{R}$-equivalence class of $(g,h)_{\iota_{(i)},(k)}$ is determined from the $\mathcal{R}$-equivalence class of $(g,h)$ and the choice of indices $(i),(k)$. 
In particular it does not depend on the choice of an immersion-germ $\iota_{(i)}$. 

\end{lemma}

\begin{proof}
We first show that the $\mathcal{R}$-equivalence class of $(g,h)_{\iota_{(i)},(k)}$ does not depend on the choice of $\iota_{(i)}$. 
Let $\eta_{(i)}$ be another immersion-germ from $(\R^{n-r+l},0)$ to $(N,{\overline{x}})$ such that $\eta_{(i)}(\R^{n-r+l})$ is equal to $(h_{i_1},\ldots, h_{i_{r-l}})^{-1}(0)$. 
One can define a diffeomorphism-germ 
\[
\Phi:(\R^{n-r+l},0)\xrightarrow{\iota_{(i)}} ((h_{i_1},\ldots, h_{i_{r-l}})^{-1}(0),{\overline{x}})\xrightarrow{\eta_{(i)}^{-1}} (\R^{n-r+l},0),
\]
and it is easy to check that $(g,h)_{\iota_{(i)},(k)}$ is equal to $(g,h)_{\eta_{(i)},(k)}\circ \Phi$. 

Let $\phi:(N,{\overline{x}})\to (N,{\overline{x}})$ be a diffeomorphism-germ. 
It is easy to see that the rank of $d((h_{i_1},\ldots, h_{i_{r-l}})\circ \phi)_{\overline{x}}$ is also equal to $r-l$, and $\phi^{-1}\circ \iota_{(i)}$ is an immersion-germ to $((h_{i_1},\ldots, h_{i_{r-l}})\circ \phi)^{-1}(0)$. 
The following equalities then hold: 
{\allowdisplaybreaks
\begin{align*}
&(g\circ \phi,h\circ \phi)_{\phi^{-1}\circ \iota_{(i)},(k)}\\
=&\left((g_1\circ\phi)\circ (\phi^{-1}\circ \iota_{(i)}),\overset{\hat{k}}{\ldots},(g_q\circ\phi)\circ (\phi^{-1}\circ \iota_{(i)}),\right.\\
&\left.\hspace{1em}(h_1\circ\phi)\circ (\phi^{-1}\circ \iota_{(i)}),\overset{\hat{i}}{\ldots},(h_r\circ\phi)\circ (\phi^{-1}\circ \iota_{(i)})\right)\\
=&(g,h)_{\iota_{(i)},(k)}. 
\end{align*}
}%
This completes the proof of Lemma~\ref{H:lem:basic properties germ (g,h)_{(i)}}.
\end{proof}

\begin{lemma}\label{lem:K[G]-codim reduction}

The $\mathcal{K}[G]$-codimension of $(g,h)_{\iota_{(i)},(k)}$ is less than or equal to that of $(g,h)$, and the same is true for the $\mathcal{K}[G]_e$-codimension. 

\end{lemma}

\begin{proof}
If $(g,h)$ is a submersion, the $\mathcal{K}[G]$-codimensions and $\mathcal{K}[G]_e$-codimensions of $(g,h)$ and its reduction are all equal to $0$, in particular the statement holds. 
In what follows, we assume that $(g,h)$ is not a submersion. 
The $\mathcal{K}[G]_e$-codimension is the sum of the $\mathcal{K}[G]$-codimension and $-n+q+r$ by Proposition~\ref{prop:relation K[G]/K[G]_e-cod} and $-n+q+r$ is invariant under reduction. (Note that $n,q,r$ are respectively the number of variables, active inequality constraints, and active equality constraints.) 
It is thus enough to show the statement for the $\mathcal{K}[G]_e$-codimension. 

We can take a diffeomorphism-germ $\phi:(\R^n,0)\to (N,{\overline{x}})$ so that $h_{i_j}\circ \phi(x)=x_j$ for any $j\in \{1,\ldots, r-l\}$. 
By Lemma~\ref{H:lem:basic properties germ (g,h)_{(i)}}, the $\mathcal{K}[G]_e$-codimension of the reduction of $(g,h)$ is equal to that of $(g\circ \phi,h\circ \phi)$. 
Furthermore, the $\mathcal{K}[G]_e$-codimension is invariant under permutation of components of equality constraints, and those of inequality constraints. 
One can thus put the following assumptions without loss of generality:
\begin{itemize}

\item
$(N,\overline{x})=(\R^n,0)$,

\item 
$(i_1,\ldots, i_{r-l})=(1,\ldots, r-l)$ and $(k_1,\ldots, k_{q-s}) = (s+1,\ldots, q)$,

\item 
$h=(x_1,\ldots, x_{r-l},h_{r-l+1},\ldots, h_r) =: (x_1,\ldots, x_{r-l},\widehat{h}(x))$, 

\item
$\iota_{(i)}(y)=(0,y)\in (\R^{r-l}\times \R^{n-r+l},(0,0))$ for $y\in (\R^{n-r+l},0)$. 

\end{itemize}
The germ $(g(x),h(x)) = (g(x),x_1,\ldots, x_{r-l},\widehat{h}(x))$ is $\mathcal{K}[G]$-equivalent to the following germ:
\[
(g(0,x'),x_1,\ldots, x_{r-l},\widehat{h}(0,x')), 
\]
where $x' =(x_{r-l+1},\ldots, x_n)\in \R^{n-r+l}$. 
Since the reduction of this germ (by $\iota_{(i)}$ given in the assumption above) is equal to that of $(g,h)$, we can further assume that $g$ and $\widehat{h}$ are contained in $\mathcal{E}_{n-r+l}^q$ and $\mathcal{E}_{n-r+l}^l$, respectively (i.e.~the values of these germs do not depend on $x_1,\ldots, x_{r-l}$). 

Let $\psi:\mathcal{E}_n^{q+r}\to \mathcal{E}_{n-r+l}^{s+l}$ be a homomorphism defined by 
{\allowdisplaybreaks
\begin{align*}
\psi(\xi_1,\ldots, \xi_q,\eta_1,\ldots,\eta_r) =& (\xi_1\circ\iota,\ldots, \xi_{s}\circ \iota,\eta_{r-l+1}\circ \iota,\ldots, \eta_{r}\circ \iota)\\
=&(\xi_1(0,y),\ldots, \xi_{s}(0,y),\eta_{r-l+1}(0,y),\ldots, \eta_{r}(0,y)).
\end{align*}
}%
It is easy to see that $\psi$ is surjective. 
Since the $\mathcal{K}[G]_e$-codimension of $(g,h)$ (resp.~$(g,h)_{\iota_{(i)},(k)}$) is the dimension of the quotient space $\mathcal{E}_n^{q+r}/T\mathcal{K}[G]_e(g,h)$ (resp.~$\mathcal{E}_{n-r+l}^{s+l}/T\mathcal{K}[G]_e(g,h)_{\iota_{(i)},(k)}$), it is enough to show that the image $\psi(T\mathcal{K}[G]_e(g,h))$ is contained in $T\mathcal{K}[G]_e(g,h)_{\iota_{(i)},(k)}$. 

By the definition, the tangent space $T\mathcal{K}[G]_e(g,h)$ is equal to 
\begin{align*}
t(g,h)(\mathcal{E}_n^{n}) + h^\ast \mathcal{M}_r \mathcal{E}_n^{q+r} + \left<g_1e_1,\ldots, g_qe_q\right>_{\mathcal{E}_n}.
%=&\left<\left.\sum_{i=1}^{q}\frac{\Pa g_i}{\Pa x_j}e_i +\sum_{i=1}^{r}\frac{\Pa h_i}{\Pa x_j}e_i~\right|~j=1,\ldots, n  \right>_{\mathcal{E}_n}
\end{align*}
The image $t(g,h)(\mathcal{E}_n^{n})$ has the following generating set as an $\mathcal{E}_n$-module: 
\[
\left\{\left.\left(\sum_{i=1}^{q}\frac{\Pa g_i}{\Pa x_j}e_i +\sum_{i=1}^{r}\frac{\Pa h_i}{\Pa x_j}e_{i +q} \right)~\right|~j=1,\ldots, n\right\}.
\]
The image of a generator by $\psi$ is calculated as follows:
\begin{align*}
\psi\left(\left(\sum_{i=1}^{q}\frac{\Pa g_i}{\Pa x_j}e_i +\sum_{i=1}^{r}\frac{\Pa h_i}{\Pa x_j}e_{i+q}\right)\right)=
\sum_{i=1}^{s}\dfrac{\Pa g_i}{\Pa x_j}(0,y)e_i +\sum_{i=r-l+1}^{r}\dfrac{\Pa \widehat{h}_{i-r+l}}{\Pa x_j}(0,y)e_{i+q}. 
\end{align*}
This germ is equal to zero if $j\leq r-l$ since $g$ (resp.~$\widehat{h}$) is contained in $\mathcal{E}_{n-r+l}^q$ (resp.~$\mathcal{E}_{n-r+l}^l$). 
If $j \geq r-l+1$, the germ above is contained in $t(g,h)_{\iota_{(i)},(k)}(\mathcal{E}_{n-r+l}^{n-r+l})$, which is further contained in $T\mathcal{K}[G]_e(g,h)_{\iota_{(i)},(k)}$. 
The set $\{h_ie_j~|~i=1,\ldots, r, j=1,\ldots, q+r\}$ is a generating set of $h^\ast \mathcal{M}_r \mathcal{E}_n^{q+r}$ as an $\mathcal{E}_n$-module. 
The image $\psi(h_ie_j)$ is equal to zero if $i\leq r-l$ or $q+1\leq j\leq q+r-l$. 
Suppose that $i$ is larger than $r-l$. 
The image $\psi(h_ie_j)$ is equal to $\widehat{h}_{i-r+l}(0,y)e_j$ (resp.~$\widehat{h}_{i-r+l}(0,y)e_{j-q+s-r+l}$) if $j\leq q$ (resp.~$j\geq q+r-l$), which is contained in $h_{\iota_{(i)}}^\ast \mathcal{M}_{n-r+l}\mathcal{E}_{n-r+l}^{q+s}\subset T\mathcal{K}[G]_e(g,h)_{\iota_{(i)},(k)}$. 
One can also show that $\psi(\left<g_1e_1,\ldots, g_qe_q\right>_{\mathcal{E}_n})$ is contained in $T\mathcal{K}[G]_e(g,h)_{\iota_{(i)},(k)}$ by direct calculation, completing the proof of Lemma~\ref{lem:K[G]-codim reduction}. 
\end{proof}

For a system $(j)=(j_1,\ldots,j_{n-r+l})$ of indices in $\{1,\ldots, n\}$, together with systems $(i)$ and $(k)$ as above, we define 
\begin{align*}
\Gamma^m({n},{q+r})&=\left\{j^m(g,h)(0)\in J^m(n,q+r)~|~h(0)=0\right\}\mbox{, and} \\
\Gamma^m_{(i),(k)}(n,{q+r})& = \left\{j^m(g,h)(0)\in \Gamma^m(n,{q+r})~\left|~\mbox{\begin{minipage}[c]{49mm}
$g_{k_1}(0),\ldots, g_{k_{q-s}}(0)\neq 0$ \\
$\rank d(h_{i_1},\ldots, h_{i_{r-l}})_0=r-l$
\end{minipage}}\right.\right\}, 
\end{align*}
One can easily check that $\Gamma^m_{(i),(k)}(n,{q+r})$ is a semi-algebraic submanifold of $J^m(n,{q+r})$ with codimension $r$.
For an $\mathcal{R}^m$-invariant subset $\Sigma\subset \Gamma^m({n-r+l},{s+l})$, we define $\tilde{\Sigma}_{(i),(k)}\subset \Gamma^m_{(i),(k)}(n,{q+r})$ and $\tilde{\Sigma}\subset \Gamma^m(n,{q+r})$ as follows:
\begin{align*}
\tilde{\Sigma}_{(i),(k)}&:=\left\{j^m(g,h)(0) \in \Gamma^m_{(i),(k)}(n,q+r)~\left|~j^m(g,h)_{\iota_{(i)},(k)}(0)\in\Sigma\mbox{ for }\exists \iota_{(i)}
\right.\right\},\\
\tilde{\Sigma}&:=\bigcup_{(i),(k)}\tilde{\Sigma}_{(i),(k)}. 
\end{align*}%

\begin{proposition}\label{H:prop:semi-alg tilde{Sigma}}

The sets $\tilde{\Sigma}_{(i),(k)}$ and $\tilde{\Sigma}$ are $\mathcal{R}^m$-invariant.
Moreover, the following hold:

\begin{itemize}

\item 
If $\Sigma$ is a semi-algebraic subset of $\Gamma^m({n-r+l},{s+l})$, $\tilde{\Sigma}_{(i),(k)}$ and $\tilde{\Sigma}$ are semi-algebraic subsets of $\Gamma^m(n,{q+r})$. 

\item 
If $\Sigma$ is a submanifold of $\Gamma^m({n-r+l},{s+l})$, $\tilde{\Sigma}_{(i),(k)}$ is a submanifold of $\Gamma^m_{(i),(k)}(n,{q+r})$. 

\end{itemize}

\noindent
In each case, the following holds.
\begin{align*}
\codim(\tilde{\Sigma},J^m(n,{q+r}))&=\codim(\tilde{\Sigma}_{(i),(k)},J^m(n,{q+r}))\\
&= \codim(\Sigma,\Gamma^m({n-r+l},{s+l}))+r,
\end{align*}
where $\codim(Y,X)=\dim X-\dim Y$ for $Y\subset X$. 

\end{proposition}

\begin{proof}
It is enough to show the statements for $\tilde{\Sigma}_{(i),(k)}$ for a fixed $(i),(k)$. 
Let $j^m(g,h)(0)\in \tilde{\Sigma}_{(i),(k)}$ and $\phi:(\R^n,0)\to (\R^n,0)$ be a diffeomorphism-germ. 
Since $\Sigma$ is $\mathcal{R}^m$-invariant, it follows from (the proof of) Lemma~\ref{H:lem:basic properties germ (g,h)_{(i)}} that $j^m((g,h)\circ \phi)_{\phi^{-1}\circ \iota_{(i)},(k)}(0)$ is contained in $\Sigma$ and thus $j^m((g,h)\circ\phi)(0)\in \tilde{\Sigma}_{(i),(k)}$. 

For a system $(j)=(j_1,\ldots, j_{n-r+l})$ of indices in $\{1,\ldots,n\}$, we put 
\[
\chi_{(i),(j)}(h)=(h_{i_1},\ldots, h_{i_{r-l}},x_{j_1},\ldots, x_{j_{n-r+l}}):(\R^n,0)\to (\R^n,0)
\]
and 
\[
\Gamma^m_{(i),(j),(k)}(n,{q+r}) = \left\{j^m(g,h)(0)\in \Gamma^m_{(i),(k)}(n,{q+r})~\left|~\rank d\left(\chi_{(i),(j)}(h)\right)_0=n\right.\right\}.
\]
Note that it is a Zariski-open subset of $\Gamma^m_{(i),(k)}(n,{q+r})$, and $\Gamma^m_{(i),(k)}(n,{q+r})$ is equal to $\bigcup_{(j)}\Gamma^m_{(i),(j),(k)}(n,{q+r})$. 
We define
$\Phi_{(i),(j),(k)}:\Gamma^m_{(i),(j),(k)}(n,{q+r})\to J^m(n,n)_{0}$ by 
\[
\Phi_{(i),(j),(k)}(j^m(g,h)(0)):=j^m\left(\chi_{(i),(j)}(h)^{-1}\right)(0).
\] 
One can check that $\Phi_{(i),(j),(k)}$ is a Nash mapping. 
Let $\rho\in J^m({n-r+l},n)_{0}$ be the $m$-jet represented by $\rho(y)=(0,y)$. 
We obtain a Nash mapping $\Psi_{(i),(j),(k)}:\Gamma^m_{(i),(j),(k)}(n,{q+r})\to \Gamma^m({n-r+l},{s+l})$ as follows:
\[
\Psi_{(i),(j),(k)}(j^m(g,h)(0)) = j^m\left(g_1,\overset{\hat{k}}{\ldots},g_q,h_1,\overset{\hat{i}}{\ldots},h_r\right)(0)\cdot\Phi_{(i),(j),(k)}(j^m(g,h)(0))\cdot \rho, 
\]
where $\cdot$ in the right-hand side means composition of representatives. 
Since $\Sigma$ is $\mathcal{R}^m$-invariant, the intersection $\tilde{\Sigma}_{(i),(k)}\cap \Gamma^m_{(i),(j),(k)}(n,{q+r})$ is equal to $\Psi_{(i),(j),(k)}^{-1}(\Sigma)$. 
As $\Psi_{(i),(j),(k)}$ is semi-algebraic mapping and $\Gamma^m_{(i),(j),(k)}(n,{q+r})$ is a semi-algebraic subset of $J^m(n,{q+r})$, we can conclude that $\tilde{\Sigma}_{(i),(k)}$ is semi-algebraic subset of $J^m(n,{q+r})$ if $\Sigma\subset J^m({n-r+l},{s+l})$ is semi-algebraic. 

For $\sigma =j^m(g,h)(0)\in \Gamma^m_{(i),(j),(k)}(n,{q+r})$, we define 
\[
\Theta_\sigma:\Gamma^m({n-r+l},{s+l})\to \Gamma^m_{(i),(j),(k)}(n,{q+r})
\]
as follows: 
Let $\pi:(\R^n,0)\to (\R^{n},0)$ and $\overline{\pi}:(\R^n,0)\to (\R^{n-r+l},0)$ be the map-germs defined by $\pi(x)= (0,\ldots, 0,x_{j_1},\ldots,x_{j_{n-r+l}})$ and $\overline{\pi}(x)= (x_{j_1},\ldots,x_{j_{n-r+l}})$, respectively. 
We take indices $\hat{i}_1,\ldots,\hat{i}_{l}\in \{1,\ldots, r\}$ so that $\hat{i}_j<\hat{i}_{j+1}$ and $\{i_1,\ldots, i_{r-l},\linebreak\hat{i}_1,\ldots, \hat{i}_l\}=\{1,\ldots,r\}$. 
We also take indices $\hat{k}_1,\ldots, \hat{k}_s\in \{1,\ldots,q\}$ in the same manner. 
Since the rank of $d\left(\chi_{(i),(j)}(h)\right)_0$ is equal to $n$, we can take map-germs $\tilde{g}_{k}:(\R^n,0)\to \R^q$ and $\tilde{h}_{k}:(\R^n,0)\to \R^r$ ($k=1,\ldots, r-l$) so that $g$ and $h$ are respectively equal to $g\circ \chi_{(i),(j)}(h)^{-1}\circ\pi + \sum_{k=1}^{r-l}h_{i_k}\tilde{g}_{k}$ and $h\circ \chi_{(i),(j)}(h)^{-1}\circ \pi + \sum_{k=1}^{r-l}h_{i_k}\tilde{h}_{k}$. 
For $j^m(g',h')(0)\in \Gamma^m({n-r+l},{s+l})$, we put $\Theta_\sigma(j^m(g',h')(0))=j^m(g'',h'')(0)$, where $g''_{\hat{k}_i}=g'_{i}\circ \overline{\pi}+\sum_{k=1}^{r-l}h_{i_k}(\tilde{g}_{k})_{\hat{k}_i}$ for $i=1,\ldots, s$, $g''_{k_i}=g_{k_i}$ for $i=1,\ldots, q-s$, $h''_{\hat{i}_j}=h'_j\circ \overline{\pi}+ \sum_{k=1}^{r-l}h_{i_k}(\tilde{h}_{k})_{\hat{i}_j}$ for $j=1,\ldots, l$, and $h''_{i_j}=h_{i_j}$ for $j=1,\ldots, r-l$.
It is then easy to see that $\Theta_\sigma$ is smooth, $\Theta_\sigma\circ \Psi_{(i),(j),(k)}(\sigma) = \sigma$, and $\Psi_{(i),(j),(k)}\circ \Theta_\sigma=\id_{\Gamma^m({n-r+l},{s+l})}$. 
In particular, $(d\Psi_{(i),(j),(k)})_\sigma$ is surjective, and thus $\Psi_{(i),(j),(k)}$ is a submersion. 

Since $\Psi_{(i),(j),(k)}^{-1}(\Sigma) = \tilde{\Sigma}_{(i),(k)}\cap \Gamma_{(i),(j),(k)}^m(n,{q+r})$ is an open subset of $\tilde{\Sigma}_{(i),(k)}$, $\tilde{\Sigma}_{(i),(k)}$ is a submanifold if $\Sigma$ is. 
Furthermore, in each case that $\Sigma$ is a semi-algebraic subset or a submanifold, $\codim(\Psi_{(i),(j),(k)}^{-1}(\Sigma),\Gamma_{(i),(j),(k)}^m(n,{q+r}))$ is equal to $\codim (\Sigma,\Gamma^m({n-r+l},{s+l}))$. 
We eventually obtain
{\allowdisplaybreaks
\begin{align*}
&\codim(\tilde{\Sigma}_{(i),(k)},J^m(n,{q+r}))\\
=&\codim(\Psi_{(i),(j),(k)}^{-1}(\Sigma),\Gamma_{(i),(j),(k)}^m(n,{q+r}))+\codim(\Gamma_{(i),(j),(k)}^m(n,{q+r}),J^m(n,{q+r}))\\
=&\codim (\Sigma,\Gamma^m({n-r+l},{s+l}))+r.
\end{align*}
}%
This completes the proof of Proposition~\ref{H:prop:semi-alg tilde{Sigma}}. 
\end{proof}

By Proposition~\ref{H:prop:semi-alg tilde{Sigma}} and observations in the next section, the following value is important when analyzing local behavior of generic parameter families of constraint maps. 

\begin{definition}\label{H:def:extended codimension}

For a submanifold or a semi-algebraic subset $\Sigma\subset \Gamma^m(n,{q+r})$, the value 
\[
\codim(\Sigma,\Gamma^m(n,{q+r}))-n+r=\codim(\Sigma,J^m(n,{q+r}))-n
\]
is called the \textit{extended codimension} of $\Sigma$ and denoted by $d_e(\Sigma)$. 

\end{definition}

\noindent
Indeed, one can deduce from Proposition~\ref{H:prop:semi-alg tilde{Sigma}} that $\codim(\tilde{\Sigma},J^m(n,{q+r}))$ is equal to $d_e(\Sigma)+n$.
Note that for $\sigma\in J^m(n,q+r)_0$ the extended codimension $d_e(\mathcal{K}[G]^m\cdot \sigma)$ and the $\mathcal{K}[G]^m$-codimension of $\sigma$, which is equal to $\codim(\mathcal{K}[G]^m\cdot \sigma,J^m(n,q+r)_0)$, are related as follows:
\begin{equation}\label{eq:relation d_e K[G]-codim}
d_e(\mathcal{K}[G]^m\cdot \sigma)=\codim(\mathcal{K}[G]^m\cdot \sigma,J^m(n,q+r)_0)-n+q+r. 
\end{equation}
Thus, one can deduce from Proposition~\ref{prop:relation K[G]/K[G]_e-cod} that the $\mathcal{K}[G]_e$-codimension of an $m$-determined non-submersion map-germ $(g,h):(\R^n,0)\to (\R^{q+r},0)$ is equal to $d_e(\mathcal{K}[G]^m\cdot j^m(g,h)(0))$.
Note also that $d_e(\Sigma)$ is equal to $d_e((\pi^m_{m'})^{-1}(\Sigma))$ for a semi-algebraic subset $\Sigma\subset \Gamma^{m'}(n,q+r)$ since $\pi^m_{m'}:\Gamma^m(n,q+r)\to \Gamma^{m'}(n,q+r)$ is a submersion.  

\section{Transversality for parameter families of constraint mappings}\label{sec:transversality parameter family}

In this section, we will briefly review a result on transversality for parameter families of mappings, and show that transversality of (parameter families of) constraint mappings imply that of their reductions. 

Let $\pi: J^m(N,\R^{q+r})\to N$ be the source mapping, which is a fiber bundle with fiber $J^m(n,{q+r})$ and structure group $\mathcal{R}^m$. 
For an $\mathcal{R}^m$-invariant submanifold (resp.~semi-algebraic subset) $\Sigma\subset \Gamma^m({n-r+l},{s+l})$, one can define $\mathcal{R}^m$-invariant submanifolds (resp.~fiberwise semi-algebraic subsets) $\tilde{\Sigma}_{(i),(k),N}$ and $\tilde{\Sigma}_N$ in $J^m(N,\R^{q+r})$ so that their intersections with a fiber of the projection $\pi:J^m(N,\R^{q+r})\to N$ are $\tilde{\Sigma}_{(i),(k)}$ and $\tilde{\Sigma}$, respectively. 
Let $b\geq 0$ be an integer and $U\subset \R^b$ be an open set. 
We endow $C^\infty(N\times U,\R^{q+r})$ with the Whitney $C^\infty$-topology. 
We regard $(g,h)\in C^\infty(N\times U,\R^{q+r})$ as a $b$-parameter family of constraint mappings and put $(g_u({\overline{x}}),h_u({\overline{x}}))=(g({\overline{x}},u),h({\overline{x}},u))$ for $u\in U$. 
We define $j_1^m(g,h):N\times U\to J^m(N,\R^{q+r})$ by $j_1^m(g,h)({\overline{x}},u)=j^m(g_u,h_u)({\overline{x}})$. 
By the parametric transversality theorem ([Wall Lemma 2.1]) and Proposition~\ref{H:prop:semi-alg tilde{Sigma}}, we can show the following: 

\begin{theorem}\label{H:thm:transversality constraints}

Let $\Sigma\subset \Gamma^m({n-r+l},{s+l})$ be an $\mathcal{R}^m$-invariant submanifold.  
Then, the set of mappings $(g,h)\in C^\infty(N\times U,\R^{q+r})$ with $j_1^m(g,h)\pitchfork \tilde{\Sigma}_{(i),(k),N}$ is a residual (and thus dense) subset of $C^\infty(N\times U,\R^{q+r})$. 

\end{theorem}

\noindent
Since a semi-algebraic set with codimension $d$ can be decomposed into a finite union of submanifolds with codimension greater than or equal to $d$, we obtain

\begin{corollary}

Let $\Sigma\subset \Gamma^m({n-r+l},{s+l})$ be an $\mathcal{R}^m$-invariant submanifold or semi-algebraic subset with $d_e(\Sigma)>b$. 
The set of mappings $(g,h)\in C^\infty(N\times U,\R^{q+r})$ with $j_1^m(g,h)(N\times U)\cap \tilde{\Sigma}_N=\emptyset$ is a residual (and thus dense) subset of $C^\infty(N\times U,\R^{q+r})$. 

\end{corollary} 

Let $(g,h):(N\times U,({\overline{x}},u))\to (\R^{q+r},(\overline{y},0))$ be a map-germ with $g_{k_i}({\overline{x}},u)\leq 0$ for $i=1,\ldots, q-s$ and $\rank d((h_{i_1})_u,\ldots, (h_{i_{r-l}})_u)_{\overline{x}}=r-l$. 
One can take a diffeomorphism germ $\lambda_{(i)}:(\R^{b},0)\to (U,u)$ and a map-germ $\iota_{(i)}:(\R^{n-r+l}\times \R^b,(0,0))\to (N,{\overline{x}})$ so that $\iota_{(i),v}:(\R^{n-r+l},0)\to N\times \{\lambda_{(i)}(v)\}$ (defined by $\iota_{(i),v}(y)=\iota_{(i)}(y,v)$) is an immersion-germ to $((h_{i_1})_{\lambda_{(i)}(v)},\ldots, (h_{i_{r-l}})_{\lambda_{(i)}(v)})^{-1}(0)$ for $v\in \R^b$ (sufficiently close to $0$). 
We define a map-germ $(g,h)_{\iota_{(i)},\lambda_{(i)},(k)}:(\R^{n-r+l}\times \R^b,(0,0))\to (\R^{s+l},0)$, called a \textit{reduction} of $(g,h)$ as follows:
\[
(g,h)_{\iota_{(i)},\lambda_{(i)},(k)}=\left(g_1\circ(\iota_{(i)},\lambda_{(i)}),\overset{\hat{k}}{\ldots},g_q\circ(\iota_{(i)},\lambda_{(i)}),h_1\circ(\iota_{(i)},\lambda_{(i)}),\overset{\hat{i}}{\ldots},h_r\circ(\iota_{(i)},\lambda_{(i)})\right).
\]
We define a \textit{full reduction} in the same way as the non-parametric case.
We need the following proposition in order to analyze local behavior of generic parameter families of constraint functions.

\begin{proposition}\label{H:prop:equivalence transversality reduction}

Let $(g,h):(N\times U,({\overline{x}},u))\to (\R^{q+r},\overline{y})$ be as above and $\Sigma$ be an $\mathcal{R}^m$-invariant submanifold of $\Gamma^m({n-r+l},{s+l})$.
Let $\overline{\Sigma}$ be the submanifold in $J^m(\R^{n-r+l},\R^{s+l})$ whose intersection with a fiber of the projection $\pi:J^m(\R^{n-r+l},\R^{s+l})\to \R^{n-r+l}$ is $\Sigma$. 
If $j_1^m(g,h)$ is transverse to $\tilde{\Sigma}_{(i),(k),N}$ at $({\overline{x}},u)$, then $j_1^m(g,h)_{\iota_{(i)},\lambda_{(i)},(k)}$ is transverse to $\overline{\Sigma}$ at $(0,0)$. 

\end{proposition}

\begin{proof}
Since the proposition concerns local property around $({\overline{x}},u)\in N\times U$, we can assume that $N=\R^n, U=\R^b, {\overline{x}}=0$, and $u=0$ without loss of generality. 
We take a system $(j)=(j_1,\ldots, j_{n-r+l})$ of indices in $\{1,\ldots, n\}$ so that $j^m(g_0,h_0)(0)$ is contained in $\Gamma^m_{(i),(j),(k)}(n,{q+r})$. 
One can show that (parametric) $\mathcal{R}^m$-equivalence class of $(g,h)_{\iota_{(i)},\lambda_{(i)},(k)}$ does not depend on the choice of $\iota_{(i)}$ and $\lambda_{(i)}$. 
Thus, one can further assume that $\lambda_{(i)}=\id_{\R^b}$ and 
\[
\iota_{(i)}(y,v) = ((h_{i_1})_v,\ldots, (h_{i_{r-l}})_v,x_{j_1},\ldots, x_{j_{n-r+l}})^{-1}(0,y). 
\] 
Put
\begin{align*}
&J^m_{(i),(j)}(\R^n,\R^{q+r})\\
=& \left\{j^m(g,h)(y)\in J^m(\R^n,\R^{q+r})~\left|~\rank d(h_{i_1},\ldots, h_{i_{r-l}},x_{j_1},\ldots, x_{j_{n-r+l}})_y=n\right.\right\}
\end{align*}
and define $\tilde{\Psi}_{(i),(j),(k)}:J^m_{(i),(j)}(\R^n,\R^{q+r}) \to J^m(\R^{n-r+l},\R^{s+l})$ as follows: 
\[
\tilde{\Psi}_{(i),(j),(k)}(j^m(g,h)(x))=j^m\left((g_1,\overset{\hat{k}}{\ldots},g_q,h_1,\overset{\hat{i}}{\ldots},h_r)\circ \rho_{(j),x}\right)(\pi_{(j)}(x)),
\]
where $\rho_{(j),x}:\R^{n-r+l}\to \R^n$ is defined as follows:
\[
\mbox{The $k$-th component of }\rho_{(j),x}(y) = \begin{cases}
y_{k'} & (k=j_{k'} \mbox{ for some }k'\in \{1,\ldots, n-r+l\}),\\
x_k & (\mbox{otherwise}). 
\end{cases}
\]
As in the proof of Proposition~\ref{H:prop:semi-alg tilde{Sigma}}, we can show that $\tilde{\Psi}_{(i),(j),(k)}$ is a submersion. 
Furthermore, it is easy to see that the following equality holds: 
\[
j^m_1(g,h)_{\iota_{(i)},\lambda_{(i)},k} = \tilde{\Psi}_{(i),(j),(k)}\circ j^m_1(g,h)\circ (\iota_{(i)},\lambda_{(i)}). 
\]
By the assumption on transversality, we have the following equality: 
\[
T_{j^m_1(g,h)(0,0)}J^m(\R^n,\R^{q+r}) = T_{j^m_1(g,h)(0,0)}\tilde{\Sigma}_{(i),(k),N} + d(j_1^m(g,h))_{(0,0)}(T_{(0,0)}\R^n\times \R^b). 
\]
We take indices $\hat{i}_1,\ldots, \hat{i}_l\in \{1,\ldots, r\}$ and $\hat{j}_1,\ldots, \hat{j}_{r-l}\in \{1\ldots,n\}$ so that $\{i_1,\ldots, i_{r-l},\linebreak\hat{i}_1,\ldots, \hat{i}_{l}\}=\{1,\ldots,r\}$ and $\{j_1,\ldots, j_{n-r+l},\hat{j}_1,\ldots, \hat{j}_{r-l}\}=\{1,\ldots,n\}$, and put
{\allowdisplaybreaks
\begin{align*}
&W = T_{(0,0)}(\R^n\times \R^b), \\
&W_1 = \left<\left(\frac{\Pa}{\Pa x_{\hat{j}_1}}\right),\ldots, \left(\frac{\Pa}{\Pa x_{\hat{j}_{r-l}}}\right)\right> \subset W, \\
&W_2= \Im d(\iota_{(i)},\lambda_{(i)})_{(0,0)}= \Ker d(h_{i_1},\ldots,h_{i_{r-l}})_0\subset W, \\
&V = T_{j^m_1(g,h)(0,0)}J^m(\R^n,\R^{q+r}) \cong T_0\R^n\oplus T_0 \R^{q}\oplus T_0 \R^{r}\oplus T_{j^m_1(g,h)(0,0)}J^m(\R^n,\R^{q+r})_{(0,0)}, \\
&V_1 = \left<\left(\frac{\Pa}{\Pa X_{i_1}}\right),\ldots,\left(\frac{\Pa}{\Pa X_{i_{r-l}}}\right) \right>\subset T_0\R^r\subset V,\\
&V_2 = T_0\R^n\oplus T_0 \R^{q}\oplus \left<\left(\frac{\Pa}{\Pa X_{\hat{i}_1}}\right),\ldots,\left(\frac{\Pa}{\Pa X_{\hat{i}_{l}}}\right) \right>\oplus T_{j^m_1(g,h)(0,0)}J^m(\R^n,\R^{q+r})_{(0,0)}\subset V.
\end{align*}
}%
It is easy to see that

\begin{itemize}

\item
$V_1\oplus V_2 = T_{j^m_1(g,h)(0,0)}\tilde{\Sigma}_{(i),(k),N} + dj_1^m(g,h)_{(0,0)}(W_1\oplus W_2)$, 

\item 
$T_{j^m_1(g,h)(0,0)}\tilde{\Sigma}_{(i),(k),N}$ and $dj_1^m(g,h)_0(W_2)$ are contained in $V_2$.  

\item
$p_1\circ (dj_1^m(g,h)_{(0,0)})|_{W_1}:W_1\to V_1$ is an isomorphism, where $p_1:V_1\oplus V_2\to V_1 $ is the projection. 

\end{itemize}

\noindent
By these conditions, we can deduce $V_2= T_{j^m_1(g,h)(0,0)}\tilde{\Sigma}_{(i),(k),N} + dj_1^m(g,h)_{(0,0)}(W_2)$. 
Since $V_1$ is contained in $\Ker (d\tilde{\Psi}_{(i),(j),(k)})_{j^m_1(g,h)(0,0)}$ and $\tilde{\Sigma}_{(i),(k),N}\cap J^m_{(i),(j)}(\R^n,\R^{q+r})$ is equal to $\tilde{\Psi}_{(i),(j),(k)}^{-1}(\overline{\Sigma})$, we obtain
\begin{align*}
&T_{j^m_1(g,h)_{\iota_{(i)},\lambda_{(i)},(k)}(0,0)}J^m(\R^{n-r+l},\R^{s+l}) \\
=& T_{j^m_1(g,h)_{\iota_{(i)},\lambda_{(i)},(k)}(0,0)}\overline{\Sigma} + d(\tilde{\Psi}_{(i),(j),(k)}\circ j_1^m(g,h))_{(0,0)}(W_2)\\
=& T_{j^m_1(g,h)_{\iota_{(i)},\lambda_{(i)},(k)}(0,0)}\overline{\Sigma} + (d(j^m_1(g,h)_{\iota_{(i)},\lambda_{(i)},k}))_{(0,0)}(T_{(0,0)}(\R^{n-r+l}\times \R^b)). 
\end{align*}
This completes the proof of the proposition. 
\end{proof}

Suppose that $\Sigma$ is a $\mathcal{K}[G]^m$-orbit of an $m$-determined jet $j^m(g_0,h_0)(0)\in \Gamma^m(n-r+l,s+l)$. 
Since the group $\mathcal{K}[G]$ is geometric in the sense of Damon \cite{Damon1984Book}, transversality of $j^m_1(g,h)_{\iota_{(i)},\lambda_{(i)},(k)}$ to $\overline{\Sigma}$ is equivalent to versality of $(g,h)_{\iota_{(i)},\lambda_{(i)},(k)}$ as an unfolding. 
Thus, by the proposition above, transversality of $j^m_1(g,h)$ to $\tilde{\Sigma}_{(i),(k),N}$ implies that the reduction $(g,h)_{\iota_{(i)},\lambda_{(i)},(k)}$ is a versal unfolding of $(g_0,h_0)$, which is $\mathcal{K}[G]$-equivalent to $(g_0,h_0) + \sum_{i=1}^{b}u_i\xi_i$, where $\xi_1,\ldots, \xi_b\in \mathcal{E}_{n-r+l}^{s+l}$ are representatives of generators of $\mathcal{E}_{n-r+l}^{s+l}/T\mathcal{K}[G]_e(g_0,h_0)$.

\section{Structure of jet spaces relative to $\mathcal{K}[G]$-actions} \label{sec:classification_constraints}

In this section, we will completely classify jets appearing as full reductions of generic $b$-parameter families of constraint mappings with $b\leq 4$ (Theorem~\ref{thm:classification jets}), and then give the main theorem in full detail (Theorem~\ref{thm:main theorem detail}). 
In the context of multi-objective optimization, usually, 
\begin{math}
n
\end{math}
is much larger than 
\begin{math}
q+r
\end{math}. In what follows, we assume that is always the case. 

Let $W^m_{n,q,r}\subset J^m(n,{q+r})_{0}$ be the set of jets $j^m(g,h)(0)$ with $\rank dh_0 = 0$. 
Since a full reduction of any map-germ is contained in $W^m_{n,q,r}$ for some $n,q,r$, it is enough to examine jets in this subset.

\begin{theorem}\label{thm:classification jets}

\begin{enumerate}

\item 
The extended codimension $d_e(W^m_{n,q,r})$ is equal to $q+r +n(r-1)$. 
In particular it is greater than $4$ if either $r\geq 2$ or $r=1\land q\geq 4$. (Note that we assume $n\gg q+r$.)

\item 
Let $A_{1,k}$ and $A_2\subset W^5_{n,0,1}$ be respectively the set of $5$-jets $\mathcal{K}[G]^5$-equivalent to those of type $(1,k)$ and $(2)$ (with any possible signs) in Table~\ref{table:generic constraint q=0}.
The extended codimensions $d_e(A_{1,k})$ and $d_e(A_2)$ are respectively equal to $k-1$ and $4$. 
Furthermore, the extended codimension of 
\[
W^5_{n,0,1}\setminus \left(\bigsqcup_{k=2}^5 A_1^k \sqcup A_2\right)
\]
is equal to $5$. 

\item 
Let $B_0$ be the set of $5$-jets represented by submersions, and $B_{i,k}$ and $B_j$ be respectively the sets of $5$-jets $\mathcal{K}[G]^5$-equivalent to those of type $(i,k)$ and $(j)$ (with any possible signs and parameters) in Table~\ref{table:generic constraint r=0} ($i=1,3,4$, $j=2,5,\ldots ,10$). 
Then $d_e(B_0)$ is equal to $q-n$, and $d_e(B_{i,k})$ and $d_e(B_j)$ are as shown in the far right column of Table~\ref{table:generic constraint r=0}.
Furthermore, the extended codimension of  
\[
J^5(n,q)_0\setminus \left(B_0\sqcup \left(\bigsqcup_{i,k}B_{i,k}\right)\sqcup \left(\bigsqcup_{j}B_{j}\right)\right)
\]
is equal to $5$.

\item 
Let $C_{i,k}$ and $C_j$ be respectively the sets of jets $\mathcal{K}[G]^m$-equivalent to those of type $(i,k)$ and $(j)$ (with any possible signs and parameters) in Table~\ref{table:generic constraint q>0 r=1} ($i=1,3$, $j=2,4,\ldots, 8$, and $m$ depends on the type).
Then $d_e(C_{i,k})$ and $d_e(C_{j})$ are as shown in the far right column of Table~\ref{table:generic constraint q>0 r=1}. 
Furthermore, the extended codimensions of 
\[
W^4_{n,1,1}\setminus \left(\bigsqcup_{i,k}\left(C_{i,k}\right)\sqcup C_2\right), \hspace{.3em}W^3_{n,2,1}\setminus \left(\bigsqcup_{j=4}^7C_j\right),\hspace{.3em} W^3_{n,3,1}\setminus C_8
\]
are equal to $5$.

\end{enumerate}

\noindent
For each map-germ $(g,h)$ in Tables~\ref{table:generic constraint q=0}, \ref{table:generic constraint r=0} and \ref{table:generic constraint q>0 r=1}, one can take representatives of generators of the quotient space $\mathcal{E}_n^{q+r}/T\mathcal{K}[G]_e(g,h)$ as shown in Table~\ref{table:generator quotient}. 

\end{theorem}

\begin{table}[htbp]
 \begin{center}
  \begin{tabular}{|c|l|c|c|c|} \hline
   type & jet & range &$\mathcal{K}$-determinacy & ex.~cod. \\ \hline
   $(1,k)$ & $x_1^k \pm x_2^2 \pm \cdots \pm x_n^2$ & $2\leq k \leq 5$ &$k$& $k-1$ \\
\hline   $(2)$ & $x_1^3 \pm x_1 x_2^2 + x_3^2 \pm \cdots \pm x_n^2$ & &$3$ & $4$ \\
 \hline
  \end{tabular}
  \caption{$5$-jets in $W^5_{n,0,1}$ appearing as a full reduction of a generic $b\leq 4$-parameter family of constraint mappings.}
  \label{table:generic constraint q=0}
 \end{center}
\end{table}

\begin{landscape}
\begin{table}[htbp]
 \begin{center}
{\renewcommand{\arraystretch}{1.2}
  \begin{tabular}{|c|l|c|c|c|c|} \hline
   type& $\tilde{g} \left( x_{l+1}, \ldots, x_n \right)$ & $l$ & range & 
$\mathcal{K}[G]$-determinacy
& 
ex.~cod.
\\ \hline
   $(1,k)$ & $\pm x_q^k + \sum_{j=q+1}^n \pm x_j^2$ & \multirow{2}{*}{$q-1$} & $2\leq k \leq 5$ &$k$ &$k-1$ \\ \cline{1-2}\cline{4-6}
   $(2)$ & $x_q^3 \pm x_q x^2_{q+1} + \sum_{j=q+2}^n \pm x_j^2$ &&  &$3$& $4$ \\ \hline
   $(3,k)$ & $\pm x_{q-1}^k + \sum_{j=q}^n \pm x_j^2$ & &  $2\leq k \leq 4$ &$k$& $k$ \\ \cline{1-2}\cline{4-6}
   $(4,k)$ & $\pm x_q^k \pm x_{q-1} x_q + \sum_{j=q+1}^n \pm x_j^2$ & $q-2$ & $k=3,4$ &$k$& $k$ \\ \cline{1-2}\cline{4-6}
   $(5)$ & $\pm x_{q-1}^2 \pm x_q^3 + \sum_{j=q+1}^n \pm x_j^2$ & &  &$3$& $4$ \\ \hline
   $(6)$ & $\sum_{j=1}^{2} \delta_{j} x_{q-j}^2 + \alpha x_{q-2} x_{q-1} + \sum_{j=q}^n \pm x_j^2$ & & $\alpha\in \R, (\ast)$ &$2$& $3$ \\ \cline{1-2}\cline{4-6}
   $(7)$ & $\pm \left( x_{q-2} \pm x_{q-1} \right)^2 \pm x_{q-1}^3 + \sum_{j=q}^n \pm x_j^2$ &  & &$3$& $4$ \\ \cline{1-2}\cline{4-6}
   $(8)$ & $\pm x_{q-2}^3 \pm x_{q-1}^2 \pm x_{q-2} x_{q-1} + \sum_{j=q}^n \pm x_j^2$ &$q-3$ &  &$3$& $4$ \\ \cline{1-2}\cline{4-6}
   $(9)$ & $\begin{matrix*}[l] &x_q^3 \pm x_{q-2} x_q \pm x_{q-1} x_q \\ &\pm x_{q-2} x_{q-1} + \sum_{j=q+1}^n \pm x_j^2\end{matrix*}$ & && $3$ & $4$ \\ \hline
   $(10)$ & $\begin{matrix*}[l] &\sum_{j=1}^3 \delta_j x_{q-4+j}^2 \\ &+ \sum_{1 \le i < j \le 3} \alpha_{ij} x_{q-4+i} x_{q-4+j} \\ &\pm x_{q-3} x_{q-2} x_{q-1} + \sum_{j=q}^n \pm x_j^2\end{matrix*}$ & $q-4$ & $
\alpha_{ij}\in \R,(\ast\ast)$ &$3$& $4$ \\
\hline
  \end{tabular}
}
  \captionsetup{singlelinecheck=off}
  \caption[.]{Normal forms 
$  	\left( x_1, \ldots, x_{q-1}, \sum_{j=1}^{l_1} x_j - \sum_{j=l_1+1}^l x_j + \tilde{g} \left( x_{l+1}, \ldots, x_n \right) \right)$
  of jets in $W^5_{n,q,0}=J^5(n,q)_0$ appearing as a full reduction of a generic $b\leq 4$-parameter family of constraint mappings, where $0\leq l_1 \leq \lceil \frac{l}{2} \rceil$, $\delta_j=\pm1$, $(\ast)$ (for type $(6)$) is the condition $4 \delta_{1} \delta_{2} - \alpha^2 \neq 0$, and $(\ast\ast)$ (for type $(10)$) is the following condition: 
\begin{center}
$4 \delta_i \delta_j - \alpha_{ij}^2 \neq 0\hspace{.3em}(i,j\in\{1,2,3\},\hspace{.3em}i\neq j),\hspace{.3em} 4 \delta_1 \delta_2 \delta_3 + \alpha_{12} \alpha_{13} \alpha_{23} - \delta_3 \alpha_{12}^2 - \delta_{2} \alpha_{13}^2 - \delta_{3} \alpha_{23}^2 \neq 0.$
\end{center}
Note that the values in the far right column are the extended codimensions of $B_{\mathrm{type}}$, which are not necessarily equal to the $\mathcal{K}[G]_e$-codimensions of the corresponding map-germs (especially for types (6) and (10)). 
}
  \label{table:generic constraint r=0}
 \end{center}
\end{table}
\end{landscape}

\begin{table}[htbp]
 \begin{center}
{\renewcommand{\arraystretch}{1.2}
  \begin{tabular}{|c|l|c|c|c|c|} \hline
   type & $h$ &$q$& range &\begin{minipage}[c]{10mm}
\centering
\V{.2em}
$\mathcal{K}[G]$-

det.
\V{.2em}
\end{minipage}& ex.~cod. \\ \hline
   $(1,k)$ & $x_1^k + \sum_{j=2}^n \pm x_j^2$ & &$2\leq k \leq 4$ &$k$ & $k$ \\ \cline{1-2}\cline{4-6}
   $(2)$ & $x_2^3 \pm x_1^2+ \sum_{j=3}^n \pm x_j^2$ &$1$&  &$3$& $4$ \\ \cline{1-2}\cline{4-6}
   $(3,k)$ & $x_2^k \pm x_1 x_2 + \sum_{j=3}^n \pm x_j^2$ && $3\leq k \leq 4$ &$k$& $k$ \\ \hline
   $(4)$ & $\delta_1 x_1^2 + \delta_2 x_2^2 + \alpha x_1 x_2 + \sum_{j=3}^n \pm x_j^2$ &\multirow{4}{*}{$2$}&$\alpha\in \R, (\ast)$ &$2$& $3$ \\ \cline{1-2}\cline{4-6}
   $(5)$ & $x_1^3 \pm x_2^2 \pm x_1 x_2 + \sum_{j=3}^n \pm x_j^2$ && $\delta_j = \pm 1 $ &$3$& $4$ \\ \cline{1-2}\cline{4-6}
   $(6)$ & $\left( x_1 \pm x_2 \right)^2 \pm x_2^3 + \sum_{j=3}^n \pm x_j^2$ && &$3$& $4$ \\ \cline{1-2}\cline{4-6}
   $(7)$ & $x_3^3 \pm x_1 x_3 \pm x_2 x_3 \pm x_1 x_2 + \sum_{j=4}^n \pm x_j^2$ & & &$3$& $4$ \\ \hline
   $(8)$ & $\begin{matrix}
\sum_{j=1}^3 \delta_j x_j^2 + \sum_{1 \le i < j \le 3} \alpha_{ij} x_i x_j \\ \pm x_1 x_2 x_3 + \sum_{j=4}^n \pm x_j^2\end{matrix}$ &$3$& $\alpha_{ij}\in \R, (\ast\ast)$ &$3$& $4$ \\
\hline
  \end{tabular}
}
\caption{Normal forms $(g_1(x),\ldots, g_q(x),h(x))=\left(x_1,\ldots, x_q,h(x)\right)$ of jets in $W^4_{n,1,1}, W^3_{n,2,1}$ and $W^3_{n,3,1}$ appearing as a full reduction of a generic $b\leq 4$-parameter family of constraint mappings, where $(\ast)$ and $(\ast\ast)$ are the same conditions as those in Table~\ref{table:generic constraint r=0}.
Note that the extended codimensions in the table are not necessarily equal to the $\mathcal{K}[G]_e$-codimensions of the corresponding map-germs (especially for types (4) and (8)). }
  \label{table:generic constraint q>0 r=1}
 \end{center}
\end{table}

\begin{table}[htbp]
\begin{center}
{\renewcommand{\arraystretch}{1.2}
\begin{tabular}{|c|c|c|l|}\hline
$q$ & $r$ & type & generators \\
\hline
\multirow{2}{*}{$0$}&\multirow{2}{*}{$1$}&$(1,k)$&$1,x_1,\ldots, x_1^{k-2}$\\
\cline{3-4}
&&$(2)$&$1,x_1,x_2,x_1^2$\\
\hline
&\multirow{10}{*}{$0$}&$(1,k)$&$e_q,x_qe_q,\ldots, x_q^{k-2}e_q$\\
\cline{3-4}
&&$(2)$&$e_q,x_qe_q,x_{q+1}e_q,x_q^2e_q$\\
\cline{3-4}
&&$(3,k)$&$e_q,x_{q-1}e_q,\ldots,x_{q-1}^{k-1}e_q$\\
\cline{3-4}
&&$(4,k)$&$e_q,x_qe_q,\ldots, x_q^{k-1}e_q$\\
\cline{3-4}
&&$(5)$&$e_q,x_{q-1}e_q,x_qe_q,x_{q-1}x_qe_q$\\
\cline{3-4}
&&$(6)$&$e_q,x_{q-2}e_q,x_{q-1}e_q,x_{q-2}x_{q-1}e_q$\\
\cline{3-4}
&&$(7)$&$e_q,x_{q-2}e_q,x_{q-1}e_q,x_{q-1}^2e_q$\\
\cline{3-4}
&&$(8)$&$e_q,x_{q-2}e_q,x_{q-1}e_q,x_{q-1}^2e_q$\\
\cline{3-4}
&&$(9)$&$e_q,x_{q-2}e_q,x_{q-1}e_q,x_qe_q$\\
\cline{3-4}
&&$(10)$&$e_q,x_{q-3}e_q,x_{q-2}e_q,x_{q-1}e_q,x_{q-3}x_{q-2}e_q,x_{q-3}x_{q-1}e_q,x_{q-2}x_{q-1}e_q$\\
\hline
\multirow{3}{*}{$1$}&\multirow{8}{*}{$1$}&$(1,k)$&$e_2,x_1e_2,\ldots,x_1^{k-1}e_2$\\
\cline{3-4}
&&$(2)$&$e_2,x_1e_2,x_2e_2,x_1 x_2 e_2$\\
\cline{1-1}\cline{3-4}
&&$(3,k)$&$e_2,x_2e_2,\ldots, x_2^{k-1}e_2$\\
\cline{3-4}
\multirow{4}{*}{$2$}&&$(4)$&$e_3,x_1e_3,x_2e_3,x_1x_2e_3$\\
\cline{3-4}
&&$(5)$&$e_3,x_1e_3,x_2e_3,x_1^2e_3$\\
\cline{3-4}
&&$(6)$&$e_3,x_1e_3,x_2e_3,x_2^2e_3$\\
\cline{3-4}
&&$(7)$&$e_3,x_1e_3,x_2e_3,x_3e_3$\\
\cline{1-1}\cline{3-4}
$3$&&$(8)$&$e_4,x_1e_4,x_2e_4,x_3e_4,x_1x_2e_4,x_1x_3e_4,x_2x_3e_4$\\
\hline
\end{tabular}
}
\caption{Representatives of generators of the quotient space $\mathcal{E}_n^{q+r}/T\mathcal{K}[G]_e(g,h)$.
Here, $e_1,\ldots, e_{q+r}\in \mathcal{E}_n^{q+r}$ consist of the standard basis.}
\label{table:generator quotient}
\end{center}
\end{table}

\begin{remark}

The \textit{stratum $\mathcal{K}[G]_e$-codimension} of a germ of type in the tables is defined to be the extended codimension of the corresponding semi-algebraic set $A_{\mathrm{type}},B_{\mathrm{type}}$ or $C_{\mathrm{type}}$. 
If the semi-algebraic set does not contain an uncountable family of $\mathcal{K}[G]^m$-orbits (i.e.~the type is not (6), (10) in Table~\ref{table:generic constraint r=0} or (4), (10) in Table~\ref{table:generic constraint q>0 r=1}), the stratum $\mathcal{K}[G]_e$-codimension coincides with the usual $\mathcal{K}[G]_e$-codimension. 

\end{remark}

\begin{remark}\label{rem:previous work classification}

The results in Table~\ref{table:generic constraint q=0} are not new. We reproduce the table for the sake of completeness. The classification lists of function-germs on boundaries, corners due to Siersma \cite{Siersma1981} are similar to those in Table~\ref{table:generic constraint r=0} and Table~\ref{table:generic constraint q>0 r=1} but he considers an equivalence relation different from 
\begin{math}
\mathcal{K} \left[ G \right]
\end{math}-equivalence.
In the paper of Dimca \cite{Dimca1984}, the classification lists of simple complex analytic function-germs corresponding to the case $q = 1$ are shown. 
Type (1) in Table~\ref{table:generic constraint q>0 r=1} corresponds to $C_2$ for $k=2$ and $B_k$ for $k \ge 3$ in Table~1 in \cite{Dimca1984}. 
Types (2) and (3) in Table~\ref{table:generic constraint q>0 r=1} correspond to $F_4$ and $C_{k+1}$, respectively. 
This implies that 
\begin{math}
\mathcal{K} \left[ G \right]
\end{math}-classes up to stratum 
\begin{math}
\mathcal{K} \left[ G \right]_e
\end{math}-codimension $4$ are simple. However, in the case of 
\begin{math}
q \ge 2
\end{math}, that is no longer the case and moduli families appear in the very beginning of the classification table. 
Types (4) and (7) are such moduli families in the case of $q=2$ and $q=3$, respectively.
\end{remark}

\begin{proof}[Proof of Theorem~\ref{thm:classification jets}]
We can easily calculate the extended codimension of $W^m_{n,q,r}$ as follows:
\[
d_e(W^m_{n,q,r}) = \codim(W^m_{n,q,r},J^m(n,q+r))-n = (q+r+rn)-n=q+r+n(r-1). 
\]
This shows 1 of Theorem~\ref{thm:classification jets}.

\subsection*{Classification of jets in $W^m_{n,0,1}$}

In what follows, we will show 2 of Theorem~\ref{thm:classification jets}. 
The group $\mathcal{K}[G]$ is equal to $\mathcal{K}$ for the case $q=0$, in particular the extended codimension of the $\mathcal{K}[G]^m$-orbit of an $m$-jet with the trivial $1$-jet is equal to its $\mathcal{K}_e^m$-codimension. 
Function-germs with small $\mathcal{R}_e$-codimensions have been classified in \cite{Arnold1972classificationfunctiongerm}. 
Although necessary analysis for proving 2 of Theorem~\ref{thm:classification jets} have already been done implicitly in \cite{Arnold1972classificationfunctiongerm}, we will give the full proof below both for the sake of completeness of this manuscript, and as a preparation for the proof of the other parts of Theorem~\ref{thm:classification jets}. 
(Note that we need not only to classify jets with small $\mathcal{K}^m$-codimensions, but also to show that the extended codimension of the complement of the union of the $\mathcal{K}^m$-orbits of the jets in the classification list is greater than $4$.)

Let $Q_s\subset W_{n,0,1}$ be the set of $2$-jets which is $\mathcal{K}$-equivalent to $\sum_{j=s}^n \pm x_j^2$ (with some signs). 
We can deduce from the Morse lemma that $Q_s$ is a finite union of $\mathcal{K}$-orbits. The extended codimension of $Q_s$ (and thus that of $(\pi^m_2)^{-1}(Q_s)$ for any $m\geq 3$) is equal to $1+s(s-1)/2$ \cite{Gibson_book}. 
In particular, the extended codimension of $\bigsqcup_{s\geq 4}Q_s$ is greater than $4$. 
For this reason, we will only focus on jets in $(\pi^m_2)^{-1}(Q_s)$ for $s\leq 3$ and suitable orders $m$ below.

\subsubsection*{Jets in $Q_1$($=(\pi^2_2)^{-1}(Q_1)$)}

A jet in $Q_1$ is $\mathcal{K}^2$-equivalent to $\sum_{j=1}^{n}\pm x_j^2$. 
It follows from the Morse Lemma that it is $2$-determined relative to $\mathcal{R}$ (and thus $\mathcal{K}$). 
In particular, the preimage $(\pi^5_2)^{-1}(Q_1)$ is equal to $A_{1,2}$. 

\subsubsection*{Jets in $(\pi^5_2)^{-1}(Q_2)$}

A jet in $Q_2$ is $\mathcal{K}^2$-equivalent to the $2$-jet represented by $f_1=\sum_{j=2}^{n}\pm x_j^2$. For any 
\begin{math}
m \geq 3
\end{math}, any
\begin{math}
m
\end{math}-jet 
\begin{math}
f \in \left( \pi_2^m \right)^{-1} \left( f_1 \right)
\end{math}
is 
\begin{math}
\mathcal{R}^m
\end{math}-equivalent to 
\begin{math}
\sum_{i = 3}^m c_i x_1^i + f_1
\end{math} \cite{Gibson_book} for some 
\begin{math}
c_i \; \left( i = 3, \ldots, m \right)
\end{math}. 
Therefore, we can deduce that an $m$-jet $\sigma\in J^m(n,1)_0$ with $\pi^m_{m-1}(\sigma)=j^{m-1}f_1(0)$ is $\mathcal{K}^m$-equivalent to either the $m$-jet of the germ of type $(1,m)$ in Table~\ref{table:generic constraint q=0} or $j^mf_1(0)$ for $m\geq 3$. The germ of type $(1,m)$ in Table~\ref{table:generic constraint q=0} is $m$-determined relative to $\mathcal{K}$ \cite{Siersma1972}, and thus any jet in $(\pi^5_2)^{-1}(Q_2)$ is $\mathcal{K}$-equivalent to either the jet of type $(1,m)$ for $m=3,4,5$ or the jet $j^5f_1(0)$ (with some signs). 
The $\mathcal{K}^5$-codimension of the germ of type $(1,m)$ is equal to $n-2+m$ \cite{Siersma1972}. 
On the other hand, in the same way as that in Appendix~\ref{sec:calculation codim determinacy type 1k}, one can show that \begin{math}
J^5 \left( n, 1 \right)_0 / T \mathcal{K}^5 \left( j^5 f_1 \left( 0 \right) \right)
\end{math}
is isomorphic to 
\begin{math}
\bigl<\overbrace{x_1, \ldots, x_n}^n, \overbrace{x_1^2, x_1^3, x_1^4, x_1^5}^4\bigr>_{\R}\subset \R[[x]]
\end{math}, in particular the $\mathcal{K}^5$-codimension of  $j^5f_1(0)$ is equal to $n+4$.
We can thus deduce from the relation~\eqref{eq:relation d_e K[G]-codim} that the extended codimension of $A_{1,m}$ is equal to $m-1$, and that of the complement $(\pi^5_2)^{-1}(Q_2)\setminus \left(\bigsqcup_{3\leq m \leq 5} A_{1,m}\right)$ is equal to $5$.

\subsubsection*{Jets in $(\pi^3_2)^{-1}(Q_3)$}

A jet in $Q_3$ is $\mathcal{K}^2$-equivalent to the $2$-jet represented by $f_2= \sum_{j=3}^{n}\pm x_j^2$.
By using the result in \cite{Gibson_book}, a $3$-jet $\sigma\in J^3(n,1)_0$ with $(\pi^3_2)(\sigma)=j^2f_2(0)$ is $\mathcal{K}^3$-equivalent to one of the $3$-jets
\begin{math}\label{eq:3jet Q_3}
x_1^3 \pm x_1 x_2^2 + \sum_{j=3}^{n}\pm x_j^2
\end{math} (the jet represented by the germ of type $(2)$), 
\begin{math}
\sigma_1 = x_1^2 x_2 + \sum_{j=3}^{n}\pm x_j^2
\end{math},
\begin{math}
\sigma_2 = x_1^3 + \sum_{j=3}^{n}\pm x_j^2
\end{math}, and 
\begin{math}
j^3f_2(0)=\sum_{j=3}^{n}\pm x_j^2
\end{math}.
The germ of type $(2)$ is $3$-determined and has codimension $n+3$ as shown in \cite{Siersma1972}. 
On the other hand, one can show the following in the same way as that in Appendix~\ref{sec:calculation codim determinacy type 1k}: 

\begin{itemize}

\item
$J^3(n,1)_0/T\mathcal{K}^3(\sigma_1)$ is isomorphic to 
\begin{math}
\bigl<\overbrace{x_1, \ldots, x_n}^n, \overbrace{x_1^2, x_1 x_2, x_2^2, x_2^3}^4\bigr>\subset \R[[x]]
\end{math}, in particular the $\mathcal{K}^3$-codimension of $\sigma_1$ is $n+4$,

\item 
$J^3(n,1)_0/T\mathcal{K}^3(\sigma_2)$ is isomorphic to 
\begin{math}
\bigl<\overbrace{x_1, \ldots, x_n}^n, \overbrace{x_1^2, x_1 x_2, x_2^2, x_1 x_2^2, x_2^3}^5\bigr>_\R\subset \R[[x]]
\end{math}, in particular the $\mathcal{K}^3$-codimension of $\sigma_2$ is $n+5$,

\item 
$J^3(n,1)_0/T\mathcal{K}^3(j^3f_2(0))$ is isomorphic to 
\begin{math}
\bigl<\overbrace{x_1, \ldots, x_n}^n, \overbrace{x_1^2, x_1 x_2, x_2^2, x_1^3, x_1^2 x_2, x_1 x_2^2, x_2^3}^7\bigr>_\R\linebreak\subset \R[[x]]
\end{math}, in particular the $\mathcal{K}^3$-codimension of $j^3f_2(0)$ is $n+7$,

\end{itemize}

\noindent
We can eventually conclude that any jet in $(\pi^3_2)^{-1}(Q_3)$ is $\mathcal{K}^3$-equivalent to that represented by the germ of type $(2)$, the jets $\sigma_1$, $\sigma_2$ or $j^3f_2(0)$. The germ of type $(2)$ is $3$-determined relative to $\mathcal{K}$ \cite{Siersma1972}, and thus, the preimage by $\pi^5_3$ of the union of the $\mathcal{K}^3$-orbits of the germs of type $(2)$ (with all possible signs) is equal to $A_2$. 
Furthermore, the calculations of $\mathcal{K}^3$-codimensions we have done above imply that the extended codimensions of $A_2$ and the complement $(\pi^5_2)^{-1}(Q_3)\setminus A_2$ are equal to $4$ and $5$, respectively. 

In summary, we have shown the following equality:
\begin{align*}
&W^5_{n,0,1}\setminus \left(\bigsqcup_{k=2}^5 A_{1,k}\sqcup A_2\right)\\
 =& (\pi^5_2)^{-1}\left(\bigsqcup_{s\geq 4}Q_s\right)\sqcup \left(\bigsqcup_{\begin{minipage}[c]{11mm}
\scriptsize
\centering
all signs \\
in $f_1$
\end{minipage}}\mathcal{K}^5\cdot j^5f_1(0)\right)\\
&\sqcup  \left((\pi^5_3)^{-1}\left(\left(\bigsqcup_{\begin{minipage}[c]{11mm}
\scriptsize
\centering
all signs \\
in $\sigma_1$
\end{minipage}}\mathcal{K}^3\cdot\sigma_1\right)\sqcup\left(\bigsqcup_{\begin{minipage}[c]{11mm}
\scriptsize
\centering
all signs \\
in $\sigma_2$
\end{minipage}}\mathcal{K}^3\cdot\sigma_2\right)\sqcup\left(\bigsqcup_{\begin{minipage}[c]{11mm}
\scriptsize
\centering
all signs \\
in $f_2$
\end{minipage}}\mathcal{K}^3\cdot j^3f_2(0)\right)\right)\right).
\end{align*}
We have also shown that the extended codimension of this complement is equal to $5$. 
This completes the proof of 2 of Theorem~\ref{thm:classification jets}.

\subsection*{Classification of jets in $W^m_{n,q,0}$ for $q > 0$}

In what follows, we will show 3 of Theorem~\ref{thm:classification jets}. 
Let $\Sigma_k\subset J^1(n,q)_0$ be the set of $1$-jets $j^1f(0)$ with $\corank (df)_0\geq k$, which is an algebraic subset with codimension $k(n-q+k)$ in $J^1(n,q)_0$. (Note that we assume $n\gg q$.) 
It is easy to see that $B_0$ is equal to $(\pi^5_1)^{-1}(\Sigma_0\setminus \Sigma_1)$, and its extended codimension is calculated as follows: 
{\allowdisplaybreaks
\begin{align*}
d_e(B_0) &= \codim((\pi^5_1)^{-1}(\Sigma_0\setminus \Sigma_1),J^5(n,q))-n \\
&=\codim(\Sigma_0\setminus \Sigma_1,J^1(n,q))-n =q-n.
\end{align*}
}%
Furthermore, the extended codimension of $(\pi^5_1)^{-1}(\Sigma_2)$, which is the set of $5$-jets $j^5g(0)$ with $\corank(dg)_0\geq 2$, is equal to $\codim(\Sigma_2,J^1(n,q))-n=2(n-q+2)+q-n$, which is much larger than $4$. 
For this reason, we will consider $5$-jets $j^5g(0)$ with $\corank(dg)_0=1$ below. 

\begin{lemma}\label{lem:classification corank1 1-jet q=0}

For $l\in \{0,\ldots, q-1\}$, let $\Lambda_l$ be the set of $1$-jets $\mathcal{K}[G]^1$-equivalent to 
\begin{equation}
\left( x_1, \ldots, x_{q-1}, \sum_{j=1}^{l_1} x_j - \sum_{j=l_1+1}^l x_j \right) \label{eq:q0_1jet}
\end{equation}
for some $l_1\in \left\{ 0, 1, \ldots, \lceil \frac{l}{2} \rceil \right\}$. 
Then, $\Lambda_0,\ldots, \Lambda_{q-1}$ are mutually distinct submanifolds in $\Sigma_1\setminus \Sigma_2$. 
Furthermore, the following equality holds: 

\begin{itemize}

\item
$\Sigma_1\setminus \Sigma_2 = \bigsqcup_{l=0}^{q-1}\Lambda_l$, 

\item
$\codim (\Lambda_l,\Sigma_1\setminus \Sigma_2) = q-1-l$. 

\end{itemize}

\noindent
In particular, the extended codimension of $\Lambda_l$ is equal to $q-l$ (and thus $d_e((\pi^5_1)^{-1}(\Lambda_l))=q-l$). 

\end{lemma}

\begin{proof}[Proof of Lemma~\ref{lem:classification corank1 1-jet q=0}]
It is easy to see that a $1$-jet in $\Sigma_1\setminus \Sigma_2$ is $\mathcal{K}[G]^1$-equivalent to the following jet for some $l\in \{1,\ldots, q-1\}$ and $\delta_j = \pm 1$: 
\begin{equation}\label{eq:original 1-jet}
\left( x_1, \ldots, x_{q-1}, \sum_{j=1}^l \delta_j x_j \right). 
\end{equation}
If there exists at least one 
\begin{math}
j \in \left\{ 1, \ldots, l \right\}
\end{math}
for which 
\begin{math}
\delta_j = 1
\end{math} (say $\delta_1=1$), this
\begin{math}
1
\end{math}-jet can be transformed to the following form. 
\begin{equation}\label{eq:transformed 1-jet}
\left( x'_1, \ldots, x'_{q-1}, x'_1 - \sum_{j=2}^l \delta_j x'_j \right).
\end{equation}
Indeed, the $1$-jet in Eq.~\eqref{eq:original 1-jet} can be transformed to
\begin{equation}
\left( x'_1 - \sum_{j=2}^l \delta_j x'_j, x'_2, \ldots, x'_{q-1}, x'_1 \right),
\end{equation}
by the coordinate transformation 
\begin{equation}
\left( x_1, \ldots, x_n \right) \mapsto \left( x_1 + \sum_{j=2}^l \delta_j x_j, x_2, \ldots, x_n \right).
\end{equation}
By permuting the
\begin{math}
1
\end{math}-st and the 
\begin{math}
q
\end{math}-th components of the $1$-jet, we get the $1$-jet in Eq.~\eqref{eq:transformed 1-jet}.
In summary, one can flip the signs of 
\begin{math}
\delta_j \; \left( j \in \left\{ 2, \ldots, l \right\} \right)
\end{math}
simultaneously by using the $\mathcal{K}[G]^1$-action provided 
\begin{math}
\delta_1 = 1
\end{math}. 
This shows that $\Sigma_1\setminus \Sigma_2 $ is equal to $ \bigcup_{l=0}^{q-1}\Lambda_l$. 

The subset $\Lambda_l$ is equal to the union of the $\mathcal{K}[G]^1$-orbits of the $1$-jets in Lemma~\ref{lem:classification corank1 1-jet q=0} (with $l_1\in \left\{ 0, 1, \ldots, \lceil \frac{l}{2} \rceil \right\}$), which is a submanifold of $\Sigma_1\setminus \Sigma_2$. The $\mathcal{K}[G]^1$-codimension of the $1$-jet in Lemma~\ref{lem:classification corank1 1-jet q=0} is $n-l$ since
\begin{math}
J^1(n,q)_0/T\mathcal{K}^1(g_l)
\end{math}
is isomorphic to 
\begin{math}
\bigl<\overbrace{x_{l+1} e_q, \ldots, x_n e_q}^{n-l}\bigr>_{\R}\linebreak\subset \R[[x]]^q
\end{math}.

Since this $\mathcal{K}[G]^1$-codimension is equal to $\codim(\Lambda_l,J^1(n,q)_0)$, the codimension of $\Lambda_l$ in $\Sigma_1\setminus \Sigma_2$ is equal to $q-1-l$ and $\Lambda_l\cap \Lambda_{l'}=\emptyset$ for $l\neq l'$. 
The last statement follows from the relation~\eqref{eq:relation d_e K[G]-codim}.
\end{proof}

\noindent
By this lemma, the extended codimension of 
$\bigsqcup_{l\leq q-5} \Lambda_l$ is equal to $5$. 
For this reason, we will only focus on jets in $(\pi^m_1)^{-1}\left(\Lambda_l\right)$ for $l\geq q-4$ below.

\begin{lemma}\label{lem:normal form corank1 2-jet q=0}

Any $2$-jet in $(\pi^2_1)^{-1}\left(\Lambda_l\right)$ is $\mathcal{K}[G]^2$-equivalent to 
\begin{equation}\label{eq:normal form 2-jet q=0}
\left( x_1, \ldots, x_{q-1}, \sum_{j=1}^{l} \pm x_j+\sum_{j_1=l+1}^{q-1}\sum_{j_2=l+1}^{s-1}a_{j_1j_2}x_{j_1}x_{j_2}+ \sum_{j=s}^{n}\pm x_j^2\right)
\end{equation}
for some $s\in \{q,\ldots, n\}$ and $a_{j_1j_2}\in \R$. 

\end{lemma}

\begin{proof}[Proof of Lemma~\ref{lem:normal form corank1 2-jet q=0}]
It is easy to see that a representative of a $2$-jet in $(\pi^2_1)^{-1}(\Lambda_l)$ is $\mathcal{K}[G]$-equivalent to the following germ:
\begin{equation}\label{eq:g and tildeg}
\left( x_1, \ldots, x_{q-1}, \sum_{j=1}^{l} \pm x_j+ \tilde{g} \left(x_1,\ldots, x_n\right) \right),
\end{equation}
where $\tilde{g}\in \mathcal{M}_n^2$.
By using the Taylor theorem (Lemma~3.3 in p.~60 in \cite{Golubitsky_Book_v1}), 
$\tilde{g}$ can be written as follows:
\[
\tilde{g} \left(x_1,\ldots,x_n \right) = \tilde{g}_0 \left( x_{l+1},\ldots, x_n \right) + \sum_{j=1}^{l} x_j \tilde{g}_j \left( x_1,\ldots,x_n \right),
\]
where 
\begin{math}
j^1 \tilde{g}_0 = 0
\end{math}
and
\begin{math}
j^0\tilde{g}_j = 0
\end{math}
for all 
\begin{math}
j \in \left\{ 1, \ldots, l \right\}
\end{math}. 
By plugging this expression, we obtain 
\begin{equation}
\left( x_1, \cdots, x_{q-1}, \sum_{j=1}^{l} \pm x_j \left( 1 \mp \tilde{g}_j \left( x \right) \right) + \tilde{g}_0 \left( x_{l+1},\ldots,x_n \right) \right).
\end{equation}
By changing coordinates 
\begin{equation}
\left( x_1,\ldots, x_n \right) \mapsto \left( x_1 \left( 1 \mp \tilde{g}_1 \left( x \right) \right), \ldots, x_{l} \left( 1 \mp \tilde{g}_{l} \left( x \right) \right), x_{l+1}, \ldots, x_n \right), 
\end{equation}
along with multiplying positive factors 
\begin{math}
1 / \left( 1 \mp \tilde{g}_j \left( x \right) \right)
\end{math}
to the 
\begin{math}
j
\end{math}-th component of the map-germ above, we obtain the following map-germ:
\begin{equation}
\left( x_1, \ldots, x_{q-1}, \sum_{j=1}^{l} \pm x_j + \tilde{g}_0 \left( x_{l+1}, \ldots,x_n \right) \right).\label{eq:tildeg_0}
\end{equation}
The $2$-jet of the germ above is equal to that of the following germ:
\[
\left( x_1, \ldots, x_{q-1}, \sum_{j=1}^{l} \pm x_j +\sum_{j_1=l+1}^{q-1}\sum_{j_2=l+1}^n\alpha_{j_1j_2}x_{j_1}x_{j_2}+ \sum_{j_1,j_2\geq q}\beta_{j_1j_2}x_{j_1}x_{j_2} \right)
\]
where $\alpha_{j_1j_2},\beta_{j_1j_2}\in \R$. 
We can change the last term $\sum_{j_1,j_2\geq q}\beta_{j_1,j_2}x_{j_1}x_{j_2}$ to $\sum_{j=s}^{n}\pm x_j^2$ by changing coordinates $(x_1,\ldots, x_n)$ preserving $(x_1,\ldots, x_{q-1})$.
This coordinate transformation changes the $2$-jet above to 
the following jet:
\[
\left( x_1, \ldots, x_{q-1}, \sum_{j=1}^{l} \pm x_j +\sum_{j_1=l+1}^{q-1}\sum_{j_2=l+1}^n\tilde{\alpha}_{j_1j_2}x_{j_1}x_{j_2} + \sum_{j=s}^n\pm x_{j}^2 \right)
\]
By completing the square, we can further change this jet to that in \eqref{eq:normal form 2-jet q=0} by a coordinate transformation preserving $(x_1,\ldots, x_{q-1})$. 
\end{proof}

Let $\Lambda_{l,s}\subset (\pi^2_1)^{-1}(\Lambda_l)$ be the set of $2$-jets $\mathcal{K}[G]^2$-equivalent to that in \eqref{eq:normal form 2-jet q=0} for some signs and $a_{j_1j_2}\in \R$. 

\begin{lemma}\label{lem:estimate codimension Lambda_l,s}

$\Lambda_{l,s}$ is a submanifold of $J^2(n,q)_0$ and its extended codimension is equal to $q-l+(s-q+1)(s-q)/2$. 

\end{lemma}

\begin{proof}[Proof of Lemma~\ref{lem:estimate codimension Lambda_l,s}]
Let $Q_s\subset (\pi^2_1)^{-1}(\Sigma_1\setminus \Sigma_2)$ be the set of $2$-jets $\mathcal{K}$-equivalent to the $2$-jet represented by  $f_s = \left(x_1,\ldots,x_{q-1},\sum_{j=s}^{n}\pm x_j^2\right)$ (with some signs). 
It is easy to see that $\Lambda_{l,s}$ is contained in $(\pi^2_1)^{-1}(\Lambda_l)\cap Q_s = (\pi^2_1|_{Q_s})^{-1}(\Lambda_l)$, and the proof of Lemma~\ref{lem:normal form corank1 2-jet q=0} implies the opposite inclusion $(\pi^2_1|_{Q_s})^{-1}(\Lambda_l)\subset \Lambda_{l,s}$. 
For any $\sigma\in Q_s$, $Q_s$ is equal to the $\mathcal{K}^2$-orbit of $\sigma$ (especially a submanifold of $J^2(n,q)_0$), and $\Sigma_1\setminus \Sigma_2$ is the $\mathcal{K}^1$-orbit of $\pi^2_1(\sigma)$. 
Moreover, the following diagram commutes:
\[
\begin{CD}
\mathcal{K}^2 @> a_{\sigma}>> Q_s\\
@V\pi^2_1VV @VV\pi^2_1|_{Q_s}V\\
\mathcal{K}^1 @>a_{\pi^2_1(\sigma)}>> \Sigma_1\setminus \Sigma_2,
\end{CD}
\]
where $a_\sigma(\tau) = \tau\cdot \sigma$ for $\tau \in \mathcal{K}^2$ and $a_{\pi^2_1(\sigma)}$ is defined in the same way. 
Since $(d\pi^2_1)_{1_{\mathcal{K}^2}}$ and $(da_{\pi^2_1(\sigma)})_{1_{\mathcal{K}^1}}$ are both surjective, $(d\pi^2_1|_{Q_s})_\sigma$ is also surjective.
In particular, $\Lambda_{l,s}=(\pi^2_1|_{Q_s})^{-1}(\Lambda_l)$ is a submanifold of $Q_s$.

One can easily show that $\{[x_je_q]~|~q\leq j\leq n\}\cup\{[x_{j_1}x_{j_2}e_q]~|~q\leq j_1\leq j_2\leq s-1\}$ is a basis of the quotient space $\mathcal{M}_n\mathcal{E}_n^q/(T\mathcal{K}f_s+\mathcal{M}_n^3\mathcal{E}_n^q)$. 
In particular, the $\mathcal{K}$-codimension of $f_s$ is equal to $n-q+1+(s-q+1)(s-q)/2$, which is further equal to the codimension of $Q_s$ in $J^2(n,q)_0$.
We thus obtain:
\begin{align*}
d_e(\Lambda_{l,s})=& \codim ((\pi^2_1|_{Q_s})^{-1}(\Lambda_l),J^2(n,q))-n \\
=& \codim(Q_s,J^2(n,q)) +\codim (\Lambda_l,\Sigma_1\setminus \Sigma_2)-n \\
=&n-q+1+(s-q+1)(s-q)/2 +q+ (q-1-l) -n\\
=&q-l+(s-q+1)(s-q)/2. 
\end{align*}
This completes the proof of Lemma~\ref{lem:estimate codimension Lambda_l,s}.
\end{proof}

\noindent
By this lemma, the extended codimensions of $\bigsqcup_{s\geq q+3} \Lambda_{q-1,s}$, $\bigsqcup_{s\geq q+2} \Lambda_{l,s}$ for $l=q-2,q-3$, and $\bigsqcup_{s\geq q+1} \Lambda_{q-4,s}$ are greater than $4$. 
For this reason, in what follows, we will only analyze jets in $(\pi^m_2)^{-1}(\Lambda_{l,s})$ for 
\begin{align*}
(l,s)=&(q-1,q),(q-1,q+1),(q-1,q+2),(q-2,q),\\
&(q-2,q+1),(q-3,q),(q-3,q+1),(q-4,q)
\end{align*}
with suitable orders $m$ one by one.

\subsubsection*{Jets in $\Lambda_{q-1,q}$($=(\pi^2_2)^{-1}(\Lambda_{q-1,q})$)}

A jet in $\Lambda_{q-1,q}$ is $\mathcal{K}[G]^2$-equivalent to 
\begin{equation}\label{eq:2jet Lambda_q-1,q}
\left( x_1, \ldots, x_{q-1}, \sum_{j=1}^{q-1} \pm x_j+\sum_{j=q}^{n}\pm x_j^2\right). 
\end{equation}
As is shown in Appendix~\ref{sec:calculation codim determinacy type 1k}, this $2$-jet is $2$-determined and its $\mathcal{K}[G]^2$-codimension is $n-q+1$. 
In particular, $(\pi^5_2)^{-1}(\Lambda_{q-1,q})$ is equal to $B_{1,2}$, and the extended codimension of $B_{1,2}$ is $1$.

\subsubsection*{Jets in $(\pi^3_2)^{-1}(\Lambda_{q-1,q+1})$}

A jet in $\Lambda_{q-1,q+1}$ is $\mathcal{K}[G]^2$-equivalent to the $2$-jet represented by 
\begin{equation}\label{eq:2jet Lambda_q-1,q+1}
f=\left( x_1, \ldots, x_{q-1}, \sum_{j=1}^{q-1} \delta_jx_j+\sum_{j=q+1}^{n}\delta_jx_j^2\right). 
\end{equation}
The followings then hold (where $\alpha$ is a multi-index with $|\alpha|>1$): 
{\allowdisplaybreaks
\begin{align*}
2\delta_s x_sx_{\alpha}e_q &= tf(x_\alpha e_s)\in T\mathcal{K}[G]_1(f) &(s\geq q+1), \\
\delta_j x_jx_{\alpha}e_q &= tf(x_\alpha e_j)-f^\ast X_je_j\in T\mathcal{K}[G]_1(f) &(j\leq q-1), \\
x_{\alpha}(e_j+\delta_j e_q) & = tf(x_\alpha e_j)\in T\mathcal{K}[G]_1(f) &(j\leq q-1). 
\end{align*}
}%
We can thus deduce the following inclusion for $m\geq 3$:
\[
\mathcal{M}_n^m\mathcal{E}_n^q\subset \left<x_q^me_q\right>_{\R}+T\mathcal{K}[G]_1(f)+\mathcal{M}_n^{m+1}\mathcal{E}_n^q. 
\]
Therefore, using Theorem~\ref{thm:complete transversal}, we can deduce that an $m$-jet $\sigma \in J^m(n,q)_0$ with $\pi^{m}_{m-1}(\sigma)=j^{m-1}f(0)$ is $\mathcal{K}[G]^m$-equivalent to either the $m$-jet of the germ of type $(1,m)$ in Table~\ref{table:generic constraint r=0} or $j^mf(0)$ for $m\geq 3$. 
As shown in Appendix~\ref{sec:calculation codim determinacy type 1k}, the $m$-jet represented by the germ of type $(1,m)$ is $m$-determined and its $\mathcal{K}[G]^m$-codimension is $n-q+m-1$. 
We can thus conclude that any jet in $(\pi^5_2)^{-1}(\Lambda_{q-1,q+1})$ is $\mathcal{K}[G]^5$-equivalent to either the jet represented by the germ of type $(1,m)$ for $m=3,4,5$, or the jet $j^5f(0)$ (with some signs), and the extended codimension of $B_{1,m}$ is equal to $m-1$, whereas that of the complement $(\pi^5_2)^{-1}(\Lambda_{q-1,q+1})\setminus \left(\bigsqcup_mB_{1,m}\right)$ is equal to $5$.

\subsubsection*{Jets in $(\pi^3_2)^{-1}(\Lambda_{q-1,q+2})$}

A jet in $\Gamma_{q-1,q+2}$ is $\mathcal{K}[G]^2$-equivalent to the $2$-jet represented by 
\begin{equation}\label{eq:2jet Lambda_q-1,q+2}
f=\left( x_1, \ldots, x_{q-1}, \sum_{j=1}^{q-1} \pm x_j+\sum_{j=q+2}^{n}\pm x_j^2\right). 
\end{equation}
We can deduce the following inclusion in the same way as that for jets in $(\pi^5_2)^{-1}(\Lambda_{q-1,q+1})$: 
\[
\mathcal{M}_n^3\mathcal{E}_n^q\subset \left<x_q^3e_q,x_q^2x_{q+1}e_q,x_qx_{q+1}^2e_q,x_{q+1}^3e_q\right>_{\R}+T\mathcal{K}[G]_1(f)+\mathcal{M}_n^{4}\mathcal{E}_n^q. 
\]
By Theorem~\ref{thm:complete transversal}, a $3$-jet $\sigma\in J^3(n,q)_0$ with $(\pi^3_2)(\sigma)=j^2f(0)$ is $\mathcal{K}[G]^3$-equivalent to the following $3$-jet for some $\alpha_0,\ldots,\alpha_3\in \R$:
\begin{equation}\label{eq:3jet Lambda_q-1,q+2}
\left( x_1, \ldots, x_{q-1}, \sum_{j=1}^{q-1} \pm x_j+\alpha_0x_q^3+\alpha_1x_q^2x_{q+1}+\alpha_2 x_qx_{q+1}^2 + \alpha_3 x_{q+1}^3+\sum_{j=q+2}^{n}\pm x_j^2\right). 
\end{equation}

In the same way as that in \cite{Gibson_book}, one can show that an appropriate linear transformation in
\begin{math}
\left( x_q, x_{q+1} \right)
\end{math}
brings the $3$-jet to one of those in Table~\ref{table:jet in Gamma_q-1,q+2}.
\begin{table}[h]
 \begin{center}
  \begin{tabular}{|c|c|c|} \hline
   \# & normal form & $\mathcal{K}[G]^3$-cod. \\ \hline
$1$ & $\left( x_1, \ldots, x_{q-1}, \sum_{j=1}^{q-1} \pm x_j+x_q^3 \pm x_qx_{q+1}^2+\sum_{j=q+2}^{n}\pm x_j^2\right)$ & $n-q+4$ \\
   \hline
$2$ & $\left( x_1, \ldots, x_{q-1}, \sum_{j=1}^{q-1} \pm x_j+x_qx_{q+1}^2+\sum_{j=q+2}^{n}\pm x_j^2\right)$ & $n-q+5$ \\
   \hline
$3$ & $\left( x_1, \ldots, x_{q-1}, \sum_{j=1}^{q-1} \pm x_j+x_q^3+\sum_{j=q+2}^{n}\pm x_j^2\right)$ & $n-q+6$ \\
   \hline
$4$ & $\left( x_1, \ldots, x_{q-1}, \sum_{j=1}^{q-1} \pm x_j+\sum_{j=q+2}^{n}\pm x_j^2\right)$ & $n-q+8$ \\
   \hline   
  \end{tabular}
\caption{List of the normal forms and their $\mathcal{K}[G]^3$-codimension of the $3$-jet \eqref{eq:3jet Lambda_q-1,q+2}.}
  \label{table:jet in Gamma_q-1,q+2}
 \end{center}
\end{table}
The 
\begin{math}
\mathcal{K} \left[ G \right]^3
\end{math}-codimension of the $3$-jets in Table~\ref{table:jet in Gamma_q-1,q+2} can be computed as follows. Let 
\begin{math}
g_{\# i}
\end{math}
be the corresponding $3$-jet for 
\begin{math}
i = 1, \ldots, 4
\end{math}.

\noindent
\textbf{\# $1$:} The quotient space 
\begin{math}
J \left( n, q \right)_0 / T \mathcal{K} \left[ G \right]^3 \left( g_{\# 1} \right)
\end{math}
is isomorphic to 
\[
\bigl<\overbrace{x_q e_q, \ldots, x_n e_q}^{n-q+1}, \overbrace{x_q^2 e_q, x_q x_{q+1} e_q, x_{q+1}^2 e_q}^3\bigr>_\R\subset \R[[x]]^q.
\]

\noindent
\textbf{\# $2$:} The quotient space 
\begin{math}
J \left( n, q \right)_0 / T \mathcal{K} \left[ G \right]^3 \left( g_{\# 2} \right)
\end{math}
is isomorphic to 
\[
\bigl<\overbrace{x_q e_q, \ldots, x_n e_q}^{n-q+1}, \overbrace{x_q^2 e_q, x_q x_{q+1} e_q, x_{q+1}^2 e_q, x_q^3 e_q}^4\bigr>_\R\subset \R[[x]]^q.
\]

\noindent
\textbf{\# $3$:} The quotient space 
\begin{math}
J \left( n, q \right)_0 / T \mathcal{K} \left[ G \right]^3 \left( g_{\# 3} \right)
\end{math}
is isomorphic to
\[
\bigl<\overbrace{x_q e_q, \ldots, x_n e_q}^{n-q+1}, \overbrace{x_q^2 e_q, x_q x_{q+1} e_q, x_{q+1}^2 e_q, x_q x_{q+1}^2 e_q, x_{q+1}^3 e_q}^5\bigr>_\R\subset \R[[x]]^q.
\]

\noindent
\textbf{\# $4$:} The quotient space 
\begin{math}
J \left( n, q \right)_0 / T \mathcal{K} \left[ G \right]^3 \left( g_{\# 4} \right)
\end{math}
is isomorphic to
\begin{equation}
\bigl<\overbrace{x_q e_q, \ldots, x_n e_q}^{n-q+1}, \overbrace{x_q^2 e_q, x_q x_{q+1} e_q, x_{q+1}^2 e_q, x_q^3 e_q, x_q^2 x_{q+1} e_q, x_q x_{q+1}^2 e_q, x_{q+1}^3 e_q}^7\bigr>_\R\subset \R[[x]]^q.
\end{equation}

We can eventually conclude that any jet in $(\pi^3_2)^{-1}(\Lambda_{q-1,q+2})$ is $\mathcal{K}[G]^3$-equivalent to the $3$-jets in Table~\ref{table:jet in Gamma_q-1,q+2}. The jet 
\begin{math}
g_{\# 1}
\end{math}
is $3$-determined since
\begin{math}
\mathcal{M}_n^3 \mathcal{E}_n^q \subset T \mathcal{K} \left[ G \right] \left( g \right) +\mathcal{M}_n^{4}\mathcal{E}_n^q
\end{math}
holds for any germ 
\begin{math}
g
\end{math}
representing 
\begin{math}
g_{\# 1}
\end{math} (cf.~Appendix~\ref{sec:calculation codim determinacy type 1k}). The germs of type $(2)$ is one of the germ representing 
\begin{math}
g_{\# 1}
\end{math}.
Therefore, the preimage by $\pi^5_3$ of the union of the $\mathcal{K}[G]^3$-orbits of the germs of type $(2)$ (with all possible signs) is equal to $B_2$. 
Furthermore, the calculations of $\mathcal{K}[G]^3$-codimensions we have done above imply that the extended codimensions of $B_2$ and the complement $(\pi^5_2)^{-1}(\Lambda_{q-1,q+2})\setminus B_2$ are equal to $4$ and $5$, respectively.

\subsubsection*{Jets in $(\pi^4_2)^{-1}(\Lambda_{q-2,q})$}

A jet in $\Lambda_{q-2,q}$ is $\mathcal{K}[G]^2$-equivalent to
\[
\left( x_1, \ldots, x_{q-1}, \sum_{j=1}^{q-2} \pm x_j+ax_{q-1}^2+ \sum_{j=q}^{n}\pm x_j^2\right)
\]
for some $a\in \R$. 
This $2$-jet is further $\mathcal{K}[G]^2$-equivalent to the jet represented by either the germ of type $(3,2)$ in Table~\ref{table:generic constraint r=0} or the following germ: 
\[
f=\left( x_1, \ldots, x_{q-1}, \sum_{j=1}^{q-2} \pm x_j+ \sum_{j=q}^{n}\pm x_j^2\right).
\]
We can deduce the following inclusion for $m\geq 3$ in the same way as that for jets in $(\pi^5_2)^{-1}(\Lambda_{q-1,q+1})$:
\[
\mathcal{M}_n^m\mathcal{E}_n^q\subset \left<x_{q-1}^me_q\right>_{\R}+T\mathcal{K}[G]_1(f)+\mathcal{M}_n^{m+1}\mathcal{E}_n^q. 
\]
Therefore, using Theorem~\ref{thm:complete transversal}, we can deduce that an $m$-jet $\sigma \in J^m(n,q)_0$ with $\pi^{m}_{m-1}(\sigma)=j^{m-1}f(0)$ is $\mathcal{K}[G]^m$-equivalent to either the $m$-jet of the germ of type $(3,m)$ in Table~\ref{table:generic constraint r=0} or $j^mf(0)$ for $m=3,4$. 
The germ of type $(3,m)$, denoted by $g_{3,m}$, is $m$-determined relative to $\mathcal{K}[G]$ since 
\begin{math}
\mathcal{M}_n^m \mathcal{E}_n^q \subset T \mathcal{K} \left[ G \right] \left( g_{3,m} \right)
\end{math}
holds and thus Proposition~\ref{prop:basic properties K[G]-eq/codim} implies the claim. The 
\begin{math}
\mathcal{K} \left[ G \right]
\end{math}-codimension of 
\begin{math}
g_{3,m}
\end{math}
is 
\begin{math}
n - q + m
\end{math}
since
\begin{math}
\mathcal{M}_n \mathcal{E}_n^q / T \mathcal{K} \left[ G \right] \left( g_{3,m} \right)
\end{math}
is isomorphic to
\begin{math}
\bigl<\overbrace{x_{q-1} e_q, \ldots, x_n e_q}^{n-q+2}, \overbrace{x_{q-1}^2 e_q, \ldots, x_{q-1}^{m-1} e_q}^{m-2}\bigr>_\R\subset \R[[x]]^q
\end{math}.

We can thus conclude that any jet in $(\pi^4_2)^{-1}(\Lambda_{q-2,q})$ is $\mathcal{K}[G]^4$-equivalent to either the jet represented by the germ of type $(3,m)$ for $m=3,4$, or the jet $j^4f(0)$ (with some signs), and that the preimage by $\pi^5_4$ of the union of the $\mathcal{K}[G]^4$-orbits of the $4$-jets of the germ of type $(3,m)$ (with all possible signs) is equal to $B_{3,m}$.
The $\mathcal{K}[G]^4$-codimensions of the germ of type $(3,m)$ is equal to $n-q+m$. 
On the other hand, the $\mathcal{K}[G]^4$-codimension of the jet $j^4f(0)$ is equal to $n-q+5$ since 
\begin{math}
J^4 \left( n, q \right)_0 / T\mathcal{K}^4 \left[ G \right] \left( j^4f(0) \right)
\end{math}
is isomorphic to
\begin{math}
\bigl<\overbrace{x_{q-1} e_q, \ldots, x_n e_q}^{n-q+2}, \overbrace{x_{q-1}^2 e_q, x_{q-1}^3 e_q, x_{q-1}^4 e_q}^3\bigr>_\R\subset \R[[x]]^q
\end{math}.
We can thus deduce from the relation~\eqref{eq:relation d_e K[G]-codim} that the extended codimension of $B_{3,m}$ is equal to $m$, whereas that of the complement $(\pi^5_2)^{-1}(\Lambda_{q-2,q})\setminus \left(\bigsqcup_m B_{3,m}\right)$ is equal to $5$.

\subsubsection*{Jets in $(\pi^4_2)^{-1}(\Lambda_{q-2,q+1})$}

A jet in $\Lambda_{q-2,q+1}$ is $\mathcal{K}[G]^2$-equivalent to
\begin{equation}\label{eq:2jet Lambda_q-2,q+1}
\left( x_1, \ldots, x_{q-1}, \sum_{j=1}^{q-2} \pm x_j+a_1x_{q-1}^2+a_2x_{q-1}x_q+ \sum_{j=q+1}^{n}\pm x_j^2\right)
\end{equation}
for some $a_1,a_2\in \R$. 
If $a_2\neq 0$, one can change this jet to that represented by the following germ by a coordinate transformation preserving $(x_1,\ldots, x_{q-1},x_{q+1},\ldots, x_n)$:
\[
f_1=\left( x_1, \ldots, x_{q-1}, \sum_{j=1}^{q-2} \pm x_j+x_{q-1}x_q+ \sum_{j=q+1}^{n}\pm x_j^2\right)
\]
If $a_2=0$ and $a_1\neq 0$, the jet \eqref{eq:2jet Lambda_q-2,q+1} is $\mathcal{K}[G]^2$-equivalent to that represented by
\[
f_2=\left( x_1, \ldots, x_{q-1}, \sum_{j=1}^{q-2} \pm x_j\pm x_{q-1}^2+ \sum_{j=q+1}^{n}\pm x_j^2\right).
\]
If $a_1=a_2=0$, the quotient space \begin{math}
J^2 \left( n, q \right)_0 / T \mathcal{K} \left[ G \right]^2 \left( g \right)
\end{math}
for the jet $g$ in \eqref{eq:2jet Lambda_q-2,q+1} is isomorphic to
\begin{math}
\bigl<\overbrace{x_{q-1} e_q, \ldots, x_n e_q}^{n-q+2}, \overbrace{x_{q-1}^2 e_q, x_{q-1} x_q e_q, x_q^2 e_q}^3\bigr>_\R\subset \R[[x]]^q
\end{math}.
In particular, the $\mathcal{K}[G]^2$-codimension of the jet \eqref{eq:2jet Lambda_q-2,q+1} is equal to $n-q+5$.

Since $x_{q-1}x_{\alpha}e_q = tf_1(x_\alpha e_q)\in T\mathcal{K}[G]_1(f_1)$ for any multi-index $\alpha$ with $|\alpha|>1$, we can deduce the following inclusion for $m\geq 3$ in the same way as that for jets in $(\pi^5_2)^{-1}(\Lambda_{q-1,q+1})$:
\[
\mathcal{M}_n^m\mathcal{E}_n^q\subset \left<x_{q}^me_q\right>_{\R}+T\mathcal{K}[G]_1(f_1)+\mathcal{M}_n^{m+1}\mathcal{E}_n^q. 
\]
Therefore, using Theorem~\ref{thm:complete transversal}, we can deduce that an $m$-jet $\sigma \in J^m(n,q)_0$ with $\pi^{m}_{m-1}(\sigma)=j^{m-1}f_1(0)$ is $\mathcal{K}[G]^m$-equivalent to either the $m$-jet of the germ of type $(4,m)$ in Table~\ref{table:generic constraint r=0} or $j^mf_1(0)$ for $m=3,4$. 
We denote the germ of type $(4,m)$ by $g_{4,m}$.
In the same way as before, one can show that \begin{math}
\mathcal{M}_n^m \mathcal{E}_n^q
\end{math}
is contained in $T \mathcal{K} \left[ G \right] \left( g_{4,m} \right)$ and \begin{math}
\mathcal{M}_n \mathcal{E}_n^q / T \mathcal{K} \left[ G \right] \left( g_{4,m} \right)
\end{math}
is isomorphic to $\left<x_{q-1} e_q, \ldots, x_n e_q, x_q^2 e_q, \ldots, x_q^{m-1} e_q\right>_{\R}\subset \R[[x]]^q$.
In particular $g_{4,m}$ is $m$-determined and has the $\mathcal{K}[G]^m$-codimension $n-q+m$.
Thus, the union of the $\mathcal{K}[G]^4$-orbits of the germs of type $(4,m)$ (with all possible signs) is equal to $B_{4,m}$ and its extended codimension is $m$. 
On the other hand, the $\mathcal{K}[G]^4$-codimension of $j^4f_1(0)$ is equal to $n-q+5$ since
\begin{math}
J ^4 \left( n, q \right)_q / T \mathcal{K} \left[ G \right]^4 \left( j^4f_1 \left( 0 \right) \right)
\end{math}
is isomorphic to \begin{math}
\left<x_{q-1} e_q, \ldots, x_n e_q,x_q^2 e_q, x_q^3 e_q, x_q^4 e_q\right>_\R\subset\R[[x]]^q
\end{math}.

Since $\pm 2 x_{q-1}x_{\alpha}e_q = tf_2(x_\alpha e_{q-1})-f_2^\ast X_{q-1}x_\alpha e_{q-1}\in T\mathcal{K}[G]_1(f_2)$ for any multi-index $\alpha$ with $|\alpha|>1$, we can deduce the following inclusion in the same way as before:
\[
\mathcal{M}_n^3\mathcal{E}_n^q\subset \left<x_{q}^3e_q\right>_{\R}+T\mathcal{K}[G]_1(f_2)+\mathcal{M}_n^{4}\mathcal{E}_n^q. 
\]
Therefore, using Theorem~\ref{thm:complete transversal}, we can deduce that an $3$-jet $\sigma \in J^3(n,q)_0$ with $\pi^{3}_{2}(\sigma)=j^{2}f_2(0)$ is $\mathcal{K}[G]^3$-equivalent to either the $3$-jet of the germ of type $(5)$ in Table~\ref{table:generic constraint r=0} or $j^3f_2(0)$.
We denote the germ of type $(5)$ by $g$. 
In the same way as before, one can show that $\mathcal{M}_n^3 \mathcal{E}_n^q$ is contained in $T \mathcal{K} \left[ G \right] \left( g \right)$ and \begin{math}
\mathcal{M}_n \mathcal{E}_n^q / T \mathcal{K} \left[ G \right] \left( g \right)
\end{math}
is isomorphic to \begin{math}
\left<x_{q-1} e_q, \ldots, x_n e_q,x_q^2 e_q, x_{q-1} x_q e_q\right>_\R\subset \R[[x]]^q
\end{math}.
In particular $g$ is $3$-determined and its $\mathcal{K}[G]$-codimension is $n-q+4$. 
Thus, the union of the $\mathcal{K}[G]^4$-orbits of the germs of type $(5)$ (with all possible signs) is equal to $B_{5}$ and its extended codimension is $4$.
On the other hand, the $\mathcal{K}[G]^3$-codimension of $j^3f_2(0)$ is $n-q+6$ since 
\begin{math}
J^3 \left( n, q \right)_0 / T \mathcal{K} \left[ G \right]^3 \left( j^3 f_2(0) \right)
\end{math}
is isomorphic to 
\[
\left<x_{q-1} e_q, \ldots, x_n e_q,x_{q-1} x_q e_q, x_q^2 e_q, x_{q-1} x_q^2 e_q, x_q^3 e_q\right>_\R \subset \R[[x]]^q.
\]

The complement of $B_{4,3}\sqcup B_{4,4}\sqcup B_5$ in $(\pi^5_2)^{-1}(\Lambda_{q-2,q+1})$ is the following union:
\[
(\pi^5_4)^{-1}\biggl(\bigsqcup_{\begin{minipage}[c]{11mm}
\scriptsize
\centering
all signs 

in $f_1$
\end{minipage}}\mathcal{K}[G]^4\cdot j^4f_1(0)\biggr)
\sqcup (\pi^5_3)^{-1}\biggl(\bigsqcup_{\begin{minipage}[c]{11mm}
\scriptsize
\centering
all signs 

in $f_2$
\end{minipage}}(\mathcal{K}[G]^3\cdot j^3f_2(0))
\biggr)
\]
The extended codimension of the union is equal to $5$ since the $\mathcal{K}[G]^4$- (resp.~$\mathcal{K}[G]^3$-) codimension of $j^4f_1(0)$ (resp.~$j^3f_2(0)$) is equal to $n-q+5$. 

\subsubsection*{A digression on extended intrinsic derivatives for jets in $\Lambda_{l,q}$}

Before proceeding with the proof of Theorem~\ref{thm:classification jets}, we will give invariants of jets in $\Lambda_{l,q}$ under the $\mathcal{K}[G]^2$-action. 
For $\sigma = j^2f(0)\in \Lambda_{l,q}$, the number of zero entries of $\mu_f$ is $q-l-1$. 
We take indices $k_1,\ldots, k_{q-l-1}\in \{1,\ldots, q\}$ so that $k_i<k_{i+1}$ and $(\mu_f)_{k_i}=0$ (for the definition of $\mu_f$, see Subsection~\ref{sec:extended_intrinsic_derivative}). 
We take vectors $v_1(f),\ldots, v_{q-l-1}(f)\in W_f$ satisfying the following conditions:

\begin{itemize}

\item 
$\tilde{D}^2f(v_i(f)\otimes w)=0$ for any $w\in \Ker df_0$, 

\item 
$d(f_{k_i})_0(v_j(f)) = \delta_{ij}$.

\end{itemize}

\noindent
Since $D^2f$ is non-degenerate and $\tilde{D}^2f$ depends only on the $2$-jet $j^2f(0)$, the vector $v_1(f),\ldots, \linebreak v_{q-l-1}(f)$ satisfying the conditions above are uniquely determined from $\sigma = j^2f(0)$. 
For this reason, we denote $v_i(f)$ by $v_i(\sigma)$. 

\begin{lemma}\label{lem:zero set alpha_ij}

The subset $\Omega_0 = \{\sigma=j^2f(0)\in \Lambda_{l,q}~|~\tilde{D}^2f(v_i(\sigma)\otimes v_j(\sigma))=0\mbox{ for any }i\leq j\}$ is a submanifold of $\Lambda_{l,q}$ with codimension $\tilde{l}:=(q-l-1)(q-l)/2$. 

\end{lemma}

\begin{proof}[Proof of Lemma~\ref{lem:zero set alpha_ij}]
For subsets of indices 
\begin{align*}
&L=\{l_1,\ldots, l_{q-1}\}\subset \{1,\ldots, q\}, \\
&M=\{m_1,\ldots, m_{q-1}\}\subset \{1,\ldots,n\},\mbox{ and}\\ &K=\{k_1,\ldots,k_{q-l-1}\}\subset L, 
\end{align*}
we define $U_{L,M,K}\subset \Lambda_{l,q}$ as follows: 
\[
U_{L,M,K}=\left\{ \sigma=j^2f(0)\in \Lambda_{l,q}~\left|~
\begin{minipage}[c]{50mm}
$\det\left(\frac{\Pa f_{l_i}}{\Pa x_{m_j}}(0)\right)_{1\leq i,j\leq q-1}\neq 0$\\
$(\mu_f)_{k_1}=\cdots =(\mu_f)_{k_{q-l-1}}=0$
\end{minipage}\right.\right\}.
\]
The family $\{U_{L,M,K}\}_{L,M,K}$ is an open cover of $\Lambda_{l,q}$. 
Thus, it is enough to show that the intersection $\Omega_0\cap U_{L,M,K}$ is a submanifold of $U_{K,L,M}$ with codimension $\tilde{l}$. 
For simplicity, we assume $L=M=\{1,\ldots, q-1\}$ and $K=\{l+1,\ldots, q-1\}$. (One can deal with the other subsets of indices in the same way.)

For a map-germ $f\in \mathcal{M}_n\mathcal{E}_n^q$ with $j^2f(0)\in U_{L,M,K}$, we denote the diffeomorphism-germ $(f_1,\ldots, f_{q-1},x_q,\ldots, x_n)$ by $\Phi(f)$. 
Since the $i$-th component $(f\circ\Phi(f)^{-1})_i$ is equal to $x_i$ for $1\leq i \leq q-1$, $\Ker df_0$ is generated by 
\[
w_q(f):=(d\Phi(f)_0)^{-1}(\Pa_q),\ldots, w_n(f):=(d\Phi(f)_0)^{-1}(\Pa_n)\in T_0\R^n,
\]
where $\Pa_1,\ldots, \Pa_n$ are the canonical basis of $T_0\R^n$. 
Moreover, $[\Pa_q]\in \Coker df_0$ is a basis of $\Coker df_0$. 
We put $b_{ij}(f)=\frac{\Pa^2(f\circ \Phi(f)^{-1})_q}{\Pa x_i\Pa x_j}(0)$. 
Since $(b_{ij}(f))_{q\leq i,j \leq n}$ is a representation matrix of the intrinsic derivative $D^2f$ with respect to the bases above, this matrix is regular and thus there exists $c_{kj}(f)\in \R$ ($k=1,\ldots, q-l-1$, $j=q,\ldots,n$) such that the following linear equations hold:
\[
\sum_{j=q}^n c_{kj}(f)b_{ij}(f) +\frac{\Pa^2 (f\circ \Phi(f)^{-1})_q}{\Pa x_i\Pa x_{l+k}}(0) =0 \hspace{.3em}(i=q,\ldots,n).
\]
By a direct calculation, one can obtain the following equality: 
\[
v_k(\sigma) = (d\Phi(f)_0)^{-1}\left(\Pa_{l+k}+\sum_{j=q}^{n}c_{kj}(f)\Pa_j\right). 
\]
Since $(d\Phi(f)_0)^{-1}$ and $c_{kj}(f)$ depends smoothly on $j^2f(0)$, the map $A_{L,M,K}:U_{L,M,K}\to \R^{\tilde{l}}$ defined by $A_{L,M,K}(\sigma) = (\ldots, \tilde{D}^2f(v_i(\sigma)\otimes v_j(\sigma)),\ldots)$ (under the identification $\Coker df_0\cong \R$ by the basis $[\Pa_q]\in \Coker df_0$) is smooth. (Here the coordinates of $\R^{\tilde{l}}$ are labeled by $i,j$ with $1\leq i\leq j\leq q-l-1$.)

Let $\sigma = j^2f_0(0)\in U_{L,M,K}$. 
In what follows, we will show that $(dA_{L,M,K})_\sigma$ is surjective (and thus $A_{L,M,K}$ is a submersion). 
Since $\Phi(f_0)^{-1}=(f'_1,\ldots, f'_{q-1},x_q,\ldots, x_n)$ for some $f'_1,\ldots,f'_{q-1}\in \mathcal{M}_n$, one can deduce by direct calculation that the left action by $(j^2\Phi(f_0)(0),I)\in \mathcal{K}[G]^2$ (where $I$ is the unit matrix) preserves the subset $U_{L,M,K}\subset \Lambda_{l,q}$. 
Let $\ell_{\sigma}:U_{L,M,K}\to U_{L,M,K}$ be the diffeomorphism defined by this action. 
It is easy to check that the following diagram commutes: 
\[
\xymatrix{
U_{L,M,K} \ar[r]^{\ell_{\sigma}} \ar[dr]_{A_{L,M,K}} & U_{L,M,K} \ar[d]^{A_{L,M,K}} \\
& \R^{\tilde{l}}.
}
\]
Hence, one can assume $f_0=(x_1,\ldots, x_{q-1},h)$ for some $h\in \mathcal{M}_n$ without loss of generality. 
We define a map $s:\R^{\tilde{l}} \to U_{L,M,K}$ by $s(d) = j^2\tilde{f}_d(0)$, where  
\[
\tilde{f}_d:= (x_1,\ldots, x_{q-1},\tilde{h}(d)) :=  \left(x_1,\ldots, x_{q-1},h+\sum_{i<j}d_{ij}x_{l+i}x_{l+j}+\frac{1}{2}\sum_i d_{ii}x_{l+i}^2\right)
\]
for $d=(d_{ij})\in \R^{\tilde{l}}$. 
By direct calculation, one can easily check that $b_{ij}(f_0)=b_{ij}(\tilde{f}_d)$ for $q\leq i,j \leq n$, $\frac{\Pa^2 h}{\Pa x_i\Pa x_{l+k}}(0)= \frac{\Pa^2 \tilde{h}_d}{\Pa x_i\Pa x_{l+k}}(0)$ for $k=1,\ldots, q-l-1$ and $i=q,\ldots, n$, and thus $v_k(s(d)) = v_k(\sigma) = \Pa_{l+k}+ \sum_{j=q}^{n}c_{kj}(f_0)\Pa_j$. 
The following equalities thus hold:
\begin{align*}
A_{L,M,K}\circ s(d)&= \left(\ldots,\tilde{D}^2\tilde{f}_d(v_i(s(d))\otimes v_j(s(d))),\ldots\right)\\
&= \left(\ldots,\tilde{D}^2f_0(v_i(\sigma)\otimes v_j(\sigma))+d_{ij},\ldots\right) = A_{L,M,K}(\sigma)+d.
\end{align*}
In particular, the differential $d(A_{L,M,K}\circ s)_0 = (dA_{L,M,K})_\sigma \circ ds_0$ is the identity map, and thus $(dA_{L,M,K})_\sigma$ is surjective. 

The intersection $\Omega_0\cap U_{L,M,K}$ is equal to $A_{L,M,K}^{-1}(0)$, which is a submanifold of $U_{L,M,K}$ with codimension $\tilde{l}$ since $A_{L,M,K}$ is a submersion. 
\end{proof}

Let $\mathbb{P}^{\tilde{l}-1}$ be the $(\tilde{l}-1)$-dimensional real projective space, whose homogeneous coordinates are labeled by $i,j$ with $1\leq i\leq j \leq q-l-1$. 
Taking an isomorphism $\Coker (df_0)\cong \R$, we regard $\tilde{D}^2f(v_i(f)\otimes v_j(f))$ as a real value, which we denote by $\alpha_{ij}(f)$ or $\alpha_{ij}(\sigma)$. 
We define $A:\Lambda_{l,q}\setminus \Omega_0\to \mathbb{P}^{\tilde{l}-1}$ by $A(\sigma)=[\cdots :\alpha_{ij}(\sigma):\cdots]$. 
Note that this map does not depend on the choice of an isomorphism $\Coker(df_0)\cong \R$. 

\begin{proposition}\label{prop:submersion map A}

The map $A$ is a submersion. 

\end{proposition}

\begin{proof}[Proof of Proposition~\ref{prop:submersion map A}]
The statement follows from the fact that $A|_{U_{L,M,K}\setminus \Omega_0}$ is equal to $\pi \circ A_{L,M,K}$, where $\pi:\R^{\tilde{l}}\setminus \{0\}\to \mathbb{P}^{\tilde{l}-1}$ is the projection. 
\end{proof}

\subsubsection*{Jets in $(\pi^3_2)^{-1}(\Lambda_{q-3,q})$}

A jet in $\Lambda_{q-3,q}$ is $\mathcal{K}[G]^2$-equivalent to
\begin{equation}
\left( x_1, \ldots, x_{q-1}, \sum_{j=1}^{q-3} \pm x_j + \alpha_{11} x_{q-2}^2 + \alpha_{12} x_{q-2} x_{q-1} + \alpha_{22} x_{q-1}^2 + \sum_{j=q}^n \pm x_j^2 \right) \label{eq:2jet Lambda_q-3,q}
\end{equation}
for some $\alpha_{ij}\in \R$. 
The 
\begin{math}
\mathcal{K} \left[ G \right]^2
\end{math}-codimension of $j^2g_\alpha(0)$ is as shown in Table~\ref{table:kgtangent_Lambda_q-3,q_j2} (see Appendix~\ref{sec:appendix_kg2codimension_Lambda_q-3,q}).
\begin{table}[h]
 \begin{center}
  \begin{tabular}{|l|l|l|} \hline
   class \# & $\alpha_{ij}$'s & $\mathcal{K}[G]^2$-cod.  \\ \hline
   1 & $\alpha_{11} \alpha_{22} \neq 0$ or $\alpha_{11} \alpha_{12} \neq 0$ or $\alpha_{22} \alpha_{12} \neq 0$ & $n-q+4$ \\ \hline
   2 & ($\alpha_{11} = \alpha_{12} = 0$ and $\alpha_{22} \neq 0$) & $n-q+5$ \\
   & or ($\alpha_{22} = \alpha_{12} = 0$ and $\alpha_{11} \neq 0$) & \\
   & or  ($\alpha_{11} = \alpha_{22} = 0$ and $\alpha_{12} \neq 0$) & \\ \hline
   3 & $\alpha_{11} = \alpha_{12} = \alpha_{22} = 0$ & $n-q+6$ \\
\hline
  \end{tabular}
  \caption{The $\mathcal{K}[G]^2$-codimension of the $2$-jet \eqref{eq:2jet Lambda_q-3,q}.}
  \label{table:kgtangent_Lambda_q-3,q_j2}
 \end{center}
\end{table}
It is easy to see that the set of $2$-jets $\mathcal{K}[G]^2$-equivalent to that of the class $2$ or $3$, denoted by $\Omega_{23}\subset \Lambda_{q-3,q}$, is equal to $\Omega_0\cup A^{-1}(\{[1:0:0],[0:1:0],[0:0:1]\})$. 
By Lemma~\ref{lem:zero set alpha_ij} and Proposition~\ref{prop:submersion map A}, the codimension of $\Omega_0$ in $\Lambda_{q-3,q}$ is equal to $3$, while that of $A^{-1}(\{[1:0:0],[0:1:0],[0:0:1]\})$ is equal to $2$. 
Thus, the extended codimension of $\Omega_{23}$ is equal to $2+d_e(\Lambda_{q-3,q}) = 5$. 
In what follows, we consider the jet of class 1 in Table~\ref{table:kgtangent_Lambda_q-3,q_j2}.

\noindent
{\bf Case $\boldsymbol{\alpha_{11} \alpha_{22} \neq 0}$}~:~In this case, the $2$-jet \eqref{eq:2jet Lambda_q-3,q} can be normalized to 
\begin{equation}\label{eq:normal form 2jet generic Lambda_q-3,q}
\left( x_1, \ldots, x_{q-1}, \sum_{j=1}^{q-3} \pm x_j + \delta_{1} x_{q-1}^2 + \alpha x_{q-1} x_{q-2} + \delta_{2} x_{q-2}^2 + \sum_{j=q}^n \pm x_j^2 \right), 
\end{equation}
by an appropriate scaling transformation, where $\alpha\in \R$ and $\delta_i=\pm1$.

Let $V\subset \Lambda_{q-3,q}$ be the following subset:
\begin{multline}
\mathcal{K}[G]^2\cdot \left\{ \left( x_1, \ldots, x_{q-1}, \sum_{j=1}^{q-3} \pm x_j + \delta_{1} x_{q-1}^2 + \alpha x_{q-1} x_{q-2} + \delta_{2} x_{q-2}^2 + \sum_{j=q}^n \pm x_j^2 \right) \right. \\
\left. \alpha \in \mathbb{R}, 4 \delta_{1} \delta_{2} - \alpha^2 \neq 0 \right\}.
\end{multline}
It is easy to see that $V$ is equal to $A^{-1}(W)$, where 
\[
W = \{[\alpha_{11}:\alpha_{12}:\alpha_{22}]\in \mathbb{P}^2~|~\alpha_{11},\alpha_{22}\neq 0, 4\alpha_{11}\alpha_{22}-\alpha_{12}^2\neq 0\}. 
\]
Since $W$ is an open subset of $\mathbb{P}^2$, $V$ is also an open subset of $\Lambda_{q-3,q}$ by Proposition~\ref{prop:submersion map A}. 
In particular the extended codimension of $V$ is equal to $d_e(\Lambda_{q-3,q})=3$.

Let $g\in \mathcal{M}_n\mathcal{E}_n^q$ be a map-germ representing the $2$-jet \eqref{eq:normal form 2jet generic Lambda_q-3,q}.
If 
\begin{math}
4 \delta_{1} \delta_{2} - \alpha^2 \neq 0
\end{math}
holds, 
\begin{math}
\mathcal{M}_n^3 \mathcal{E}_n^q \subset T \mathcal{K} \left[ G \right]_1 \left( g \right) + \mathcal{M}_n^4 \mathcal{E}_n^q
\end{math}
holds by the similar argument.
We can thus deduce from Proposition~\ref{prop:basic properties K[G]-eq/codim} that $g$ is $3$-determined, and further deduce from Theorem~\ref{thm:complete transversal} that $g$ is $2$-determined. 
Note that $B_6 = (\pi^5_2)^{-1}(V)$ is an open subset (and thus a submanifold) of $(\pi^5_2)^{-1}(\Lambda_{q-3,q})$.

\begin{remark}

Using the extended intrinsic derivative $\tilde{D}^2f$ (in particular considering the value $\tilde{D}^2f(v_1(f)\otimes v_2(f))$), one can also show that two $2$-jets of the form \eqref{eq:normal form 2jet generic Lambda_q-3,q} with distinct $\alpha$ are not $\mathcal{K}[G]^2$-equivalent.

\end{remark}

 If 
\begin{math}
4 \delta_{1} \delta_{2} - \alpha^2 = 0
\end{math}, 
the $2$-jet \eqref{eq:normal form 2jet generic Lambda_q-3,q} is $\mathcal{K}[G]^2$-equivalent to the jet represented by
\begin{equation}
f_3=\left( x_1, \ldots, x_{q-1}, \sum_{j=1}^{q-3} \pm x_j +\delta' \left(  x_{q-2} +\delta''  x_{q-1} \right)^2 + \sum_{j=q}^n \pm x_j^2 \right).
\end{equation}
Since $2\delta'(x_{q-2}+\delta'' x_{q-1})x_{q-2}x_\alpha e_q=tf_3(x_{q-2}x_\alpha e_{q-2})-(f_3^\ast X_{q-2})x_\alpha e_{q-2}\in T\mathcal{K}[G]_1(f_3)$ and $2\delta'\delta''(x_{q-2}+\delta'' x_{q-1})x_{q-1}x_\alpha e_q=tf_3(x_{q-1}x_\alpha e_{q-1})-(f_3^\ast X_{q-1})x_\alpha e_{q-1}\in T\mathcal{K}[G]_1(f_3)$ for a multi-index $\alpha$ with $|\alpha|>0$, we can deduce the following inclusion in the same way as that for jets in $(\pi^5_2)^{-1}(\Lambda_{q-1,q+1})$:
\[
\mathcal{M}_n^3\mathcal{E}_n^q\subset \left<x_{q-1}^3\right>_{\R}+T\mathcal{K}[G]_1(f_3)+\mathcal{M}_n^4\mathcal{E}_n^q. 
\]
By Theorem~\ref{thm:complete transversal}, there exists 
\begin{math}
\beta \in \mathbb{R}
\end{math}
such that $j^3 f_3(0)$ is 
\begin{math}
\mathcal{K} \left[ G \right]^3
\end{math}-equivalent to 
\begin{equation}
\left( x_1, \ldots, x_{q-1}, \sum_{j=1}^{q-3} \pm x_j +\delta' \left(  x_{q-2} +\delta''  x_{q-1} \right)^2 + \beta x_{q-1}^3 + \sum_{j=q}^n \pm x_j^2 \right).
\end{equation}
If 
\begin{math}
\beta \neq 0
\end{math}
holds, the
\begin{math}
\mathcal{K} \left[ G \right]
\end{math}-codimension of the jet above is $n-q+4$ since 
\begin{math}
\mathcal{M}_n \mathcal{E}_n^q/T \mathcal{K} \left[ G \right] \left( g \right)
\end{math}
is isomorphic to 
\begin{math}
\left<x_{q-2} e_q, \ldots, x_n e_q,x_{q-1}^2 e_q\right>_{\R}\subset \R[[x]]^q \end{math} for a germ $g$ representing the jet above. 
Furthermore, 
$\mathcal{M}_n^3 \mathcal{E}_n^q$ is contained in $T \mathcal{K} \left[ G \right] \left( g \right)$ and thus 
\begin{math}
g
\end{math}
is 
\begin{math}
3
\end{math}-determined relative to 
\begin{math}
\mathcal{K} \left[ G \right]
\end{math}
by Proposition~\ref{prop:basic properties K[G]-eq/codim}.
An appropriate scaling of the coordinate brings the jet above to the normal form of type $(7)$ in Table~\ref{table:generic constraint r=0}. 
If 
\begin{math}
\beta = 0
\end{math}, the $3$-jet above is equal to $j^3f_3(0)$ and it has 
\begin{math}
\mathcal{K} \left[ G \right]^3
\end{math}-codimension 
\begin{math}
n - q + 5
\end{math}
since 
\begin{math}
J^3 \left( n, q \right)_0 / T \mathcal{K} \left[ G \right]^3 \left( j^3f_3(0) \right)
\end{math}
is isomorphic to 
\begin{math}
\left<x_{q-2} e_q, \ldots, x_n e_q,x_{q-1}^2 e_q, x_{q-1}^3 e_q \right>_\R\subset \R[[x]]^q
\end{math}. 

\noindent
{\bf Case $\boldsymbol{(\alpha_{11} \alpha_{12} \neq 0\land \alpha_{22} = 0)}$ or $\boldsymbol{(\alpha_{22} \alpha_{12} \neq0 \land\alpha_{11} = 0)}$}~:~In this case, the $2$-jet \eqref{eq:2jet Lambda_q-3,q} is $\mathcal{K}[G]^2$-equivalent to that represented by 
\begin{equation}
f_4=\left( x_1, \ldots, x_{q-1}, \sum_{j=1}^{q-3} \pm x_j +\delta_{q-2}x_{q-2}^2+ \delta_{q-2,q-1} x_{q-2} x_{q-1} + \sum_{j=q}^n \pm x_j^2 \right).
\end{equation}
Since $(2\delta_{q-2}x_{q-2}^2+\delta_{q-2,q-1} x_{q-2}x_{q-1})x_\alpha e_q$ is equal to $tf_4(x_{q-2}x_\alpha e_{q-2})-(f_4^\ast X_{q-2})x_\alpha e_{q-2}\in T\mathcal{K}[G]_1(f_4)$ for a multi-index $\alpha$ with $|\alpha|>0$, we can deduce the following inclusion in the same way as before:
\[
\mathcal{M}_n^3\mathcal{E}_n^q\subset \left<x_{q-2}^3\right>_{\R}+T\mathcal{K}[G]_1(f_4)+\mathcal{M}_n^4\mathcal{E}_n^q. 
\]
By Theorem~\ref{thm:complete transversal}, there exists 
\begin{math}
\beta \in \mathbb{R}
\end{math}
such that $j^3 f_4(0)$ is 
\begin{math}
\mathcal{K} \left[ G \right]^3
\end{math}-equivalent to 
\begin{equation}
\left( x_1, \ldots, x_{q-1}, \sum_{j=1}^{q-3} \pm x_j + \beta x_{q-2}^3+\delta_{q-2}x_{q-2}^2+\delta_{q-2,q-1}x_{q-2}x_{q-1}  + \sum_{j=q}^n \pm x_j^2 \right).
\end{equation}
If 
\begin{math}
\beta \neq 0
\end{math}
holds, the 
\begin{math}
\mathcal{K} \left[ G \right]
\end{math}-codimension of the jet above is $n-q+4$ by the similar argument.
Furthermore, $\mathcal{M}_n^3 \mathcal{E}_n^q$ is contained in $T \mathcal{K} \left[ G \right] \left( g \right) + \mathcal{M}_n^4 \mathcal{E}_n^q$, 
and thus the jet above is $3$-determined relative to $\mathcal{K}[G]$ by Proposition~\ref{prop:basic properties K[G]-eq/codim}.
An appropriate scaling of the coordinate brings the map-germ to the normal form of type $(8)$ in Table~\ref{table:generic constraint r=0}. 
If 
\begin{math}
\beta = 0
\end{math}, the $3$-jet above is equal to $j^3f_4(0)$ and 
it has $\mathcal{K}[G]^3$-codimension $n-q+5$ by the similar argument.

We can eventually conclude that the extended codimensions of $B_6,B_7$ and $B_8$ are equal to $3,4$ and $4$, respectively, the complement of the union of them in $(\pi^5_2)^{-1}(\Lambda_{q-3,q})$ is equal to 
\[
(\pi^5_2)^{-1}(\Omega_{23}) \sqcup (\pi^5_3)^{-1}\biggl(\biggl(
\bigsqcup_{\begin{minipage}[c]{11mm}
\scriptsize
\centering
all signs 

in $f_3$
\end{minipage}}\left(\mathcal{K}[G]^3\cdot j^3f_3(0)\right)\biggr)
\sqcup \biggl(\bigsqcup_{\begin{minipage}[c]{11mm}
\scriptsize
\centering
all signs 

in $f_4$
\end{minipage}}(\mathcal{K}[G]^3\cdot j^3f_4(0))
\biggr)\biggr),
\]
and its extended codimension is $5$. 

\subsubsection*{Jets in $(\pi^3_2)^{-1}(\Lambda_{q-3,q+1})$}

In this case, by using Lemma~\ref{lem:normal form corank1 2-jet q=0}, a jet in $\Lambda_{q-3,q+1}$ is $\mathcal{K}[G]^2$-equivalent to the $2$-jet represented by 
\begin{equation}
g_a = \left( x_1, \ldots, x_{q-1}, \sum_{j=1}^{q-3} \pm x_j+\sum_{j_1=q-2}^{q-1}\sum_{j_2=j_1}^{q}a_{j_1j_2}x_{j_1}x_{j_2}+ \sum_{j=q+1}^{n}\pm x_j^2\right)\label{eq:2jet Lambda_q-3,q+1}
\end{equation}
for $a = \left( a_{q-2,q-2}, a_{q-2,q-1}, a_{q-2,q}, a_{q-1,q-1}, a_{q-1,q} \right) \in \mathbb{R}^5$.
In the same manner as in Appendix~\ref{sec:appendix_kg2codimension_Lambda_q-3,q}, one can take a basis of the quotient space $\mathcal{M}_n\mathcal{E}_n^q/(T\mathcal{K}[G](g_a)+\mathcal{M}_n^3\mathcal{E}_n^q)$ for each $a\in \R^5$ as shown in Appendix~\ref{sec:table7_computation}. 
We investigate each case in what follows.
\begin{table}[h]
 \begin{center}
  \begin{tabular}{|l|l|l|} \hline
   \# & $a_{ij}$'s & $\mathcal{K}[G]^2$-cod.  \\ \hline
   1 & $P(a) \neq 0$ & $n-q+4$ \\ \hline
   2 & $P(a) = 0$ and ($a_{q-2,q} a_{q-1,q} \neq 0$& $n-q+5$ \\
   &  or $a_{q-2,q-2} a_{q-1,q} \neq 0$ or  $a_{q-2,q-2} a_{q-1,q-1} \neq 0$) & \\ \hline
   3 & ($a_{q-2,q-2} = a_{q-2,q} = 0$ and $a_{q-1,q} \neq 0$) & $n-q+6$ \\
   & ($a_{q-1,q-1} = a_{q-1,q} = 0$ and $a_{q-2,q} \neq 0$) & \\ \hline
   4 & $a_{q-2,q} = a_{q-1,q} = 0$ and ($a_{q-2,q-1}a_{q-1,q-1} \neq 0$ or  & $n-q+7$ \\
   & $a_{q-2,q-2}a_{q-1,q-1} \neq 0$ or $a_{q-2,q-2}a_{q-2,q-1} \neq 0$) & \\ \hline
   5 & $a_{q-2,q} = a_{q-1,q} = a_{q-2,q-2} = a_{q-2,q-1} = 0, a_{q-1,q-1} \neq 0$ & $n-q+8$ \\
   & or $a_{q-2,q} = a_{q-1,q} = a_{q-2,q-2} = a_{q-1,q-1} = 0, a_{q-2,q-1} \neq 0$ & \\
   & or $a_{q-2,q} = a_{q-1,q} = a_{q-2,q-1} = a_{q-1,q-1} = 0, a_{q-2,q-2} \neq 0$ & \\ \hline
   6 & $a_{q-2,q-2} = a_{q-2,q-1} = a_{q-2,q} = a_{q-1,q-1} = a_{q-1,q} = 0$ & $n-q+9$ \\ \hline
  \end{tabular}
  \captionsetup{singlelinecheck=off}
  \caption{The $\mathcal{K}[G]^2$-codimension of the $2$-jet \eqref{eq:2jet Lambda_q-3,q+1}, where $P(a)$ is given in \eqref{eq:definition P}. 
}
  \label{table:kgtangent_Lambda_q-3,q+1_j2}
 \end{center}
\end{table}

Let
\begin{math}
\iota \colon \mathbb{R}^5 \rightarrow \Lambda_{q-3,q+1}
\end{math}
be a map defined as 
\begin{math}
\iota \left( a \right) = j^2 g_a
\end{math}
for $a\in \R^5$.
The mapping 
\begin{math}
\iota
\end{math}
is a smooth embedding and thus 
\begin{math}
\iota \left( \mathbb{R}^5 \right)
\end{math}
is a smooth submanifold in 
\begin{math}
\Lambda_{q-3,q+1}
\end{math}.

\noindent 
\textbf{Case \#1:}
Let $A_1 = \{a\in \R^5~|~P(a)\neq 0\}$ for
\begin{equation}\label{eq:definition P}
P(a)=a_{q-2,q} a_{q-1,q} \left( a_{q-2,q}^2 a_{q-1,q-1} - a_{q-2,q-1} a_{q-2,q} a_{q-1,q} + a_{q-2,q-2} a_{q-1,q}^2 \right).
\end{equation}
The set is an open subset of the parameter space
\begin{math}
\mathbb{R}^5
\end{math}
and thus
\begin{math}
\iota(A_1)
\end{math}
is a smooth manifold in 
\begin{math}
\Lambda_{q-3,q+1}
\end{math}. In addition, 
\begin{math}
\cfrac{\partial j^2 g_a \left( 0 \right)}{\partial a_{j_1j_2}} \in T\mathcal{K} \left[ G \right]^2 \left( j^2 g_a \left( 0 \right) \right)
\end{math}
holds for all 
\begin{math}
j_1 \in \left\{ q-2, q-1 \right\}
\end{math}
and 
\begin{math}
j_2 \in \left\{ q-2, \ldots, q \right\}, j_2 \ge j_1
\end{math}, which can be checked in the same manner as in Appendix~\ref{sec:appendix_kg2codimension_Lambda_q-3,q}. 
Therefore, each connected component of 
\begin{math}
\iota(A_1)
\end{math}
is contained in a single 
\begin{math}
\mathcal{K} \left[ G \right]^2
\end{math}-orbit by Mather's lemma (\cite{Mather1969}). 
We can choose a representative from each connected component of $A_1$ as
\begin{math}
\left( a_{q-2,q-1}, a_{q-2,q}, a_{q-1,q} \right) = \left( \pm 1, \pm 1, \pm 1 \right)
\end{math}
and 
\begin{math}
a_{q-2,q-2} = a_{q-1,q-1} = 0
\end{math}. The corresponding germ representing it is 
\begin{equation}
f = \left( x_1, \ldots, x_{q-1}, \sum_{j=1}^{q-3} \pm x_j \pm x_{q-2} x_{q-1} \pm x_{q-2} x_q \pm x_{q-1} x_q + \sum_{j=q+1}^{n}\pm x_j^2\right).
\end{equation}
In the same manner as before, we can deduce the following inclusion:
\begin{equation}
\mathcal{M}_n^3 \mathcal{E}_n^q \subset \langle x_q^3 \rangle_{\mathbb{R}} + T \mathcal{K} \left[ G \right]_1 \left( g \right) + \mathcal{M}_n^4 \mathcal{E}_n^q.
\end{equation}
Therefore, using Theorem~\ref{thm:complete transversal}, there exists 
\begin{math}
\beta \in \mathbb{R}
\end{math}
such that 
\begin{math}
j^3 f \left( 0 \right)
\end{math}
is 
\begin{math}
\mathcal{K} \left[ G \right]^3
\end{math}-equivalent to 
\begin{equation}
\left( x_1, \ldots, x_{q-1}, \sum_{j=1}^{q-3} \pm x_j \pm x_{q-2} x_{q-1} \pm x_{q-2} x_q \pm x_{q-1} x_q + \beta x_q^3 + \sum_{j=q+1}^{n}\pm x_j^2\right).
\end{equation}
If 
\begin{math}
\beta \neq 0
\end{math}, it is 
\begin{math}
3
\end{math}-determined relative to 
\begin{math}
\mathcal{K} \left[ G \right]
\end{math}
and its
\begin{math}
\mathcal{K} \left[ G \right]^3
\end{math}-codimension is
\begin{math}
n-q+4
\end{math}
by the similar argument before. An appropriate scaling brings the $3$-jet to that represented by type~(9) in Table~\ref{table:generic constraint r=0}. If 
\begin{math}
\beta = 0
\end{math}, it has 
\begin{math}
\mathcal{K} \left[ G \right]^3
\end{math}-codimension 
\begin{math}
n-q+5
\end{math}
by the similar argument before.

\noindent
\textbf{Case \#2:} For the polynomial $P(a)$ given in \eqref{eq:definition P}, the singular locus of the algebraic set in 
\begin{math}
\mathbb{R}^5
\end{math}
defined by
\begin{math}
P = 0
\end{math}
is contained in 
\begin{equation}
V_{\mathbb{R}} \left( \left< P, \frac{\partial P}{\partial a_{q-2,q-2}}, \frac{\partial P}{\partial a_{q-2,q-1}}, \frac{\partial P}{\partial a_{q-2,q}}, \frac{\partial P}{\partial a_{q-1,q-1}}, \frac{\partial P}{\partial a_{q-1,q}} \right>_{\R[a]} \right)
\end{equation}
where 
\begin{math}
V_{\mathbb{R}} \left( I \right) = \left\{ a \in \mathbb{R}^5 \middle| \forall f \in I, f \left( a \right) = 0 \right\}
\end{math}. It is easy to check that the radical ideal of 
\begin{math}
I
\end{math}
is
\begin{math}
\sqrt{I} = \langle a_{q-2,q} a_{q-1,q}, a_{q-2,q-2} a_{q-1,q}, a_{q-2,q-2} a_{q-1,q-1} \rangle_{\R[a]}
\end{math} (see e.g.~\cite{BeckerBook}), and thus the singular locus is contained in the set defined by
\begin{math}
a_{q-2,q} a_{q-1,q} = a_{q-2,q-2} a_{q-1,q} = a_{q-2,q-2} a_{q-1,q-1} = 0
\end{math}. This proves that the set $A_2$ in 
\begin{math}
\mathbb{R}^5
\end{math}
defined by the condition of Case \#2 is a smooth manifold. The tangent space of the $A_2$ at 
\begin{math}
a
\end{math}
is 
\begin{math}
\langle v_1, v_2, v_3, v_4 \rangle_{\mathbb{R}}
\end{math}
where
\begin{align}
v_1 &= a_{q-2,q} \frac{\partial}{\partial a_{q-2,q-2}} + a_{q-1,q} \frac{\partial}{\partial a_{q-2,q-1}},  \\
v_2 &= a_{q-2,q} \frac{\partial}{\partial a_{q-2,q-1}} + a_{q-1,q} \frac{\partial}{\partial a_{q-1,q-1}}, \\
v_3 &= 4 a_{q-2,q-1} \frac{\partial}{\partial a_{q-2,q-1}} - 3 a_{q-2,q} \frac{\partial}{\partial a_{q-2,q}} + 8 a_{q-1,q-1} \frac{\partial}{\partial a_{q-1,q-1}} + a_{q-1,q} \frac{\partial}{\partial a_{q-1,q-1}}, \\
v_4 &= 3 a_{q-2,q-2} \frac{\partial}{\partial a_{q-2,q-2}} + 2a_{q-2,q-1} \frac{\partial}{\partial a_{q-2,q-1}}+a_{q-1,q-1} \frac{\partial}{\partial a_{q-1,q-1}}-a_{q-1,q} \frac{\partial}{\partial a_{q-1,q}}.
\end{align}
The subset 
\begin{math}
\iota(A_2)\subset \Lambda_{q-3,q+1}
\end{math}
is also a smooth manifold and whose tangent space at 
\begin{math}
j^2 g_a \left( 0 \right)
\end{math}
is spanned by 
\begin{math}
d\iota_a \left( v_i \right)
\end{math}
for 
\begin{math}
i \in \left\{ 1, \ldots, 4 \right\}
\end{math}.
\begin{math}
d \iota_a \left( v_i \right) \in T \mathcal{K} \left[ G \right]^2 \left( j^2 g_a \left( 0 \right) \right)
\end{math}
holds for all 
\begin{math}
a \in A_2
\end{math}
and
\begin{math}
i \in \left\{ 1, \ldots, 4 \right\}
\end{math}. (This can be checked by computing standard basis of 
\begin{math}
T \mathcal{K} \left[ G \right] \left( g_a \left( 0 \right) \right) + \mathcal{M}_n^3 \mathcal{E}_n^q
\end{math}
for each parameter value of $a$ and dividing a polynomial representing 
\begin{math}
d \iota_a \left( v_i \right)
\end{math}
by the standard basis.) Finally, by applying Mather's lemma, we can conclude that each connected component of 
\begin{math}
\iota(A_2)
\end{math}
is contained in the single orbit. This specifically means that jets satisfying the condition consists of finite number of 
\begin{math}
\mathcal{K} \left[ G \right]^2
\end{math}-orbits with 
\begin{math}
\mathcal{K} \left[ G \right]^2
\end{math}-codimension 
\begin{math}
n-q+5
\end{math}.

\noindent
\textbf{Case \#3:} In this case, an appropriate coordinate transformation brings the $2$-jet to
\begin{equation}
\left( x_1, \ldots, x_{q-1}, \sum_{j=1}^{q-3} \pm x_j + x_{q-1} x_q + \sum_{j=q+1}^{n}\pm x_j^2\right).
\end{equation}
In this case, the corresponding set in the $2$-jet space consists of a finite number of orbits and thus their 
\begin{math}
\mathcal{K} \left[ G \right]^2
\end{math}-codimensions are 
\begin{math}
n-q+6
\end{math}.

\noindent
\textbf{Case \#4:} In this case, an appropriate coordinate transformation brings the $2$-jet to either 
\begin{equation}
\left( x_1, \ldots, x_{q-1}, \sum_{j=1}^{q-3} \pm x_j \pm x_{q-2}^2 \pm x_{q-1}^2 + \sum_{j=q+1}^{n}\pm x_j^2\right)
\end{equation}
or
\begin{equation}
\left( x_1, \ldots, x_{q-1}, \sum_{j=1}^{q-3} \pm x_j \pm x_{q-2}^2 \pm x_{q-2} x_{q-1} + \sum_{j=q+1}^{n}\pm x_j^2\right).
\end{equation}
In this case, the corresponding set in the $2$-jet space consists of a finite number of orbits and thus their 
\begin{math}
\mathcal{K} \left[ G \right]^2
\end{math}-codimensions are 
\begin{math}
n-q+7
\end{math}.

\noindent
\textbf{Case \#5:} In this case, an appropriate coordinate transformation brings the $2$-jet to either 
\begin{equation}
\left( x_1, \ldots, x_{q-1}, \sum_{j=1}^{q-3} \pm x_j \pm x_{q-2}^2 + \sum_{j=q+1}^{n}\pm x_j^2\right)
\end{equation}
or
\begin{equation}
\left( x_1, \ldots, x_{q-1}, \sum_{j=1}^{q-3} \pm x_j \pm x_{q-2} x_{q-1} + \sum_{j=q+1}^{n}\pm x_j^2\right).
\end{equation}
In this case, the corresponding set in the $2$-jet space consists of a finite number of orbits and thus their 
\begin{math}
\mathcal{K} \left[ G \right]^2
\end{math}-codimensions are 
\begin{math}
n-q+8
\end{math}.

\noindent
\textbf{Case \#6:} In this case, the corresponding set in the $2$-jet space consists of a finite number of orbits of 
\begin{math}
j^2 g_0 \left( 0 \right)
\end{math}
and its 
\begin{math}
\mathcal{K}^2 \left[ G \right]
\end{math}
codimension is 
\begin{math}
n-q+9
\end{math}.

Let
\begin{math}
\Omega_{1c}
\end{math}
be the union of 
\begin{math}
\mathcal{K} \left[ G \right]^2
\end{math}-orbits of the classes except for \#1 in Table~\ref{table:kgtangent_Lambda_q-3,q+1_j2}. 
As we have shown, the set \begin{math}
\Omega_{1c}
\end{math}
is a finite union of the $\mathcal{K}[G]^2$-orbits in 
\begin{math}
\Lambda_{q-3,q+1}
\end{math}
whose extended codimension is $5$. 
The complement of $B_9$ in \begin{math}
\left( \pi_2^5 \right)^{-1} \left( \Lambda_{q-3,q+1} \right)
\end{math}
is the union
\[
(\pi^5_2)^{-1}(\Omega_{1c}) \sqcup (\pi^5_3)^{-1}\biggl(\biggl(
\bigsqcup_{\begin{minipage}[c]{11mm}
\scriptsize
\centering
all signs 

in $f$
\end{minipage}}\left(\mathcal{K}[G]^3\cdot j^3f(0)\right)\biggr)
\]
whose extended codimension is $5$.

\subsubsection*{Jets in $(\pi^3_2)^{-1}(\Lambda_{q-4,q})$}

A jet in $\Lambda_{q-4,q}$ is $\mathcal{K}[G]^2$-equivalent to that represented by
\begin{equation}\label{eq:2-jet Lambda_{q-4,q}}
f_{\alpha}=\left( x_1, \ldots, x_{q-1}, \sum_{j=1}^{q-4}\pm x_j + \sum_{1\leq i\leq j\leq 3} \alpha_{ij} x_{q-i} x_{q-j} + \sum_{j=q}^n \pm x_j^2 \right)
\end{equation}
for some $\alpha = (\alpha_{ij})_{i,j}\in \R^6$. 
Take subsets $W_1,W_2,W_3\subset \mathbb{P}^{5}$ as follows:
{\allowdisplaybreaks
\begin{align*}
W_1=&\{[\cdots:\alpha_{ij}:\cdots]\in \mathbb{P}^5~|~\alpha_{11}\alpha_{22}\alpha_{33}=0\}, \\
W_2 =&\{[\cdots:\alpha_{ij}:\cdots]\in \mathbb{P}^5~|~(\alpha_{11}\alpha_{22}-\alpha_{12}^2)(\alpha_{22}\alpha_{33}-\alpha_{23}^2)(\alpha_{33}\alpha_{11}-\alpha_{13}^2)=0\},\\
W_3 =&\{[\cdots:\alpha_{ij}:\cdots]\in \mathbb{P}^5~|~\alpha_{11} \alpha_{22} \alpha_{33} + 2 \alpha_{12} \alpha_{13} \alpha_{23} - \alpha_{33} \alpha_{12}^2 - \alpha_{22} \alpha_{13}^2 - \alpha_{11} \alpha_{23}^2=0\}.
\end{align*}
}%
These are proper algebraic subsets of $\mathbb{P}^5$. 
In particular, one can decompose them into submanifolds of $\mathbb{P}^5$ with codimension at least one. 
By Proposition~\ref{prop:submersion map A}, the preimage $A^{-1}(W_1\cup W_2\cup W_3)$ is a union of submanifolds of $\Lambda_{q-4,q}$ with codimension at least one. 
We can deduce from the observation above and Lemma~\ref{lem:zero set alpha_ij} that the extended codimension of $A^{-1}(W_1\cup W_2\cup W_3)\cup \Omega_0$ is (larger than or) equal to $1+d_e(\Lambda_{q-4,q})=5$. 

In what follows, we will consider a $2$-jet in the complement $\Lambda_{q-4,q}\setminus (A^{-1}(W_1\cup W_2\cup W_3)\cup \Omega_0)$, which is $\mathcal{K}[G]^2$-equivalent to $j^2f_\alpha(0)$ (where $f_\alpha$ is given in \eqref{eq:2-jet Lambda_{q-4,q}}) with 
\begin{align}
&\alpha_{11} \neq 0, \alpha_{22} \neq 0, \alpha_{33} \neq 0, \\
&4\alpha_{11} \alpha_{22} - \alpha_{12}^2 \neq 0, 4\alpha_{11} \alpha_{33} - \alpha_{13}^2 \neq 0, 4\alpha_{22} \alpha_{33} - \alpha_{23}^2 \neq 0, \label{eq:condision alpha_ij}\\
&4\alpha_{11} \alpha_{22} \alpha_{33} + \alpha_{12} \alpha_{13} \alpha_{23} - \alpha_{33} \alpha_{12}^2 - \alpha_{22} \alpha_{13}^2 - \alpha_{11} \alpha_{23}^2 \neq 0.
\end{align}
For such a $2$-jet, one can check that the inclusion
\begin{math}
\mathcal{M}_n^3 \mathcal{E}_n^q \subset T + T \mathcal{K} \left[ G \right]_1 \left( f_\alpha \right) + \mathcal{M}_n^4 \mathcal{E}_n^q
\end{math}
holds for 
\begin{math}
T = \langle x_{q-3} x_{q-2} x_{q-1} e_q \rangle_{\mathbb{R}}
\end{math}. 
By Theorem~\ref{thm:complete transversal}, a $3$-jet $\sigma\in J^3(n,q)_0$ with $\pi^3_2(\sigma)=j^2f_\alpha(0)$ is $\mathcal{K}[G]^3$-equivalent to 
\begin{equation}
\left( x_1, \ldots, x_{q-1}, \sum_{j=1}^{q-3} \delta_j x_j + \sum_{i,j = 1}^3 \alpha_{ij} x_{q-4+i} x_{q-4+j} + \beta x_{q-3} x_{q-2} x_{q-1} + \sum_{j=q}^n \delta_j x_j^2 \right),
\end{equation}
for some 
\begin{math}
\beta\in \R
\end{math}.
If 
\begin{math}
\beta \neq 0
\end{math}, the $\mathcal{K}[G]^3$-codimension of the jet above is $n-q+4$ by the similar argument. 
Furthermore, $\mathcal{M}_n^3\mathcal{E}_n^q$ is contained in $T\mathcal{K}[G](g)+\mathcal{M}_n^4\mathcal{E}_n^q$, and thus the jet above is $3$-determined relative to $\mathcal{K}[G]$ by Proposition~\ref{prop:basic properties K[G]-eq/codim}. 
An appropriate scaling of the coordinate brings the map-germ to the normal form of type $(10)$ in Table~\ref{table:generic constraint r=0}.
If $\beta =0$, the $3$-jet above is equal to $j^3f_\alpha(0)$ and it has $\mathcal{K}[G]^3$-codimension $n-q+5$ be the similar argument. 

We have shown that the extended codimension of $B_{10}$ is equal to $4$, the complement $(\pi^5_2)^{-1}(\Lambda_{q-4,q})\setminus B_{10}$ is equal to 
\[
(A\circ \pi^5_2)^{-1}(W_1\cup W_2\cup W_3)\sqcup (\pi^5_3)^{-1}\biggl(\biggl(
\bigsqcup_{\begin{minipage}[c]{11mm}
\scriptsize
\centering
all signs 

in $f_\alpha$
\end{minipage}}\left(\mathcal{K}[G]^3\cdot j^3f_\alpha(0)\right)\biggr), 
\]
and its extended codimension is $5$. 

We can eventually conclude that the complement 
\[
J^5(n,q)_0\setminus \left(B_0\sqcup \left(\bigsqcup_{i,k}B_{i,k}\right)\sqcup \left(\bigsqcup_{j}B_{j}\right)\right)
\]
has extended codimension $5$, completing the proof of 3 of Theorem~\ref{thm:classification jets}.

\subsection*{Classification of jets in $W_{n,q,1}$ with $1\leq q\leq 3$}

In what follows, we will show 4 of Theorem~\ref{thm:classification jets}. 
Since we are supposing $n\gg q$, in particular $\codim(\Sigma_2,J^1(n,q+1)_0)$ is large enough, it suffices to consider the case that $\corank (d(g,h)_0)$ is equal to $1$, which is also equal to $\corank(dg_0)+1$. 
By applying an appropriate action of 
\begin{math}
\mathcal{K} \left[ G \right]^1
\end{math}, 
one can assume 
$j^1 \left( g, h \right) \left( 0 \right) = \left( x_1, \ldots, x_q, 0 \right)$
without loss of generality.
By using the similar argument as Lemma~\ref{lem:normal form corank1 2-jet q=0}, an appropriate action of 
\begin{math}
\mathcal{K} \left[ G \right]
\end{math}
brings 
\begin{math}
j^2 \left( g, h \right) \left( 0 \right)
\end{math}
to the following form:
\begin{equation}
j^2 \left( g, h \right) \left( 0 \right) = \left( x_1, \ldots, x_q, \sum_{j_1=1}^q\sum_{j_2=1}^{s-1}a_{j_1j_2}x_{j_1}x_{j_2}+ \sum_{j=s}^{n}\pm x_j^2\right) \label{eq:normal form 2-jet r=1}
\end{equation}
for some $s\in \{q+1,\ldots, n\}$ and $a_{j_1j_2}\in \R$. 
Let $\Theta_{q,s}$ be the set of $2$-jets $\mathcal{K}[G]^2$-equivalent to that in \eqref{eq:normal form 2-jet r=1} for some signs and $a_{j_1j_2}\in \R$. 
It is easy to check that $\Theta_{q,s}$ is equal to $(\pi^2_1|_{Q_s})^{-1}(\Theta')$, where $\Theta' = \{j^1(g,0)(0)\in \Sigma_1\setminus \Sigma_2~|~dg_0\mbox{~:~regular}\}$ and $Q_s$ is given in the proof of Lemma~\ref{lem:estimate codimension Lambda_l,s}.
Since $\Theta'$ is a submanifold in $\Sigma_1\setminus \Sigma_2$ with codimension $q$, $\Theta_{q,s}$ is also a submanifold of $Q_s$ and its extended codimension is equal to $q+1+(s-q)(s-q-1)/2$. The extended codimensions of $\bigsqcup_{s\geq 4} \Theta_{1,s}$, $\bigsqcup_{s\geq 5} \Theta_{2,s}$, and $\bigsqcup_{s\geq 5} \Theta_{3,s}$ are greater than $4$. 
For this reason, in what follows, we will only analyze jets in $(\pi^m_2)^{-1}(\Theta)$ for 
\begin{math}
\Theta=\Theta_{1,2}, \Theta_{1,3}, \Theta_{2,3}, \Theta_{2,4}, \Theta_{3,4},
\end{math}
with suitable orders $m$ one by one. 
As all the calculations needed to obtain determinacies, codimensions, and complete transversals for various jets/germs are quite similar to those we have done in the proof of 3 of Theorem~\ref{thm:classification jets}, we will omit them for simplicity of the manuscript.

\subsubsection*{Jets in $(\pi^4_2)^{-1}(\Theta_{1,2})$}
A jet in 
\begin{math}
\Theta_{1,2}
\end{math}
is 
\begin{math}
\mathcal{K} \left[ G \right]^2
\end{math}-equivalent to the
\begin{math}
2
\end{math}-jet represented by
\begin{equation}
\left( g, h \right) = \left( x_1, a_{11} x_1^2+ \sum_{j=2}^{n}\pm x_j^2\right).
\end{equation}
If 
\begin{math}
a_{11} \neq 0
\end{math}
holds, an appropriate scaling brings the $2$-jet to the normal form of type 
\begin{math}
\left( 1, 2 \right)
\end{math}
in Table~\ref{table:generic constraint q>0 r=1} which is $2$-determined and its extended codimension is $2$. If 
\begin{math}
a_{11} = 0
\end{math}
holds, the following inclusion holds for 
\begin{math}
m \ge 3
\end{math}:
\begin{equation}
\mathcal{M}_n^m \mathcal{E}_n^2 \subset \langle x_1^m e_2 \rangle_{\mathbb{R}} + T \mathcal{K} \left[ G \right]_1 \left( \left( g, h \right) \right) + \mathcal{M}_n^{m+1} \mathcal{E}_n^2.
\end{equation}
Therefore, using Theorem~2.1, we can deduce that an 
\begin{math}
m
\end{math}-jet 
\begin{math}
\sigma \in J^m \left( n, 2 \right)_0
\end{math}
with 
\begin{math}
\pi_{m-1}^m \left( \sigma \right) = j^{m-1} \left( g, h \right) \left( 0 \right)
\end{math}
is 
\begin{math}
\mathcal{K} \left[ G \right]^m
\end{math}-equivalent to either the 
\begin{math}
m
\end{math}-jet of the germ of type 
\begin{math}
\left( 1, m \right)
\end{math}
in Table.~\ref{table:generic constraint q>0 r=1} or 
\begin{math}
j^m \left( g, h \right) \left( 0 \right)
\end{math}. The 
\begin{math}
m
\end{math}-jet represented by the germ of type 
\begin{math}
\left( 1, m \right)
\end{math}
is $m$-determined and its extended codimension is 
\begin{math}
m
\end{math}. We can thus conclude that any jet in 
\begin{math}
\left( \pi_2^4 \right)^{-1} \left( \Theta_{1,2} \right)
\end{math}
is 
\begin{math}
\mathcal{K} \left[ G \right]^4
\end{math}-equivalent to either the jet represented by the germ of type 
\begin{math}
\left( 1, m \right)
\end{math}
for 
\begin{math}
m = 2, 3, 4
\end{math}, or the jet 
\begin{math}
j^4 \left( g, h \right) \left( 0 \right)
\end{math} (with some signs), and the extended codimension of 
\begin{math}
C_{1,m}
\end{math}
is equal to 
\begin{math}
m
\end{math}, whereas that of the complement 
\begin{math}
\left( \pi_2^4 \right)^{-1} \left( \Theta_{1,2} \right) \setminus \left( \bigsqcup_m C_{1,m} \right)
\end{math}
is equal to $5$.

\subsubsection*{Jets in $(\pi^3_2)^{-1}(\Theta_{1,3})$}
A jet in 
\begin{math}
\Theta_{1,3}
\end{math}
is 
\begin{math}
\mathcal{K} \left[ G \right]^2
\end{math}-equivalent to the
\begin{math}
2
\end{math}-jet represented by
\begin{equation}
\left( g, h \right) = \left( x_1, a_{11} x_1^2 + a_{12} x_1 x_2 + \sum_{j=3}^{n}\pm x_j^2\right).
\end{equation}
If 
\begin{math}
a_{12} \neq 0
\end{math}
holds, an appropriate action of 
\begin{math}
\mathcal{K} \left[ G \right]^2
\end{math}
brings the $2$-jet to the $2$-jet represented by
\begin{equation}
f_1 = \left( x_1, x_1 x_2 + \sum_{j=3}^{n}\pm x_j^2\right).
\end{equation}
If 
\begin{math}
a_{12} = 0
\end{math}
and 
\begin{math}
a_{11} \neq 0
\end{math}
holds, an appropriate action of 
\begin{math}
\mathcal{K} \left[ G \right]^2
\end{math}
brings the $2$-jet to the $2$-jet represented by
\begin{equation}
f_2 = \left( x_1, x_1^2 + \sum_{j=3}^{n}\pm x_j^2\right).
\end{equation}
If 
\begin{math}
a_{11} = a_{12} = 0
\end{math}
holds, the quotient space 
\begin{math}
J^2 \left( n, 2 \right)_0 / T \mathcal{K} \left[ G \right]^2 \left( j^2 \left( g, h \right) \left( 0 \right) \right)
\end{math}
is isomorphic to
\begin{math}
\bigl<x_1 e_2, \ldots, x_n e_2, x_1^2 e_q, x_1 x_2 e_2, x_2^2 e_2\bigr>_\R\subset \R[[x]]^2
\end{math}. In particular, the 
\begin{math}
\mathcal{K} \left[ G \right]^2
\end{math}-codimension of 
\begin{math}
j^2 \left( g, h \right) \left( 0 \right)
\end{math}
is equal to 
\begin{math}
n + 3
\end{math}
and its extended codimension is $5$.

An 
\begin{math}
m
\end{math}-jet 
\begin{math}
\sigma \in J^m \left( n, q \right)_0
\end{math}
with 
\begin{math}
\pi_{m-1}^m \left( \sigma \right) = j^{m-1} f_1 \left( 0 \right)
\end{math}
is 
\begin{math}
\mathcal{K} \left[ G \right]^m
\end{math}-equivalent to either the 
\begin{math}
m
\end{math}-jet of the germ of type 
\begin{math}
\left( 3, m \right)
\end{math}
in Table~\ref{table:generic constraint q>0 r=1} or 
\begin{math}
j^m f_1 \left( 0 \right)
\end{math}
for
\begin{math}
m = 3, 4
\end{math}. The germ of type 
\begin{math}
\left( 3, m \right)
\end{math}
in Table~\ref{table:generic constraint q>0 r=1} is 
\begin{math}
m
\end{math}-determined and has the 
\begin{math}
\mathcal{K} \left[ G \right]^m
\end{math}-codimension 
\begin{math}
n - 2 + m
\end{math}. Thus, the union of the 
\begin{math}
\mathcal{K} \left[ G \right]^4
\end{math}-orbits of the germs of type 
\begin{math}
\left( 4, m \right)
\end{math}
(with all possible signs) is equal to 
\begin{math}
C_{3,m}
\end{math}
and its extended codimension is 
\begin{math}
m
\end{math}. On the other hand, the 
\begin{math}
\mathcal{K} \left[ G \right]^4
\end{math}-codimension of 
\begin{math}
j^4 f_1 \left( 0 \right)
\end{math}
is equal to 
\begin{math}
n + 3
\end{math}
since 
\begin{math}
J^4 \left( n, 2 \right)_0 / T \mathcal{K} \left[ G \right]^4 \left( j^4 f_1 \left( 0 \right) \right)
\end{math}
is isomorphic to 
\begin{math}
\langle x_1 e_2, \ldots, x_n e_2, x_2^2 e_2, x_2^3 e_2, x_2^4 e_2 \rangle_{\mathbb{R}} \subset \mathbb{R} \left[ \left[ x \right] \right]^2
\end{math}. 

An 
\begin{math}
3
\end{math}-jet 
\begin{math}
\sigma \in J^3 \left( n, q \right)_0
\end{math}
with 
\begin{math}
\pi_2^3 \left( \sigma \right) = j^2 f_2 \left( 0 \right)
\end{math}
is 
\begin{math}
\mathcal{K} \left[ G \right]^3
\end{math}-equivalent to either the 
\begin{math}
3
\end{math}-jet of the germ of type 
\begin{math}
\left( 2 \right)
\end{math}
in Table~\ref{table:generic constraint q>0 r=1} or 
\begin{math}
j^3 f_2 \left( 0 \right)
\end{math}. The germ of type 
\begin{math}
\left( 2 \right)
\end{math}
in Table~\ref{table:generic constraint q>0 r=1} is 
\begin{math}
3
\end{math}-determined and its 
\begin{math}
\mathcal{K} \left[ G \right]
\end{math}-codimension is 
\begin{math}
n + 2
\end{math}. Thus, the union of the 
\begin{math}
\mathcal{K} \left[ G \right]^4
\end{math}-orbits of the germs of type 
\begin{math}
\left( 2 \right)
\end{math}
(with all possible signs) is equal to 
\begin{math}
C_2
\end{math}
and its extended codimension is 
\begin{math}
4
\end{math}. On the other hand, the 
\begin{math}
\mathcal{K} \left[ G \right]^3
\end{math}-codimension of 
\begin{math}
j^3 f_2 \left( 0 \right)
\end{math}
is 
\begin{math}
n + 4
\end{math}.

The compliment of 
\begin{math}
C_{3,3} \bigsqcup C_{3,4} \bigsqcup C_2
\end{math}
in 
\begin{math}
\left( \pi_2^4 \right)^{-1} \left( \Theta_{1,3} \right)
\end{math}
is the following union: 
\[
\biggl(\bigsqcup_{\begin{minipage}[c]{11mm}
\scriptsize
\centering
all signs 

in $f_1$
\end{minipage}}\mathcal{K}[G]^4\cdot j^4f_1(0)\biggr)
\sqcup (\pi^4_3)^{-1}\biggl(\bigsqcup_{\begin{minipage}[c]{11mm}
\scriptsize
\centering
all signs 

in $f_2$
\end{minipage}}(\mathcal{K}[G]^3\cdot j^3f_2(0))
\biggr)
\]
The extended codimension of the union is equal to 
\begin{math}
5
\end{math}
since the 
\begin{math}
\mathcal{K} \left[ G \right]^4
\end{math}- (reps. 
\begin{math}
\mathcal{K} \left[ G \right]^3
\end{math}- ) codimension of 
\begin{math}
j^4 f_1 \left( 0 \right)
\end{math}
(resp. 
\begin{math}
j^3 f_2 \left( 0 \right)
\end{math}) is equal to 
\begin{math}
n + 3
\end{math}.

\subsubsection*{A digression on intrinsic derivatives for jets in $\Theta_{q,q+1}$}

Before proceeding with the proof of Theorem~\ref{thm:classification jets}, we will give invariants of jets in $\Theta_{q,q+1}$ under the $\mathcal{K}[G]^2$-action. 
For $\sigma = j^2(g,h)(0)\in \Theta_{q,q+1}$, we take vectors $v_1'(\sigma),\ldots, v_q'(\sigma)\in T_0\R^n$ satisfying the following conditions: 
\begin{itemize}

\item 
$D^2h(v_i'(\sigma)\otimes w) = 0$ for any $w\in \Ker dg_0$, 

\item 
$d(g_j)_0(v_i'(\sigma)) = \delta_{ij}$. 

\end{itemize}

\noindent
Since $D^2h$ is invariant under $\mathcal{K}[G]^2$-action (cf.~Lemma~\ref{L:invariance intrinsic derivative}) and $D^2(g,h) = D^2h|_{\otimes^2 \Ker dg_0}$ is non-degenerate, the vectors $v_1'(\sigma),\ldots, v_q'(\sigma)$ satisfying the conditions  above are uniquely determined from $\sigma$. 
One can further show the following lemma in the same way as that in the proof of Lemma~\ref{lem:zero set alpha_ij}. 

\begin{lemma}\label{lem:zero set alpha'_ij}

The subset 
\[
\Omega'_0 = \{\sigma=j^2(g,h)(0)\in \Theta_{q,q+1}~|~ D^2h(v_i'(\sigma),v_j'(\sigma))=0\mbox{ for any }i\leq j\}
\]
is a submanifold of $\Theta_{q,q+1}$ with codimension $\tilde{q}=q(q+1)/2$. 

\end{lemma}

\noindent
Under the canonical identification $\Coker dh_0\cong T_0\R\cong \R$, we put $\alpha'_{ij}(\sigma) = D^2h(v'_i(\sigma)\otimes v'_j(\sigma))$ and define $A':\Theta_{q,q+1}\to \mathbb{P}^{\tilde{q}-1}$ by $A'(\sigma) = [\cdots:\alpha'_{ij}(\sigma):\cdots]$.

\begin{proposition}\label{prop:submersion map A'}

The map $A'$ is a submersion. 

\end{proposition}

\noindent
The proof of this proposition is quite similar to that of Proposition~\ref{prop:submersion map A} and left for the reader.

\subsubsection*{Jets in $(\pi^3_2)^{-1}(\Theta_{2,3})$}
A jet in 
\begin{math}
\Theta_{2,3}
\end{math}
is 
\begin{math}
\mathcal{K} \left[ G \right]^2
\end{math}-equivalent to the
\begin{math}
2
\end{math}-jet represented by
\begin{equation}
\left( g, h \right) = \left( x_1, x_2, \alpha_{11} x_1^2 + \alpha_{12} x_1 x_2 + \alpha_{22} x_2^2 + \sum_{j=3}^{n}\pm x_j^2\right).
\end{equation}
Let
\begin{math}
\Omega'_{23} = \Omega'_0\cup (A')^{-1}(\{[1:0:0],[0:1:0],[0:0:1]\}) \subset \Theta_{2,3}
\end{math}. 
By Lemma~\ref{lem:zero set alpha'_ij} and Proposition~\ref{prop:submersion map A'}, the extended codimension of $\Omega'_{23}$ is 
\begin{math}
5
\end{math}.
One can further show that \begin{math}
\Theta_{2,3}
\end{math}
is the union of 
\begin{math}
\mathcal{K} \left[ G \right]^3
\end{math}-orbits of 
\begin{math}
C_4, C_5
\end{math}, and 
\begin{math}
C_6
\end{math}
whose extended codimensions are 
\begin{math}
3, 4
\end{math}, and 
\begin{math}
4
\end{math}, respectively, and 
\[
(\pi^3_2)^{-1}(\Omega'_{23}) \sqcup\biggl(
\bigsqcup_{\begin{minipage}[c]{11mm}
\scriptsize
\centering
all signs 

in $f_3$
\end{minipage}}\left(\mathcal{K}[G]^3\cdot j^3f_3(0)\right)\biggr)
\sqcup \biggl(\bigsqcup_{\begin{minipage}[c]{11mm}
\scriptsize
\centering
all signs 

in $f_4$
\end{minipage}}(\mathcal{K}[G]^3\cdot j^3f_4(0))
\biggr),
\]
whose extended codimension is $5$ where
\begin{math}
f_3 = \left( x_1, x_2, \pm \left( x_1 \pm x_2 \right)^2 + \sum_{j=3}^{n}\pm x_j^2\right)
\end{math}
and 
\begin{math}
f_4 =  \left( x_1, x_2, \pm x_1^2 \pm x_1 x_2 + \sum_{j=3}^{n}\pm x_j^2\right)
\end{math}.

\subsubsection*{Jets in $(\pi^3_2)^{-1}(\Theta_{2,4})$}
A jet in 
\begin{math}
\Theta_{2,4}
\end{math}
is 
\begin{math}
\mathcal{K} \left[ G \right]^2
\end{math}-equivalent to the
\begin{math}
2
\end{math}-jet represented by
\begin{equation}
\left( g, h \right) \left( 0 \right) = \left( x_1, x_2, a_{11} x_1^2 + a_{12} x_1 x_2 + a_{13} x_1 x_3 + a_{22} x_2^2 + a_{23} x_2 x_3 + \sum_{j=4}^{n}\pm x_j^2\right).
\end{equation}
Let 
\begin{math}
\Omega_{1c}
\end{math}
be 
\begin{math}
\mathcal{K} \left[ G \right]^2
\end{math}-orbit of the set of the $2$-jets whose coefficients satisfying 
\begin{equation}
a_{13} a_{23} \left( a_{13}^2 a_{22} - a_{12} a_{13} a_{23} + a_{11} a_{23}^2 \right) = 0
\end{equation}
and let 
\begin{equation}
f = \left( x_1, x_2, \pm x_1 x_2 \pm x_1 x_3 \pm x_2 x_3 + \sum_{j=q+1}^{n}\pm x_j^2\right).
\end{equation}
be a map-germ. In the same manner as before, we can deduce that 
\begin{math}
(\pi^3_2)^{-1}(\Theta_{2,4})
\end{math}
is the union of 
\begin{math}
C_7
\end{math}
whose extended codimension is $4$ and 
\[
(\pi^3_2)^{-1}(\Omega_{1c}) \sqcup\biggl(
\bigsqcup_{\begin{minipage}[c]{11mm}
\scriptsize
\centering
all signs 

in $f$
\end{minipage}}\left(\mathcal{K}[G]^3\cdot j^3f(0)\right)\biggr)
\]
whose extended codimension is $5$.

\subsubsection*{Jets in $(\pi^3_2)^{-1}(\Theta_{3,4})$}

A jet in $\Theta_{3,4}$ is $\mathcal{K}[G]^2$-equivalent to that represented by 
\begin{equation}\label{eq:2-jet Theta_3,4}
\left( g, h_\alpha \right) = \left( x_1, x_2, x_3, \sum_{1\leq i\leq j\leq 3} \alpha_{ij}x_ix_j + \sum_{j=4}^{n}\pm x_j^2\right)
\end{equation}
for some $\alpha = (\alpha_{ij})\in \R^6$. 
Let $W_1,W_2,W_3\subset \mathbb{P}^5$ be the subsets we took when dealing with jets in $\Lambda_{q-4,q}$. 
By Proposition~\ref{prop:submersion map A'}, the preimage $(A')^{-1}(W_1\cup W_2\cup W_3)$ is a union of submanifolds of $\Theta_{3,4}$ with codimension at least one. 
Thus, the extended codimension of $(A')^{-1}(W_1\cup W_2\cup W_3)\cup \Omega_0'$ is (larger than or) equal to $1+d_e(\Theta_{3,4})=5$. 

In what follows, we will consider a $2$-jet in the complement $\Theta_{3,4}\setminus ((A')^{-1}(W_1\cup W_2\cup W_3)\cup \Omega_0')$, which is $\mathcal{K}[G]^2$-equivalent to $j^2(g,h_\alpha)(0)$ with $\alpha=(\alpha_{ij})$ satisfying the conditions in \eqref{eq:condision alpha_ij}.  
For such a $2$-jet, one can check that the inclusion
\begin{math}
\mathcal{M}_n^3 \mathcal{E}_n^{4} \subset T + T \mathcal{K} \left[ G \right]_1 \left( g,h_\alpha \right) + \mathcal{M}_n^4 \mathcal{E}_n^{4}
\end{math}
holds for 
\begin{math}
T = \langle x_1x_2x_3e_4 \rangle_{\mathbb{R}}
\end{math}. 
By Theorem~\ref{thm:complete transversal}, a $3$-jet $\sigma\in J^3(n,4)_0$ with $\pi^3_2(\sigma)=j^2(g,h_\alpha)(0)$ is $\mathcal{K}[G]^3$-equivalent to 
\begin{equation}
\left( x_1,x_2,x_3, \sum_{i,j = 1}^3 \alpha_{ij} x_{i} x_{j} + \beta x_1x_2x_3 + \sum_{j=4}^n \delta_j x_j^2 \right),
\end{equation}
for some 
\begin{math}
\beta\in \R
\end{math}.
If 
\begin{math}
\beta \neq 0
\end{math}, the $\mathcal{K}[G]^3$-codimension of the jet above is $n$ by the similar argument. 
Furthermore, $\mathcal{M}_n^3\mathcal{E}_n^4$ is contained in $T\mathcal{K}[G](g,h)+\mathcal{M}_n^4\mathcal{E}_n^4$, and thus the jet above is $3$-determined relative to $\mathcal{K}[G]$ by Proposition~\ref{prop:basic properties K[G]-eq/codim}. 
An appropriate scaling of the coordinate brings the map-germ to the normal form of type $(8)$ in Table~\ref{table:generic constraint q>0 r=1}.
If $\beta =0$, the $3$-jet above is equal to $j^3(g,h_\alpha)(0)$ and it has $\mathcal{K}[G]^3$-codimension $n+1$ be the similar argument.

We have shown that the extended codimension of $C_{8}$ is equal to $4$, the complement $(\pi^3_2)^{-1}(\Theta_{3,4})\setminus C_8$ is equal to 
\[
(A'\circ \pi^3_2)^{-1}(W_1\cup W_2\cup W_3)\sqcup \biggl(
\bigsqcup_{\begin{minipage}[c]{11mm}
\scriptsize
\centering
all signs 

in $h_\alpha$
\end{minipage}}\left(\mathcal{K}[G]^3\cdot j^3(g,h_\alpha)(0)\right)\biggr), 
\]
and its extended codimension is $5$. 
This completes the proof of 4 of Theorem~\ref{thm:classification jets}. 

Lastly, we can obtain a basis of the quotient space $\mathcal{E}_n^{q+r}/T\mathcal{K}[G]_e(g,h)$ for each germ $(g,h)$ in Tables \ref{table:generic constraint q=0}--\ref{table:generic constraint q>0 r=1} either by calculating standard basis of $T\mathcal{K}[G]_e(g,h)$ in the same way as those in Appendix~\ref{sec:appendix}, or by using Proposition~\ref{prop:relation K[G]/K[G]_e-cod}. 
The details are left to the reader. 
We eventually complete the proof of Theorem~\ref{thm:classification jets}. 
\end{proof}

\subsection*{The main theorem in full detail}

Combining Theorem~\ref{thm:classification jets} with the results in Section~\ref{sec:transversality parameter family}, we eventually obtain the following theorem. 

\begin{theorem}\label{thm:main theorem detail}

Let $N$ be a manifold without boundary, $b\leq 4$, and $U\subset \R^b$ be an open subset. 
The set consisting of constraint mappings $(g,h)\in C^\infty(N\times U,\R^{q+r})$ with the following conditions is residual in $C^\infty(N\times U,\R^{q+r})$. 

\begin{enumerate}

\item 
For any $u\in U$ and ${\overline{x}}\in M(g_u,h_u)$, the corank of $(dh_u)_{\overline{x}}$ is at most $1$. 

\item 
For any $u\in U$ and ${\overline{x}}\in M(g_u,h_u)$ at which there is no active inequality constraint (i.e.~there is no $k\in \{1,\ldots, q\}$ with $g_k({\overline{x}},u)=0$), a full reduction of the germ $(g,h):(N\times U,({\overline{x}},u))\to \R^{q+r}$ is $\mathcal{K}[G]$-equivalent to either the trivial family of the constant map-germ, or a versal unfolding of one of the germs in Table~\ref{table:generic constraint q=0} with $\mathcal{K}[G]_e$-codimension at most $b$. 

\end{enumerate}

\noindent
In what follows, we will assume that $(g_u,h_u)$ has an active inequality constraint at $x\in M(g_u,h_u)$.

\begin{enumerate}

\addtocounter{enumi}{2}
\item 
For any $u\in U$ and ${\overline{x}}\in M(g_u,h_u)$ with $\corank((dh_u)_{\overline{x}})=0$, a full reduction of the germ $(g_u,h_u):(N,{\overline{x}})\to \R^{q+r}$ is $\mathcal{K}[G]$-equivalent to either a submersion-germ, or one of the germs in Table~\ref{table:generic constraint r=0} with stratum $\mathcal{K}[G]_e$-codimension at most $b$.
Furthermore, if a full reduction of $(g_u,h_u)$ is $\mathcal{K}[G]$-equivalent to the germ of neither type (6) nor type (10), a full reduction of $(g,h):(N\times U,({\overline{x}},u))\to \R^{q+r}$ is a versal unfolding of $(g_u,h_u)$. 

\item 
For any $({\overline{x}},u)\in N\times U$ with $\corank((dh_u)_{\overline{x}})=1$, a full reduction of the germ $(g_u,h_u):(N,{\overline{x}})\to \R^{q+r}$ is $\mathcal{K}[G]$-equivalent to one of the germs in Table~\ref{table:generic constraint q>0 r=1} with stratum $\mathcal{K}[G]_e$-codimension at most $b$ (in particular the number of active inequality constraints is at most $3$).
Furthermore, if a full reduction of $(g_u,h_u)$ is $\mathcal{K}[G]$-equivalent to the germ of neither type (4) nor type (8), a full reduction of $(g,h):(N\times U,({\overline{x}},u))\to \R^{q+r}$ is a versal unfolding of $(g_u,h_u)$. 

\end{enumerate}

\end{theorem}

\noindent
Note that one can obtain a model of a versal unfolding of each germ in the tables from Table~\ref{table:generator quotient}. (See the observation at the end of Section~\ref{sec:transversality parameter family}.)

\appendix

\section{Appendix: Standard Basis and Its Applications} \label{sec:appendix}
In this section, we provide a brief summary of standard basis and its application to module membership problems and codimension computation. 
Let 
\begin{math}
M \subset \mathcal{E}_n^q
\end{math}
be an
\begin{math}
\mathcal{E}_n
\end{math}-module. In what follows, we assume 
\begin{math}
M
\end{math}
has finite codimension, i.e.~%
\begin{math}
\dim_{\mathbb{R}} \cfrac{\mathcal{E}_n^q}{M} < \infty
\end{math}. This condition is equivalent to the existence of
\begin{math}
k \in \mathbb{N}
\end{math}
such that 
\begin{math}
\mathcal{M}_n^k \mathcal{E}_n^q \subset M
\end{math}
holds. Let 
\begin{math}
\mathbb{R} \left[ \left[ x_1, \ldots, x_n \right] \right]
\end{math}
be a formal power series ring with variables 
\begin{math}
x_1, \ldots, x_n
\end{math}. Then, 
\begin{math}
\cfrac{\mathcal{E}_n}{\mathcal{M}_n^\infty} \cong \mathbb{R} \left[ \left[ x_1, \ldots, x_n \right] \right]
\end{math}
holds, where we put $\mathcal{M}_n^\infty = \bigcap_{k\geq 0} \mathcal{M}_n^k$. Since 
\begin{math}
M
\end{math}
has finite codimension, 
\begin{math}
\mathcal{M}_n^\infty \mathcal{E}_n^q \subset M
\end{math}
and thus 
\begin{equation}
\cfrac{\mathcal{E}_n^q}{M} \cong \cfrac{\mathcal{E}_n^q/\mathcal{M}_n^\infty \mathcal{E}_n^q}{M/\mathcal{M}_n^\infty \mathcal{E}_n^q} \cong \cfrac{\mathbb{R} \left[ \left[ x_1, \ldots, x_n \right] \right]^q}{\widehat{M}}
\end{equation}
holds where 
\begin{math}
\widehat{M} = M / \mathcal{M}_n^\infty \mathcal{E}_n^q
\end{math}. 
\begin{math}
\widehat{M}
\end{math}
can be regarded as an 
\begin{math}
\mathbb{R} \left[ \left[ x_1, \ldots, x_n \right] \right]
\end{math}-module. Through this isomorphism, 
\begin{math}
\dim_{\mathbb{R}} \cfrac{\mathcal{E}_n^q}{M} = \dim_{\mathbb{R}} \cfrac{\mathbb{R} \left[ \left[ x_1, \ldots, x_n \right] \right]^q}{\widehat{M}}
\end{math}
holds and the computation of the codimension of 
\begin{math}
M
\end{math}
in 
\begin{math}
\mathcal{E}_n^q
\end{math}
can be reduced to that of the codimension of 
\begin{math}
\widehat{M}
\end{math}
in 
\begin{math}
\mathbb{R} \left[ \left[ x_1, \ldots, x_n \right] \right]^q
\end{math}. The latter computation is reduced to the computation of standard basis of 
\begin{math}
\widehat{M}
\end{math}
since 
\begin{math}
\mathbb{R} \left[ \left[ x_1, \ldots, x_n \right] \right]
\end{math}
is Noetherian. For the terminology related to the standard basis, see \cite{Greuel2008}.

Let 
\begin{math}
\prec
\end{math}
be a local term order in the set of the monomials in 
\begin{math}
\mathbb{R} \left[ \left[ x_1, \ldots, x_n \right] \right]
\end{math}
and 
\begin{math}
\prec_m
\end{math}
be the module order compatible with the term order. Take any
\begin{math}
f \in \mathbb{R} \left[ \left[ x_1, \ldots, x_n \right] \right]^q \setminus \left\{ 0 \right\}
\end{math}
and suppose that it can be expanded as 
\begin{equation}
f = c_\alpha x^\alpha e_j + \left( \textnormal{sum of terms smaller than} \; x^\alpha e_j \; \textnormal{with respect to} \; \prec_m\right), 
\end{equation}
where 
\begin{math}
\alpha \in \mathbb{Z}^n_{\ge 0}
\end{math}, 
\begin{math}
c_\alpha \in \mathbb{R} \setminus \left\{ 0 \right\}
\end{math},
\begin{math}
x^\alpha = x_1^{\alpha_1} x_2^{\alpha_2} \cdots x_n^{\alpha_n}
\end{math}. 
For such an 
\begin{math}
f
\end{math}, we define its leading monomial, leading term and leading coefficient as 
\begin{math}
\textnormal{LM} \left( f \right) = x^\alpha e_j
\end{math}, 
\begin{math}
\textnormal{LT} \left( f \right) = c_\alpha x^\alpha e_j
\end{math}, and 
\begin{math}
\textnormal{LC} \left( f \right) = c_\alpha
\end{math}, respectively. For a pair of elements 
\begin{math}
f, g \in \mathbb{R} \left[ \left[ x_1, \ldots, x_n \right] \right]^q \setminus \left\{ 0 \right\}
\end{math}
where 
\begin{math}
\textnormal{LT} \left( f \right) = c_\alpha x^\alpha e_j
\end{math}
and 
\begin{math}
\textnormal{LT} \left( g \right) = d_{\alpha'} x^{\alpha'} e_{j'}
\end{math}, we define their symmetric polynomial as 
\begin{equation}
\textnormal{spoly} \left( f, g \right) = 
\begin{cases}
\left( \frac{f}{c_\alpha x^\alpha} - \frac{g}{d_{\alpha'} x^{\alpha'}} \right) \prod_{j=1}^n x_j^{\max \left\{ \alpha_j, \alpha'_j \right\}} & \left( j = j' \right), \\
0 & \left( j \neq j' \right).
\end{cases}
\end{equation}
We say 
\begin{math}
f
\end{math}
is \textit{divisible} by 
\begin{math}
S = \left\{ f_1, \ldots, f_l \right\}
\end{math}
if there exist
\begin{math}
a_1, \ldots, a_l \in \mathbb{R} \left[ \left[ x_1, \ldots, x_n \right] \right]
\end{math}
satisfying the following conditions:

\begin{itemize}

\item 
\begin{math}
f = \sum_{j=1}^l a_j f_j,
\end{math}

\item 
\begin{math}
\textnormal{LM} \left( f \right) \ge \textnormal{LM} \left( a_j f_j \right)
\end{math}
for all 
\begin{math}
j \in \left\{ 1, \ldots, l \right\}
\end{math}
 with 
\begin{math}
a_j f_j \neq 0.
\end{math}
\end{itemize}

\noindent
We say 
\begin{math}
S = \left\{ f_1, \ldots, f_l \right\}
\end{math}
is a \textit{standard basis} of 
\begin{math}
\widehat{M}
\end{math}
if 
\begin{math}
S
\end{math}
generates
\begin{math}
\widehat{M}
\end{math}
and
\begin{math}
\textnormal{spoly} \left( f_i, f_j \right)
\end{math}
for all 
\begin{math}
i < j
\end{math}
are divisible by 
\begin{math}
S
\end{math}. 
Note that a generating set consisting of monomials is always a standard basis since $\textnormal{spoly}(f,g)=0$ for any monomials $f,g\in \R[[x_1,\ldots, x_n]]^q$.

\begin{theorem} \label{thm:standard_basis}
Let 
\begin{math}
S = \left\{ f_1, \ldots, f_l \right\}
\end{math}
be a standard basis of 
\begin{math}
\widehat{M}
\end{math}. Then, the following hold:
\begin{enumerate}[(1)]
\item 
\begin{math}
\cfrac{\mathbb{R} \left[ \left[ x_1, \ldots, x_n \right] \right]^q}{\widehat{M}}
\end{math}
is isomorphic to the 
\begin{math}
\mathbb{R}
\end{math}-vector subspace in 
\begin{math}
\mathbb{R} \left[ \left[ x_1, \ldots, x_n \right] \right]^q
\end{math} spanned by the monomials that cannot be divisible by any leading monomial of an element of 
\begin{math}
S
\end{math}.
\item For any 
\begin{math}
f \in \mathbb{R} \left[ \left[ x_1, \ldots, x_n \right] \right]^q
\end{math}, 
\begin{math}
f \in \widehat{M}
\end{math}
if and only if
\begin{math}
f
\end{math}
is divisible by 
\begin{math}
S
\end{math}.
\end{enumerate}
\end{theorem}

In what follows, we will explain two examples of applications of Theorem~\ref{thm:standard_basis}.

\subsection{The $\mathcal{K}[G]_e$-codimension and $\mathcal{K}[G]$-determinacy order of the map-germ of type $(1,k)$ in Table~\ref{table:generic constraint r=0}}\label{sec:calculation codim determinacy type 1k}

For $k\geq 2$, let $g$ be the map-germ of type $(1,k)$ in Table~\ref{table:generic constraint r=0}, i.e.~we put
\begin{equation}
g \left( x_1, \ldots, x_n \right) = \left( x_1, \ldots, x_{q-1}, \sum_{j=1}^{q-1} \delta_j x_j + \delta_q x_q^k + \sum_{j=q+1}^{n}\delta_{j} x_{j}^2\right)
\end{equation}
for $\delta_j=\pm 1$.
We first calculate the $\mathcal{K}[G]_e$-codimension of $g$.
The extended tangent space at 
\begin{math}
g
\end{math}
is 
\begin{equation}
T \mathcal{K} \left[ G \right]_e \left( g \right) = \langle \frac{\partial g}{\partial x_1}, \ldots, \frac{\partial g}{\partial x_n} \rangle_{\mathcal{E}_n} + \langle g_1 e_1, \ldots, g_q e_q \rangle_{\mathcal{E}_n}.
\end{equation}
In this case, $\frac{\partial g}{\partial x_j}$ is calculated as follows:
\[
\frac{\partial g}{\partial x_j} = \begin{cases}
e_j + \delta_j e_q & (j=1,\ldots, q-1) \\
k \delta_q x_q^{k-1} e_q & (j=q)\\
2 \delta_j x_j e_q &( j =q+1, \ldots, n ).
\end{cases}
\]
We set the monomial ordering in 
\begin{math}
\mathbb{R} \left[ \left[ x_1, \ldots, x_n \right] \right]
\end{math}
as the negative degree reverse lexicographical ordering
\begin{math}
\prec
\end{math}
satisfying 
\begin{math}
x_n \prec x_{n-1} \prec \cdots \prec x_2 \prec x_1
\end{math}
and the term over position module ordering 
\begin{math}
\prec_m
\end{math}
satisfying 
\begin{math}
e_q \prec_m e_{q-1} \prec_m \cdots \prec_m e_2 \prec_m e_1
\end{math}
compatible with the monomial ordering 
\begin{math}
\prec
\end{math}.

First note that 
\begin{math}
x_j e_q = \left( x_j \left( e_j + \delta_j e_q \right) - x_j e_j \right) / \delta_j \in \widehat{\mathcal{K} \left[ G \right]_e \left( g \right)}
\end{math}
holds for 
\begin{math}
j \in \left\{ 1, \ldots, q-1 \right\}
\end{math}. We claim that 
\begin{equation}
S = \left\{  e_1 + \delta_1 e_q, \ldots, e_{q-1} + \delta_{q-1} e_q, x_1 e_q, \ldots, x_{q-1} e_q, x_q^{k-1} e_q, x_{q+1} e_q, \ldots, x_n e_q \right\}
\end{equation}
is a standard basis of 
\begin{math}
\widehat{T \mathcal{K} \left[ G \right]_e \left( g \right)}
\end{math}. First of all, it is obvious that 
\begin{math}
\cfrac{\partial g}{\partial x_i}
\end{math}
can be written as an 
\begin{math}
\mathbb{R} \left[ \left[ x \right] \right]
\end{math}-linear combination of the elements in 
\begin{math}
S
\end{math}
for all 
\begin{math}
i \in \left\{ 1, \ldots, n \right\}
\end{math}. Second, 
\begin{math}
g_i e_i = x_i \left( e_i + \delta_i e_q \right) - \delta_i \left( x_i e_q \right)
\end{math}
holds for all 
\begin{math}
i \in \left\{ 1, \ldots, q-1 \right\}
\end{math}
and 
\begin{equation}
g_q e_q = \sum_{j=1}^{q-1} \delta_j \times \left( x_j e_q \right) + \delta_q x_q \times \left( x_q^{k-1} e_q \right) +\sum_{j=q+1}^{n} \delta_{j} x_{j} \times \left( x_{j} e_q \right).
\end{equation}
Third, it is obvious that all the elements in 
\begin{math}
S
\end{math}
are contained in 
\begin{math}
\widehat{T\mathcal{K} \left[ G \right]_e \left( g \right)}
\end{math}. Therefore, the set 
\begin{math}
S
\end{math}
generates 
\begin{math}
\widehat{T\mathcal{K} \left[ G \right]_e \left( g \right)}
\end{math}. Next, we show that
\begin{math}
\textnormal{spoly} \left( s_1, s_2 \right)
\end{math}
is divisible by 
\begin{math}
S
\end{math}
for all 
\begin{math}
s_1, s_2 \in S
\end{math}. By the definition of 
\begin{math}
\textnormal{spoly}
\end{math}, 
\begin{math}
\textnormal{spoly} \left( s, s \right) = 0
\end{math}
for all 
\begin{math}
s \in S
\end{math}, 
\begin{math}
\textnormal{spoly} \left( s_1, s_2 \right) = 0
\end{math}
for all the monomials
\begin{math}
s_1, s_2
\end{math}
in 
\begin{math}
S
\end{math}, and 
\begin{math}
\textnormal{spoly} \left( s_1, s_2 \right) = 0
\end{math}
if the components of the leading monomials of 
\begin{math}
s_1
\end{math}
and 
\begin{math}
s_2
\end{math}
are different. Regarding this fact, all the symmetric polynomials between the elements in 
\begin{math}
S
\end{math}
are zero and thus they are divisible by 
\begin{math}
S
\end{math}. This proves that 
\begin{math}
S
\end{math}
is a standard basis of 
\begin{math}
\widehat{T \mathcal{K} \left[ G \right]_e \left( g \right)}
\end{math}. For general algorithm to compute standard basis, see \cite{Greuel2008}.

Therefore, 
\begin{math}
\mathbb{R} \left[ \left[ x_1, \ldots, x_n \right] \right]^q / \widehat{T \mathcal{K} \left[ G \right]_e \left( g \right)}
\end{math}
is spanned by the monomials in 
\begin{math}
\mathbb{R} \left[ \left[ x_1, \ldots, x_n \right] \right]
\end{math}
not divisible by the leading monomials of the elements of $G$ by Theorem~\ref{thm:standard_basis} (1). In this case, that is 
\begin{math}
e_q, x_q e_q, \ldots, x_q^{k-2} e_q
\end{math}. This can be shown as follows. 
First, it is obvious that $e_q, x_q e_q, \ldots, x_q^{k-2} e_q$ are not divisible by any leading monomial of an element in $S$.
The monomial \begin{math}
x^\alpha e_i
\end{math}
is divisible by 
\begin{math}
e_i \; \left( = \textnormal{LM} \left( e_i + \delta_i e_q \right) \right) 
\end{math}
for any \begin{math}
\alpha \in \mathbb{Z}_{\ge 0}^n
\end{math}
and
\begin{math}
i \in \left\{ 1, \ldots, q-1 \right\}
\end{math}. 
If one of the components of 
\begin{math}
\alpha \in \Z_{\geq0}^n
\end{math}
except for 
\begin{math}
q
\end{math}-th one, say $\alpha_i$, is non-zero, the monomial
\begin{math}
x^\alpha e_q
\end{math}
is divisible by 
\begin{math}
x_i e_q
\end{math}.
The monomial \begin{math}
x_q^l e_q
\end{math}
is divisible by 
\begin{math}
x_q^{k-1} e_q \; \left( = \textnormal{LM} \left( x_q^{k-1} e_q \right) \right) 
\end{math} for \begin{math}
l \ge k-1
\end{math}.
By using Theorem~\ref{thm:standard_basis} (1), we obtain the claim.

This specifically implies 
\begin{math}
\widehat{T \mathcal{K} \left[ G \right]_e \left( g \right)}
\end{math}
has a finite codimension in 
\begin{math}
\mathbb{R} \left[ \left[ x_1, \ldots, x_n \right] \right]^q
\end{math}
and thus there exists 
\begin{math}
l \in \mathbb{N}
\end{math}
such that 
\begin{math}
\widehat{\mathcal{M}}_n^l \mathbb{R} \left[ \left[ x_1, \ldots, x_n \right] \right]^q \subset \widehat{T \mathcal{K} \left[ G \right]_e \left( g \right)}
\end{math}
holds. This implies that
\begin{math}
\mathcal{M}_n^l \mathcal{E}_n^q \subset T \mathcal{K} \left[ G \right]_e \left( g \right) + \mathcal{M}_n^\infty \mathcal{E}_n^q
\end{math}
and thus 
\begin{math}
\mathcal{M}_n^l \mathcal{E}_n^q \subset T \mathcal{K} \left[ G \right]_e \left( g \right) + \mathcal{M}_n^{l+1} \mathcal{E}_n^q
\end{math}. By using Nakayama's lemma, 
\begin{math}
\mathcal{M}_n^l \mathcal{E}_n^q \subset T \mathcal{K} \left[ G \right]_e \left( g \right)
\end{math}
holds. Therefore, 
\begin{math}
\cfrac{\mathcal{E}_n^q}{T \mathcal{K} \left[ G \right]_e \left( g \right)} \cong \cfrac{\mathbb{R} \left[ \left[ x_1, \ldots, x_n \right] \right]^q}{\widehat{T \mathcal{K} \left[ G \right]_e \left( g \right)}}
\end{math}
holds and thus
\begin{math}
\cfrac{\mathcal{E}_n^q}{T \mathcal{K} \left[ G \right]_e \left( g \right)} \cong \langle e_q, x_q e_q, \ldots, x_q^{k-2} e_q \rangle_{\mathbb{R}}
\end{math}
and
\begin{math}
\mathcal{K} \left[ G \right]_e
\end{math}-codimension of 
\begin{math}
g
\end{math}
is 
\begin{math}
k-1
\end{math}.

We next confirm that the map-germ 
\begin{math}
g
\end{math}
is $k$-determined relative to 
\begin{math}
\mathcal{K} \left[ G \right]
\end{math}. By using Proposition~\ref{prop:basic properties K[G]-eq/codim}, it is enough to confirm that 
\begin{math}
\mathcal{M}_n^k \mathcal{E}_n^q \subset T \mathcal{K} \left[ G \right] \left( g \right)
\end{math}
holds. This condition is equivalent to the condition that 
\begin{math}
\widehat{\mathcal{M}_n^k} \mathbb{R} \left[ \left[ x \right] \right]^q \subset \widehat{T \mathcal{K} \left[ G \right] \left( g \right)}
\end{math}
holds and thus we confirm the latter condition in what follows.

First of all, the equality 
\begin{multline}
\widehat{T \mathcal{K} \left[ G \right] \left( g \right)} = \widehat{\mathcal{M}_n} \langle \frac{\partial g}{\partial x_1}, \ldots, \frac{\partial g}{\partial x_n} \rangle_{\mathbb{R} \left[ \left[ x \right] \right]} + \langle g_1 e_1, \ldots, g_q e_q \rangle_{\mathbb{R} \left[ \left[ x \right] \right]} \\
= \widehat{\mathcal{M}_n} \langle e_1 + \delta_1 e_q, \ldots, e_{q-1} + \delta_{q-1} e_q, x_q^{k-1} e_q, \delta_{q+1} x_{q+1} e_q, \ldots, \delta_n x_n e_q \rangle_{\mathbb{R} \left[ \left[ x \right] \right]} \\
+ \left< x_1 e_1, x_2 e_2, \ldots, x_{q-1} e_{q-1}, \left( \sum_{j=1}^{q-1} \delta_j x_j + \delta_q x_q^k + \sum_{j=q+1}^n \delta_j x_j^2 \right) e_q \right>_{\mathbb{R} \left[ \left[ x \right] \right]}
\end{multline}
holds. Then, 
\begin{math}
x_j e_q = \left( x_j \left( e_j + \delta_j e_q \right) - x_j e_j \right) / \delta_j \in \widehat{T \mathcal{K} \left[ G \right] \left( g \right)}
\end{math}
holds for all 
\begin{math}
j \in \left\{ 1, \ldots, q-1 \right\}
\end{math}
and thus, one can show that the set
\begin{multline}
S = \left\{ x_i \left( e_j + \delta_j e_q \right) \middle| i \in \left\{ 1, \ldots, n \right\}, j \in \left\{ 1, \ldots, q-1 \right\} \right\} \\
\cup \left\{ x_1 e_q, \ldots, x_{q-1} e_q \right\} \cup \left\{ x_i x_j e_q \middle| i \in \left\{ q, \ldots, n \right\}, j \in \left\{ q+1, \ldots, n \right\} \right\} \cup \left\{ x_q^k e_q \right\}
\end{multline}
is a standard basis of $\widehat{T \mathcal{K} \left[ G \right] \left( g \right)}$ in the same way as that in the demonstration of Theorem~\ref{thm:standard_basis} (1). 

Since the module
\begin{math}
\widehat{\mathcal{M}_n^k} \mathbb{R} \left[ \left[ x \right] \right]^q
\end{math}
is generated by 
\begin{equation}
\left\{ x^\alpha e_j \middle| \alpha \in \mathbb{Z}_{\ge 0}^n, \left| \alpha \right| = k, j \in \left\{ 1, \ldots, q \right\} \right\},
\end{equation}
it is enough to show that all the generators are in 
\begin{math}
\widehat{T \mathcal{K} \left[ G \right] \left( g \right)}
\end{math}. 
By Theorem~\ref{thm:standard_basis} (2), this condition is equivalent to the condition that 
\begin{math}
x^\alpha e_j
\end{math}
is divisible by 
\begin{math}
S
\end{math}
for all 
\begin{math}
\alpha \in \mathbb{Z}_{\ge 0}^n
\end{math}
with $\left| \alpha \right| = k$ and 
\begin{math}
j \in \left\{ 1, \ldots, q \right\}
\end{math}. The latter condition can be shown as follows. 
For  
\begin{math}
\alpha \in \mathbb{Z}_{\ge 0}^n
\end{math}
with $\left| \alpha \right| = k$ and
\begin{math}
j \in \left\{ 1, \ldots, q-1 \right\}
\end{math}, the monomial 
\begin{math}
x^\alpha e_j
\end{math}
is equal to $x^\alpha \left( e_j + \delta_j e_q \right) - \delta_j x^\alpha e_q$. 
Therefore, it is enough to show that 
\begin{math}
x^\alpha e_q
\end{math}
is divisible by 
\begin{math}
S
\end{math}. If one of the elements 
\begin{math}
\alpha_1, \ldots, \alpha_{q-1}
\end{math}
is non-zero, 
\begin{math}
x^\alpha e_q
\end{math}
is divisible by one of the monomials
\begin{math}
x_1 e_q, \ldots, x_{q-1} e_q
\end{math} in $S$. If 
\begin{math}
\alpha_1 = \cdots = \alpha_{q-1} = 0
\end{math}
and one of the elements 
\begin{math}
\alpha_{q+1}, \ldots, \alpha_n
\end{math}
is non-zero, 
\begin{math}
x^\alpha e_q
\end{math}
is divisible by one of the monomials in 
\begin{math}
\left\{ x_i x_j e_q \middle| i \in \left\{ q, \ldots, n \right\}, j \in \left\{ q+1, \ldots, n \right\} \right\}\subset S
\end{math}. 
In the other case $\alpha_1 = \cdots = \alpha_{q-1} = \alpha_{q+1} = \cdots =\alpha_n = 0, \alpha_q = k$, 
\begin{math}
x_q^k e_q
\end{math}
is divisible by itself, which is contained in 
\begin{math}
S
\end{math}. Therefore, all the generators are divisible by 
\begin{math}
S
\end{math}. This proves the claim.
\subsection{The $\mathcal{K} \left[ G \right]^2$-codimension of the $2$-jet shown in Table~\ref{table:kgtangent_Lambda_q-3,q_j2}} \label{sec:appendix_kg2codimension_Lambda_q-3,q}
Let
\begin{math}
g_\alpha
\end{math}
be a map-germ representing \eqref{eq:2jet Lambda_q-3,q}. Then, 
\begin{multline}
\mathcal{K} \left[ G \right] \left( g_\alpha \right) = \mathcal{M}_n \langle e_1 + \delta_1 e_q, \ldots, e_{q-3} + \delta_{q-3} e_q \rangle_{\mathcal{E}_n} \\
+ \mathcal{M}_n \langle e_{q-2} + \left( 2 \alpha_{11} x_{q-2} + \alpha_{12} x_{q-1} \right) e_q, e_{q-1} + \left( 2 \alpha_{22} x_{q-1} + \alpha_{12} x_{q-2} \right) e_q \rangle_{\mathcal{E}_n} \\
+ \mathcal{M}_n \langle x_q e_q, \ldots, x_n e_q \rangle_{\mathcal{E}_n} 
+ \langle x_1 e_1, \ldots, x_{q-1} e_{q-1}, \\
\left( \sum_{j=1}^{q-3} \delta_j x_j + \alpha_{11} x_{q-2}^2 + \alpha_{12} x_{q-2} x_{q-1} + \alpha_{22} x_{q-1}^2 + \sum_{j=q}^n \delta_j x_j^2 \right) e_q \rangle_{\mathcal{E}_n}
\end{multline}
holds. 
In what follows, the elements given after ":" form a basis of 
\[
\widehat{\mathcal{M}_n} \mathbb{R} \left[ \left[ x \right] \right]^q / (\widehat{\mathcal{K} \left[ G \right] \left( g_\alpha \right)} + \widehat{\mathcal{M}_n^3} \mathbb{R} \left[ \left[ x \right] \right])
\]
for each parameter value 
\begin{math}
\left( \alpha_{11}, \alpha_{12}, \alpha_{22} \right)
\end{math}
satisfying the equations before ":". 
We can obtain these results by computing standard basis in the same way as that in the previous subsection (details are left to the readers). 

\begin{enumerate}
\item $\alpha_{12} \alpha_{22} \neq 0$~:
\begin{math}
\overbrace{x_{q-2} e_q, \ldots, x_n e_q}^{n-q+3}, \overbrace{x_{q-2}^2 e_q}^1
\end{math}.
\item $\alpha_{12} = 0, \alpha_{11} \alpha_{22} \neq0$~:
\begin{math}
\overbrace{x_{q-2} e_q, \ldots, x_n e_q}^{n-q+3}, \overbrace{x_{q-2} x_{q-1} e_q}^1
\end{math}.

\item $\alpha_{22} = 0, \alpha_{11} \alpha_{12} \neq 0$~:
\begin{math}
\overbrace{x_{q-2} e_q, \ldots, x_n e_q}^{n-q+3}, \overbrace{x_{q-1}^2 e_q}^1
\end{math}.

\item $\alpha_{11} = \alpha_{22} = 0, \alpha_{12} \neq 0$~: 
\begin{math}
\overbrace{x_{q-2} e_q, \ldots, x_n e_q}^{n-q+3}, \overbrace{x_{q-2}^2 e_q, x_{q-1}^2 e_q}^2
\end{math}.

\item $\alpha_{11} = \alpha_{12} = 0, \alpha_{22} \neq 0$~:
\begin{math}
\overbrace{x_{q-2} e_q, \ldots, x_n e_q}^{n-q+3}, \overbrace{x_{q-2}^2 e_q, x_{q-2} x_{q-1} e_q}^2
\end{math}.

\item $\alpha_{12} = \alpha_{22} = 0, \alpha_{11} \neq 0$~:
\begin{math}
\overbrace{x_{q-2} e_q, \ldots, x_n e_q}^{n-q+3}, \overbrace{x_{q-2} x_{q-1} e_q, x_{q-1}^2 e_q}^2
\end{math}.

\item $\alpha_{11} = \alpha_{12} = \alpha_{22} = 0$~:
\begin{math}
\overbrace{x_{q-2} e_q, \ldots, x_n e_q}^{n-q+3}, \overbrace{x_{q-2}^2 e_q, x_{q-2} x_{q-1} e_q, x_{q-1}^2 e_q}^3
\end{math}.
\end{enumerate}

\noindent
The
\begin{math}
\mathcal{K} \left[ G \right]^2
\end{math}-codimension of 
\begin{math}
j^2g_\alpha \left( 0 \right)
\end{math}
is 
\begin{math}
n-q+4
\end{math}
in the cases 1, 2, and 3, 
\begin{math}
n-q+5
\end{math}
in the cases 4, 5 and 6, and 
\begin{math}
n-q+6
\end{math}
in the case 7. By combining the corresponding semi-algebraic sets in the list on which 
\begin{math}
j^2 g_\alpha \left( 0 \right)
\end{math}
has the same 
\begin{math}
\mathcal{K} \left[ G \right]^2
\end{math}-codimensions, we obtain Table~\ref{table:kgtangent_Lambda_q-3,q_j2}.

\subsection{A list of bases for Table~\ref{table:kgtangent_Lambda_q-3,q+1_j2}} \label{sec:table7_computation}

The following are the list of bases of the quotient space $\mathcal{M}_n\mathcal{E}_n^q/\left(T \mathcal{K} \left[ G \right] \left(g_a\right) + \mathcal{M}_n^3 \mathcal{E}_n^q\right)$.
In what follows, the elements given after ":" form a basis of the quotient space for each $a$ in the algebraic subset of $\R^5$ given before ":" 

\begin{enumerate}
\item 
\begin{math}
V_{\mathbb{R}} \left( \langle a_{q-1,q}, a_{q-1,q-1}, a_{q-2,q}, a_{q-2,q-1}, a_{q-2,q-2} \rangle \right)
\end{math}:
\newline
\begin{math}
x_{q-2} e_q, \ldots, x_n e_q, x_q^2 e_q, x_{q-1} x_q e_q, x_{q-2} x_q e_q, x_{q-1}^2 e_q, x_{q-2}  x_{q-1} e_q, x_{q-2}^2 e_q
\end{math}.

\item 
$V_{\mathbb{R}} \left( \langle a_{q-2,q}, a_{q-2,q-2} \rangle \right) \setminus V_{\mathbb{R}} \left( \langle a_{q-2,q} a_{q-1,q}^3, a_{q-2,q-1} a_{q-1,q}^3, a_{q-2,q-2} a_{q-1,q}^3, \right.$
\newline
$\left. a_{q-2,q} a_{q-1,q-1} a_{q-1,q}^2, a_{q-2,q-1} a_{q-2,q} a_{q-1,q}^2 \rangle \right)$:
\newline
\begin{math}
x_{q-2} e_q, \ldots, x_n e_q, x_q^2 e_q, x_{q-2} x_q e_q, x_{q-2}^2 e_q
\end{math}.

\item 
\begin{math}
V_{\mathbb{R}} \left( \langle a_{q-1,q}, a_{q-1,q-1}, a_{q-2,q-1} \rangle \right) \setminus V_{\mathbb{R}} \left( \langle a_{q-2,q} \rangle \right)
\end{math}: 
\newline
\begin{math}
x_{q-2} e_q, \ldots, x_n e_q, x_{q-1}^2 e_q, x_q^2 e_q, x_{q-1} x_q e_q
\end{math}.

\item 
\begin{math}
V_{\mathbb{R}} \left( \langle a_{q-1,q}, a_{q-2,q} \rangle \right) \setminus V_{\mathbb{R}} \left( \langle a_{q-2,q-1} a_{q-1,q-1} a_{q-1,q}, a_{q-2,q-1} a_{q-2,q} a_{q-1,q-1}, \right.
\end{math}
\newline
\begin{math}
\left.a_{q-2,q-1}^2 a_{q-1,q-1}, a_{q-2,q-2} a_{q-2,q-1} a_{q-1,q-1}^2 \rangle \right)
\end{math}: 
\newline
\begin{math}
x_{q-2} e_q, \ldots, x_n e_q, x_q^2 e_q, x_{q-1} x_q e_q, x_{q-2} x_q e_q, x_{q-2}^2 e_q
\end{math}.

\item
\begin{math}
V_{\mathbb{R}} \left( \langle a_{q-1,q}, a_{q-2,q}, a_{q-1,q-1}, a_{q-2,q-2} \rangle \right) \setminus V_{\mathbb{R}} \left( \langle a_{q-1,q}, a_{q-2,q}, a_{q-2,q-1}, a_{q-2,q-2} a_{q-1,q-1} \rangle \right)
\end{math}: 
\begin{math}
x_{q-2} e_q, \ldots, x_n e_q, x_q^2 e_q, x_{q-1} x_q e_q, x_{q-2} x_q e_q, x_{q-1}^2 e_q, x_{q-2}^2 e_q
\end{math}.
\item
\begin{math}
V_{\mathbb{R}} \left( \langle a_{q-1,q}, a_{q-1,q-1}, a_{q-2,q}, a_{q-2,q-1} \rangle \right) \setminus V_{\mathbb{R}} \left( \langle a_{q-2,q-2} \rangle \right)
\end{math}: 
\newline
\begin{math}
x_{q-2} e_q, \ldots, x_n e_q, x_q^2 e_q, x_{q-1} x_q e_q, x_{q-1}^2 e_q, x_{q-2} x_q e_q, x_{q-2}  x_{q-1} e_q
\end{math}.

\item
\begin{math}
V_{\mathbb{R}} \left( \langle a_{q-1,q}, a_{q-2,q}, a_{q-1,q-1} \rangle \right) 
\end{math}
\newline
\begin{math}
\setminus V_{\mathbb{R}} \left( \langle a_{q-2,q-2} a_{q-1,q}, a_{q-2,q-2} a_{q-2,q}, a_{q-2,q-2} a_{q-2,q-1}, a_{q-2,q-2}^2 a_{q-1,q-1} \rangle \right)
\end{math}: 
\newline
\begin{math}
x_{q-2} e_q, \ldots, x_n e_q, x_q^2 e_q, x_{q-1} x_q e_q, x_{q-2} x_q e_q, x_{q-1}^2 e_q
\end{math}.

\item
\begin{math}
V_{\mathbb{R}} \left( \langle a_{q-2,q} \rangle \right) 
\end{math}
\newline
\begin{math}
\setminus V_{\mathbb{R}} \left( \langle a_{q-2,q-2} a_{q-2,q} a_{q-1,q}^3, a_{q-2,q-2} a_{q-2,q-1} a_{q-1,q}^3, a_{q-2,q-2}^2 a_{q-1,q}^3, \right.
\end{math}
\newline
\begin{math}
\left. a_{q-2,q-2} a_{q-2,q} a_{q-1,q-1} a_{q-1,q}^2, a_{q-2,q-2} a_{q-2,q-1} a_{q-2,q} a_{q-1,q}^2 \rangle \right)
\end{math}: 
\newline
\begin{math}
x_{q-2} e_q, \ldots, x_n e_q, x_q^2 e_q, x_{q-2} x_q e_q
\end{math}.

\item
\begin{math}
V_{\mathbb{R}} \left( \langle a_{q-1,q}, a_{q-2,q}, a_{q-2,q-1} \rangle \right) 
\end{math}
\newline
\begin{math}
\setminus V_{\mathbb{R}} \left( \langle a_{q-1,q-1} a_{q-1,q}, a_{q-2,q} a_{q-1,q-1}, a_{q-2,q-1} a_{q-1,q-1}, a_{q-2,q-2} a_{q-1,q-1}^2 \rangle \right)
\end{math}: 
\newline
\begin{math}
x_{q-2} e_q, \ldots, x_n e_q, x_q^2 e_q, x_{q-1} x_q e_q, x_{q-2} x_q e_q, x_{q-2}  x_{q-1} e_q
\end{math}.

\item
\begin{math}
V_{\mathbb{R}} \left(\langle  a_{q-1,q}, a_{q-2,q}, a_{q-2,q-1}, a_{q-2,q-2} \rangle \right) \setminus V_{\mathbb{R}} \left( \langle a_{q-1,q-1} \rangle \right)
\end{math}: 
\newline
\begin{math}
x_{q-2} e_q, \ldots, x_n e_q, x_q^2 e_q, x_{q-2} x_q e_q, x_{q-2}^2 e_q, x_{q-1} x_q e_q, x_{q-2} x_{q-1} e_q
\end{math}.

\item
\begin{math}
V_{\mathbb{R}} \left(\langle  a_{q-1,q} \rangle \right) 
\end{math}
\newline
\begin{math}
\setminus V_{\mathbb{R}} \left( \langle a_{q-2,q} a_{q-1,q-1} a_{q-1,q}, a_{q-2,q-1} a_{q-1,q-1} a_{q-1,q}, \right.
\end{math}
\newline
\begin{math}
\left. a_{q-2,q-2} a_{q-1,q-1} a_{q-1,q}, a_{q-2,q} a_{q-1,q-1}^2, a_{q-2,q-1} a_{q-2,q} a_{q-1,q-1} \rangle \right)
\end{math}: 
\newline
\begin{math}
x_{q-2} e_q, \ldots, x_n e_q, x_q^2 e_q, x_{q-1} x_q e_q
\end{math}.

\item
\begin{math}
V_{\mathbb{R}} \left(\langle a_{q-1,q}, a_{q-1,q-1} \rangle \right) 
\end{math}
\newline
\begin{math}
\setminus V_{\mathbb{R}} \left( \langle a_{q-2,q} a_{q-1,q}, a_{q-2,q-1} a_{q-1,q}, a_{q-2,q-2} a_{q-1,q}, a_{q-2,q} a_{q-1,q-1}, a_{q-2,q-1} a_{q-2,q} \rangle \right)
\end{math}: 
\newline
\begin{math}
x_{q-2} e_q, \ldots, x_n e_q, x_q^2 e_q, x_{q-1} x_q e_q, x_{q-1}^2 e_q
\end{math}.

\item
\begin{math}
\R^5 \setminus V_{\mathbb{R}} \left( \langle a_{q-2,q-2} a_{q-2,q}^2 a_{q-1,q}^4-a_{q-2,q-1} a_{q-2,q}^3 a_{q-1,q}^3+a_{q-2,q}^4 a_{q-1,q-1} a_{q-1,q}^2, \right. 
\end{math}
\newline
\begin{math}
\left. a_{q-2,q-2} a_{q-2,q-1} a_{q-2,q} a_{q-1,q}^4-a_{q-2,q-1}^2 a_{q-2,q}^2 a_{q-1,q}^3+a_{q-2,q-1} a_{q-2,q}^3 a_{q-1,q-1} a_{q-1,q}^2, \right.
\end{math}
\newline
\begin{math}
\left. a_{q-2,q-2}^2 a_{q-2,q} a_{q-1,q}^4-a_{q-2,q-2} a_{q-2,q-1} a_{q-2,q}^2 a_{q-1,q}^3+a_{q-2,q-2} a_{q-2,q}^3 a_{q-1,q-1} a_{q-1,q}^2, \right.
\end{math}
\newline
\begin{math}
\left. a_{q-2,q-2} a_{q-2,q}^2 a_{q-1,q-1} a_{q-1,q}^3-a_{q-2,q-1} a_{q-2,q}^3 a_{q-1,q-1} a_{q-1,q}^2+a_{q-2,q}^4 a_{q-1,q-1}^2 a_{q-1,q}, \right.
\end{math}
\newline
\begin{math}
\left. a_{q-2,q-2} a_{q-2,q-1} a_{q-2,q}^2 a_{q-1,q}^3-a_{q-2,q-1}^2 a_{q-2,q}^3 a_{q-1,q}^2+a_{q-2,q-1} a_{q-2,q}^4 a_{q-1,q-1} a_{q-1,q} \rangle \right)
\end{math}: 
\newline
\begin{math}
x_{q-2} e_q, \ldots, x_n e_q, x_q^2 e_q
\end{math}.

\item
\begin{math}
V_{\mathbb{R}} \left(\langle a_{q-2,q-2} a_{q-1,q}^2-a_{q-2,q-1} a_{q-2,q} a_{q-1,q}+a_{q-2,q}^2 a_{q-1,q-1} \rangle \right) 
\end{math}
\newline
\begin{math}
\setminus V_{\mathbb{R}} \left( \langle a_{q-2,q}^2 a_{q-1,q}^2, a_{q-2,q-1} a_{q-2,q} a_{q-1,q}^2, a_{q-2,q-2} a_{q-2,q} a_{q-1,q}^2, \right. 
\end{math}
\newline
\begin{math}
\left. a_{q-2,q}^2 a_{q-1,q-1} a_{q-1,q}, a_{q-2,q-1} a_{q-2,q}^2 a_{q-1,q} \rangle \right)
\end{math}: 
\newline
\begin{math}
x_{q-2} e_q, \ldots, x_n e_q, x_q^2 e_q, x_{q-2}^2 e_q
\end{math}.

\item
\begin{math}
V_{\mathbb{R}} \left( \langle a_{q-2,q}, a_{q-2,q-1}, a_{q-2,q-2} \rangle \right) \setminus V_{\mathbb{R}} \left( \langle a_{q-1,q} \rangle \right)
\end{math}: 
\newline
\begin{math}
x_{q-2} e_q, \ldots, x_n e_q, x_q^2 e_q, x_{q-2} x_q e_q, x_{q-2}^2 e_q
\end{math}.
\end{enumerate}
In summary, we obtain Table~\ref{table:the computation detail of Table 7}. By combining the strata of $\R^5$ on which $j^2 g_a \left( 0 \right)$ has the same $\mathcal{K} \left[ G \right]^2$-codimension, we obtain Table~\ref{table:kgtangent_Lambda_q-3,q+1_j2}.
\begin{table}[htbp]
 \begin{center}
  \begin{tabular}{|c|l|c|c|c|} \hline
   $\mathcal{K} \left[ G \right]^2$-cod. & strata numbers \\ \hline
   $n-q+4$ & $13$ \\
   $n-q+5$ & $8, 11, 14$ \\
   $n-q+6$ & $2, 3, 12, 15$ \\
   $n-q+7$ & $4, 7, 9$ \\
   $n-q+8$ & $5, 6, 10$ \\
   $n-q+9$ & $1$ \\
 \hline
  \end{tabular}
  \caption{$\mathcal{K} \left[ G \right]^2$-codimension of $j^2 g_a \left( 0 \right)$ for the parameter $a$ in each stratum.}
  \label{table:the computation detail of Table 7}
 \end{center}
\end{table}

\end{document}